\documentclass[leqno,11pt]{amsart}
\usepackage{amssymb, amsmath}
\usepackage{amssymb,amsfonts}
\usepackage{amsmath,latexsym}
\usepackage{pdfsync}
\usepackage{bbm}

\def\cO{{\mathcal{O}}}
\def\tan{\mbox{tan}}
\def\e{\varepsilon}
\def\t{\langle t\rangle }
\def\s{\langle s\rangle }
\def\curl{\text{curl}\,}
\def\f{\frac}
\def\p{\partial}
\def\Om{\Omega}
\def\na{\nabla}
\def\D{\Delta}
\def\la{\lambda}
\def\bn{{\bf n}}

\def\no{\noindent}
\def\CN{\mathcal{N}}

\numberwithin{equation}{section}
\newtheorem{theorem}{Theorem}[section]
\newtheorem{proposition}[theorem]{Proposition}
\newtheorem{remark}[theorem]{Remark}
\newtheorem{definition}[theorem]{Definition}
\newtheorem{corollary}[theorem]{Corollary}
\newtheorem{lemma}[theorem]{Lemma}

\newcommand{\andf}{\quad\hbox{and}\quad}
\newcommand{\with}{\quad\hbox{with}\quad}

\newcommand{\beq}{\begin{equation}}
\newcommand{\eeq}{\end{equation}}
\newcommand{\ben}{\begin{eqnarray}}
\newcommand{\een}{\end{eqnarray}}
\newcommand{\beno}{\begin{eqnarray*}}
\newcommand{\eeno}{\end{eqnarray*}}

\providecommand{\R}{\mathbb{R}}
\providecommand{\N}{\mathbb{N}}
\providecommand{\dive}{\mathrm{div} \,}

\date{\today}
\title[ ]{Smooth  controllability  of the Navier-Stokes equation with Navier  conditions. Application to  Lagrangian controllability.}
\author[J. Liao]{Jiajiang Liao}%
\address[J. Liao]
 {Academy of Mathematics $\&$ Systems Science, The Chinese Academy of
Sciences, Beijing 100190, CHINA. } \email{liaojiajiang15@mails.ucas.ac.cn}
\author[F. Sueur]{Franck Sueur}
\address [F. Sueur]%
{Universit\'e  Bordeaux \\
 Institut de Math\'ematiques de Bordeaux  \\
F-33405 Talence Cedex, France, and  Institut Universitaire de France.} \email{Franck.Sueur@math.u-bordeaux.fr}
\author[P. Zhang]{Ping Zhang}%
\address[P. Zhang]
 {Academy of Mathematics $\&$ Systems Science
and  Hua Loo-Keng Key Laboratory of Mathematics, The Chinese Academy of
Sciences, Beijing 100190, CHINA, and School of Mathematical Sciences, University of Chinese Academy of Sciences, Beijing 100049, CHINA. } \email{zp@amss.ac.cn}


\begin{document}

\begin{abstract}
We deal with the 3D Navier-Stokes equation in a smooth simply connected
bounded domain, with controls  on a non-empty open
part of the boundary and a Navier slip-with-friction boundary condition  on the remaining, uncontrolled,
part of the boundary.
 We extend the small-time global exact controllability
 result in \cite{CMS} from Leray weak solutions to  the case of smooth solutions.
 Our strategy relies on a refinement of the method of  well-prepared dissipation  of the
  viscous boundary layers which appear near the uncontrolled part of the boundary, which allows to
  handle the multi-scale features in a finer topology.
   As a byproduct of our analysis we also obtain a
  small-time global  approximate Lagrangian controllability result, extending to the case of the Navier-Stokes equations  the recent results   \cite{GH1,GH2,HK}  in the case of the Euler equations and the result \cite{GH3} in the case of the steady Stokes equations.
\end{abstract}

\maketitle
\tableofcontents

\section{Introduction and main results}

\subsection{Setting}
We consider an incompressible viscous fluid  in
a smooth bounded simply connected domain $\Omega$ in $\mathbb{R}^3$.
We denote by $u$ and $p$ its velocity and its pressure respectively and we assume that they evolve according to the Navier-Stokes equations.
We assume that we can act on a non-empty open part $\Sigma$ of the boundary $\partial \Omega$. On the remaining part of the boundary, we assume the fluid satisfies a Navier-slip-with-friction boundary condition. To formalize this boundary condition we introduce the  normal $\mathbf{n}$ pointing outward  the domain, and  for a vector field $f$, we define its tangential part $f_{\tan}$, the strain tensor $D(f)$ and the tangential Navier boundary operator $\CN(f)$ respectively as
\begin{equation}
\label{def-dep}
f_{\tan}:=f-(f\cdot \bn)\bn,\quad
D_{ij}(f):=\frac{1}{2}(\partial_if_j+\partial_jf_i) \quad \mbox{ and }   \quad
\CN(f):=(D(f)\mathbf{n}+Mf)_{\tan},
\end{equation}
where $M$ is a given smooth symmetric matrix-valued function, describing the friction near the boundary.
The Navier condition then reads $\CN(u)=0$; it dates back to \cite{navier1823memoire}.
Finally we prescribe an initial data $u_0$ for the fluid velocity $u$ at time $t=0$.
Then  the system at stake for the unknowns $u$ and $p$ is:
\begin{gather}\label{NS}
\begin{cases}
\partial_t u+u\cdot \nabla u- \Delta u +\nabla p=0  \quad \mbox{ and }   \quad \dive  u=0 \quad \mbox{ in }\Omega,\\
u\cdot \mathbf{n}=0 \quad \mbox{ and }   \quad\CN(u)=0\quad\mbox{ on }\partial\Omega\backslash\Sigma , \\
u=u_0\quad\mbox{ at } t=0 .
\end{cases}
\end{gather}
Let us highlight that, in \eqref{NS}, there is no boundary condition on the part $\Sigma$ of the boundary $\partial \Omega$.
This is typical of the controllability issue, when one chooses not to mention explicitly the controls.
Indeed the controls which will be used in this paper are quite intricate, in particular because of their multi-scale feature.
Let us only point out right now that this freedom of choice on $\Sigma$ allows in particular some fluid to go into and out the domain.
Let us also mention here that we are not going to really use a control all the time in the sense that it will be relevant on some time intervals to choose as boundary condition on
  $\Sigma$ the same Navier condition as on $\partial \Omega \setminus \Sigma$  so that the system then coincides with its uncontrolled counterpart for which $\Sigma = \emptyset$.

\subsection{First main result: smooth small-time global exact controllability}

Our first main result is the following small-time global exact controllability by solutions for which the velocity vector field $u$ is in the class
\begin{equation}\label{classe}
 C([0,T];H^1(\Omega))\cap L^2((0,T); H^2(\Omega)) .
 \end{equation}
\begin{theorem}\label{th}
{\sl Let $T>0$, and $u_0 $ in $H^1(\Omega)$  satisfying $\dive u_{0}=0 $ in $ \Omega$ and $u_{0} \cdot n =0 $ on $ \partial \Omega$.
Then there exists $u$ in the space \eqref{classe} satisfying (\ref{NS}) and $u(T,\cdot)=0$.}
\end{theorem}

Theorem \ref{th} extends the  result in \cite{CMS} where the existence of  $u$ in the weaker class
\begin{equation}\label{w-classe}
  C_w ([0,T];L^2(\Omega))\cap L^2((0,T); H^1(\Omega)) ,
  \end{equation}
 is obtained.  Indeed the result in \cite{CMS} deals with the case where the initial data  $u_0 $  has only a  $L^2(\Omega)$   regularity but the proof developed there fails to guarantee that the constructed solution  propagates higher regularity.  One underlying reason is the multi-scale feature of the constructed  solution which makes small scales  more singular  in a finer topology.
 Indeed the question of whether or not a result such as Theorem \ref{th}  holds true was explicitly raised in \cite[Remark 2]{CMS} and in  \cite[Perspective 1]{CMS-proc}.

 \begin{remark}
 \label{impli}
 Theorem~\ref{th} is stated as an existence result. The lack of
	uniqueness comes from the fact that multiple controls can drive the
	initial state to zero, that is from the fact that there is no boundary condition
on $\Sigma$ for the initial boundary value problem \eqref{NS}. 	
However, with some bookkeeping, it is possible to exhibit (though in a quite non-explicit way) from the proof of Theorem \ref{th} below
a boundary condition to be prescribed on $\Sigma$ (which is inhomogeneous and depends on $u_0$) that generates a unique solution $u$ in the space \eqref{classe} to the corresponding initial boundary value problem, that is satisfying  \eqref{NS} and this boundary condition on  $\Sigma$, and this unique solution $u$ satisfies $u(T,\cdot)=0$.
\end{remark}

\begin{remark}
Controllability results such as the one obtained in  \cite{CMS} or in Theorem \ref{th} should not be confused with results on the existence of wild solutions vanishing after a finite time, such as the ones obtained in \cite{BCV,BV1,BV2}.
The latter rely on the lack of regularity, in particular these solutions do not belong to $L^2((0,T); H^1(\Omega)).$ On the other hand the setting of these papers does not allow any freedom of action, neither through a part of the boundary nor through an interior  part of the domain.
On the contrary, the controllability results of  \cite{CMS} and of Theorem \ref{th}  take advantage of the possibility to choose some appropriate boundary conditions on the permeable  part $\Sigma$ of the boundary to drive the fluid  to rest in finite time. Since the  controllability result of  \cite{CMS} holds for Leray's class of solutions \eqref{w-classe}, it  concerns  solutions which are more regular than in \cite{BCV,BV1,BV2}.  However, perhaps, one may think that the gap is narrow and perhaps only due to temporary technical limitations. The result of Theorem \ref{th} shows that it is not the case and that the possibility of a localized action allows to drive a fluid to rest in finite time  in a smooth setting as well. Indeed Theorem \ref{th} is stated for $H^1$ initial data and for solutions in the regularity class \eqref{classe},
 but it could be easily extended to higher regularity, as the $H^1$ norm is super-critical for the blow-up issue of the 3D Navier-Stokes equations.
\end{remark}

\begin{remark}
Indeed, as in \cite{CMS}
for the case of weak solutions,  the proof of Theorem~\ref{th} can be easily adapted to prove that one may intercept at any given positive time $T$ any  smooth uncontrolled solution to the Navier-Stokes system, that is any solution to the Navier-Stokes system
with Navier condition on the whole boundary $\partial \Omega$, by the mean of a smooth controlled solution starting from any given initial data.
\end{remark}

\begin{remark}
We deal here with the case of a simply connected domain just for simplicity. The multiply-connected domain could be covered by some simple modifications of our method  in the case where   $\Sigma$ intersects all the connected components of $\partial\Omega$.
\end{remark}

\begin{remark}
To simplify the exposition,  Theorem~\ref{th} is stated  in the case of an initial data  which is tangent to the whole boundary.
The result also holds in the case where the initial data is only tangent to the uncontrolled part $\partial \Omega \setminus \Sigma$ of the boundary.
Indeed, to deduce this slightly more general statement from the one considered in  Theorem~\ref{th},
 it is sufficient to  evolve the system on a short time interval with an appropriate control on  $\Sigma$, smooth in time, initially compatible with the initial data and vanishing  after some small positive time.
\end{remark}

\subsection{Second main result: Lagrangian small-time global approximate controllability}

The question that we now address is the possibility of prescribing the motion of a set of particles, following the Lagrangian description of  fluids consisting in following fluid particles along the flow map associated with a velocity field satisfying the system (\ref{NS}).
This type of Lagrangian controllability notion was raised in \cite{GH1}, where the authors showed that for the $2$-D incompressible Euler equations, one can indeed prescribe approximately the motion of some specific sets of fluids, and extended in \cite{GH2} to the case of the dimension $3$. Let us also mention  the paper  \cite{HK} where an alternative approach was considered, the result \cite{GH3} in the case of the steady Stokes equations  and the result in  \cite{Gagnon} about the Lagrangian controllability of the 1-D Korteweg-de Vries equation.

  Our second main result establishes the   small-time global  approximate Lagrangian controllability of (\ref{NS}) meaning that for  two smooth contractible sets of fluid particles, surrounding the same volume, for any given smooth initial velocity field and any positive time interval, one can find a boundary control such that the corresponding solution of (\ref{NS}) makes the first of the two sets approximately reach the second one, while staying in the domain in the meantime.
\begin{theorem}\label{t-lag}
{\sl
Let $T_0>0,$
$\alpha \in (0,1)$ and $k \in \N \setminus \{0\}$.
Let $u_0\in C^{k,\alpha}({\Omega};\R^{3})$ satisfy $\dive u_{0}=0 $ in $ \Omega$ and $u_{0} \cdot n =0 $ on $ \partial \Omega$.
 Let  $\gamma_0$ and $\gamma_1$ be two Jordan surfaces  included in $\Omega$ such that
$\gamma_{0} $ and $ \gamma_{1} $ are isotopic in $ \Omega$ and surrounding the same volume.
Then for any $\eta>0$, there are a time $T \in (0,T_0)$ and a solution $(u,p)$ in $L^{\infty}(0,T;C^{k,\alpha}(\Omega;\R^{4}))$ to (\ref{NS}) on $[0,T]$ such that
\begin{gather}
\label{eq:exact2}
\forall t\in [0,T],\ \phi^u(t,0,\gamma_0)\subset \Omega, \\
\label{eq:approx}
\| \phi^u(T,0,\gamma_0)-\gamma_1 \|_{C^k}<\eta,
\end{gather}
 hold (up to reparameterization), where  $\phi^u$ is the flow map associated with $u$ by
$ \partial_t\phi^u(t,s,x)=u(t,\phi^u(t,s,x)) $ for any $t,s$ in $[0,T]$ and for any $x$ in $\Omega$,
and  $ \phi^u (s,s,x)=x$ for any $s$ in $[0,T]$ and for any $x$ in $\Omega$.

Moreover the same result holds true
in the case where  $u_0$ is only in $H^1({\Omega};\R^{3})$ with $\dive u_{0}=0 $ in $ \Omega$ and $u_{0} \cdot n =0 $ on $ \partial \Omega$, with the two following modifications:
 one only guarantees the existence of a   solution $u$   in  the class
\eqref{classe} and that \eqref{eq:approx}  holds true with $k=0$. }
\end{theorem}
%
Theorem \ref{t-lag} therefore extends to the case of the  Navier-Stokes equations the results mentioned above for the case of the Euler equations and of the steady Stokes system.
It answers, in the case of the Navier conditions,  to  an open problem mentioned
 at the end of the introduction of  \cite{GH3}, in  \cite[Section 3.3.3]{G}   and in  \cite[Perspective 2]{CMS-proc}.

\begin{remark}  \label{assert}
In Theorem \ref{t-lag} we only succeed to assert that there exists a time $T \in (0,T_0)$ for which the conclusion holds, and we are not able to guarantee that $T=T_0$ is convenient.   The difficulty is to prevent a possible blowup due to the vorticity associated with the surface. This difficulty  is typical of the 3D case and was already observed in the case of the Euler  equations, see \cite{GH2}.
\end{remark}
\begin{remark}
The condition that $\gamma_{0} $ and $ \gamma_{1} $  surround the same volume is well defined since by the Jordan-Brouwer separation theorem the set $\R^{3} \setminus \gamma$ has two connected components, only one of which being bounded.
\end{remark}
\begin{remark}
The conditions that
$\gamma_0$ and $\gamma_1$ are isotopic and surround the same volume are necessary for the existence of a smooth volume-preserving flow driving  $\gamma_{0}$ exactly to $\gamma_{1}$.
\end{remark}
\begin{remark}
As in the previous result, see Remark \ref{impli},  the boundary control is implicit in the statement of Theorem \ref{t-lag} as it is given as  traces on $(0,T)\times \Sigma$ of the solution.
\end{remark}

\subsection{Organization of the rest of the paper}

In Section \ref{sec-scheme} we give a scheme of the proof of Theorem \ref{th}.
It will rely on two main intermediate results: Theorem \ref{uj} where an approximate solution is built thanks to a multi-scale asymptotic expansion involving some boundary layers correctors, and the \textit{a priori} estimate
\eqref{APRIORI-R} for the remainder term associated with this approximate solution.
An auxiliary problem associated with the boundary layer is investigated in Section \ref{sec-aux}.
Then the proof of Theorem \ref{uj} is given in
Section \ref{sec-approx}. The proof of the  \textit{a priori} estimate
\eqref{APRIORI-R} is given in Section \ref{sec-Remainder}.
Finally Section \ref{sec-proof-lagp}
is devoted to the proof of   Theorem \ref{t-lag}.

\section{Scheme of proof of Theorem \ref{th}}
\label{sec-scheme}

This section is devoted to a scheme of proof of Theorem \ref{th}.
We only highlight here the key steps of the proof, postponing to the next sections the proofs of several important intermediate results.
As in \cite{Marbach,CMS,CMSZ} we will use the ``well-prepared dissipation" method which consists in a rapid and violent stage where one makes use of the inviscid part of the system and of a second stage devoted to the dissipation of the boundary layers due to the discrepancy between the inviscid and the viscous case.
As in  \cite{CMS,CMSZ} this method is implemented by the means of multi-scale asymptotic expansions.
The extension of this strategy to solutions of the Navier-Stokes equations in the space \eqref{classe}, rather than in the weaker class \eqref{w-classe},   requires much attention, in particular due to the fast scale associated with the boundary layer which leads to a more accurate asymptotic expansion and to a more involved preparation of the dissipation of  various terms describing the fluid behaviour in the boundary layer.

\subsection{Reduction to approximate controllability problem  from a smooth data}
\label{serre-chevalier}

In this section we reduce the proof of Theorem \ref{th} to a combination of a regularisation result on the uncontrolled Navier-Stokes system, that is on the Navier-Stokes system
with Navier condition on the whole boundary $\partial \Omega$, of a small-time
local exact null controllability result and of a global approximate null controllability result.

\smallskip

\begin{itemize}
 \item[(1)] Let us first state the regularization result.
\end{itemize}

\begin{theorem}\label{S}
{\sl Let $T>0$. For any $p\geq 1$,
there exists a continuous function $C_{T,p}$ with $C_{T,p}(0)=0$, such that, if $u_0$ is  in $H^1(\Omega)$, divergence free and tangent to $\partial \Omega$, then there  are $T_1$ in $(0,T)$ and a unique strong solution $u\in C([0,T_1];H^1(\Omega))\cap L^2([0,T_1];H^2(\Omega))$ to
\eqref{NS}
 with  $\CN(u)=0$ on $\partial\Omega$
 and
\begin{eqnarray}\label{S1eqS}
\|u(T_1,\cdot)\|_{H^p(\Omega)}\leq C_{T_1,p}(\|u_0\|_{H^1(\Omega)}).
\end{eqnarray}}
\end{theorem}
In the case where the no-slip conditions is imposed on the boundary $\partial \Omega$,  rather than the Navier conditions  $\CN(u)=0$, such a result dates back to the pioneering work of Leray and Hopf, see \cite{Leray, hopf}.
In the case of the Navier conditions  the part  of Theorem \ref{S} regarding the existence and uniqueness of local-in-time strong solutions with $H^1$ initial data  is also very classical; we refer to the introduction of  \cite{CMS} for an overview of the literature on the subject.
The  part  of Theorem \ref{S}  regarding the regularization, that is the bounds \eqref{S1eqS} for $p>1$, is also  part of the folklore on the Navier-Stokes equations with Navier boundary conditions, see for instance \cite[Lemma 9]{CMS}. As we will need a slight generalization of the result in \cite{CMS} we present a detailed proof of Theorem \ref{S}
 in  the Appendix
 \ref{appa}. In fact, Theorem \ref{thS} in the Appendix \ref{appa} will exhibit the exact singular behavior of the solution near the time zero.

\begin{itemize}
 \item[(2)]
The second ingredient is the  following small-time
local exact null controllability result when the initial data is small in $H^3$ established in~\cite{Guerrero} by Guerrero.
\end{itemize}

\begin{theorem}\label{th-guerrero}
{\sl Let $T>0$. There exists $\eta >0$ such that for any
 $u_0\in H^{3} (\Omega)$  divergence free, tangent to $\partial \Omega$ and satisfying $\|u(T,\cdot)\|_{H^3(\Omega)}<\eta$, there exists $u\in C([0,T];H^3(\Omega))\cap L^2((0,T); H^4(\Omega))$ satisfying (\ref{NS}) and $u(T,\cdot)=0.$}
\end{theorem}

\begin{itemize}
 \item[(3)]
 The third ingredient will be  the following global approximate result.
 \end{itemize}

\begin{theorem}\label{th-app}
{\sl Let $T>0$, and $u_0\in H^{200} (\Omega)$  divergence free and tangent to $\partial \Omega$. For any $\delta >0$, there exists $u\in C([0,T];H^1(\Omega))\cap L^2((0,T); H^2(\Omega))$ satisfying (\ref{NS}) and $\|u(T,\cdot)\|_{H^1(\Omega)}<\delta.$}
\end{theorem}
This last result requires some hard work which will be done below.

On the other hand, with these three ingredients, the proof of Theorem \ref{th} is plain sailing.
\begin{proposition}\label{implie1}
{\it A combination of Theorem \ref{S}, Theorem \ref{th-guerrero} and Theorem \ref{th-app} implies Theorem \ref{th}.}
\end{proposition}
\begin{proof}
The proof will make use of Theorem \ref{th-guerrero}, of Theorem \ref{th-app} and of Theorem \ref{S} twice. We therefore cut the time interval
 $[0,T]$ in four parts and consider $T/4$ as a basic time to which applies each of the three theorems mentioned above. We also need to care about the choice of the small parameters in the right order.
Let  $\eta >0$ be associated with  $T/4$ by Theorem \ref{th-guerrero}.
Let
 $T_1$ in $(0,T/4)$ and
  $\delta >0$  such that $C_{T_1,3}(\delta ) < \eta$, where $C_{\cdot,3}$ is the function mentioned in Theorem \ref{S} in the case where $p=3$.
With these preliminaries at hand we can now proceed to the proof of Proposition \ref{implie1} by chaining
 some appropriate applications of the three theorems:
 we apply first Theorem \ref{S} with $T/4$ instead of $T$ and $p = 200$, then Theorem \ref{th-app} with
$\delta > 0$ as previously chosen, then Theorem \ref{S} again, with  $T/4$ instead of $T$ and $p = 3$, so that
$\|u(T/2 + T_1,\cdot)\|_{H^3(\Omega)}< C_{T_1,3}(\delta ) < \eta$. Then we apply Theorem \ref{th-guerrero}   with  $T/4$ instead of $T$
and therefore obtain the existence of
$u$ in  $C([0,T'];H^1(\Omega))\cap L^2([0,T']; H^2(\Omega)) $
satisfying (\ref{NS}) on $[0,T']$ and $u(T',\cdot)=0$, where $T' = 3T/4 + T_1 < T$.
Then extending $u$ by $0$ for $t$ in $(T',T)$ provides the existence of
$u$ in   $C([0,T];H^1(\Omega))\cap L^2((0,T); H^2(\Omega)) $
satisfying (\ref{NS}) on $[0,T]$  and $u(T,\cdot)=0$.
\end{proof}

\subsection{Domain extension}
\label{sec-ext}

Let $\mathcal{O}$ be a  smooth extension of the initial domain $\Omega$ such that $\Sigma \subset \mathcal{O}$ and $\partial\Omega \backslash\Sigma\subset \partial\mathcal{O}.$ We denote $\mathbf{n}$ to be the outward pointing normal to the extended domain $\mathcal{O}$, which coincides with the outward pointing normal to $\Omega$ on the uncontrolled boundary $\partial\Omega\backslash\Sigma$. We also need to introduce a smooth function $\varphi:\mathbb{R}^3\rightarrow\mathbb{R}$ such that $\varphi=0$ on $\partial \mathcal{O}$, $\varphi>0$ in $\mathcal{O}$ and $\varphi<0$ outside of $\overline{\mathcal{O}}.$ Moreover, we assume that $|\varphi(x)|=\mbox{dist}(x,\partial\mathcal{O})$ in a small neighborhood of $\partial\mathcal{O}.$ Hence we can extend the normal $\mathbf{n}$ smoothly by $-\nabla\varphi$ to the full domain $\cO.$ We define
$\mathcal{V}_{\delta}:=\{x\in \cO:0\leq \varphi(x)<\delta\}$.
Thus there exists a $\delta_0>0$, such that $\varphi=0$ on $\partial \mathcal{O}$ and $|\mathbf{n}|=1$ in $\mathcal{V}_{\delta_0}$.

Theorem \ref{th-app} follows  from the following result.
\begin{theorem}\label{th-extended}
{\sl Let $T>0$ and $u_*\in H^{200}  (\cO)$ divergence free and tangent to $\partial\cO$. Then  for any $\delta >0$, there are
$u$ in $ C([0,T];H^1(\cO))\cap L^2((0,T);H^2(\cO))$, $\xi\in C([0,T];H^1(\cO))$, supported in $\overline{\cO}\backslash\overline{\Omega}$ and $\sigma  $  a smooth scalar function  supported in $(0,T) \times \overline{\cO}\backslash\overline{\Omega}$,
such that
\begin{gather}\label{NSC}
\begin{cases}
\partial_t u+u\cdot\nabla u-\Delta u+\nabla p=\xi \quad \mbox{ and } \quad
\dive u =\sigma  \quad \mbox{ in } \cO,\\
u\cdot \mathbf{n}=0 \quad \mbox{ and } \quad\CN(u)=0
\mbox{ on } \partial\cO,\\
u(0,\cdot)=u_* \quad \mbox{ in } \cO,
\end{cases}
\end{gather}
and $\|u(T,\cdot)\|_{H^1(\cO)} <\delta.$}
\end{theorem}
We will see in the next section how the proof  of Theorem \ref{th-extended} can be reduced to the proof of an asymptotic result, see Theorem \ref{th-extended-eps} below.
For the moment let us see how it allows to conclude to the proof of Theorem \ref{th-app}.
\begin{proposition}\label{implie2}
{\sl Theorem \ref{th-extended} implies Theorem \ref{th-app}.}
\end{proposition}
\begin{proof}
Let $T>0$, and $u_0\in H^{200} (\Omega)$  divergence free and tangent to $\partial \Omega$.  Then there is an extension $u_*$ in $H^{200}  (\cO)$ of $u_0$  into a divergence free vector field  on $\cO$
 tangent to $\partial\cO$. Then applying Theorem \ref{th-extended} we are left with  considering the restrictions of $u$ to $\Omega$ to obtain a vector field in $ C([0,T];H^1(\Omega))\cap L^2((0,T); H^2(\Omega))$ satisfying (\ref{NS}) and $\|u(T,\cdot)\|_{H^1(\Omega)}<\delta.$
\end{proof}

\subsection{Time scaling and small viscosity}

As mentioned above  we will use the ``well-prepared dissipation" method which consists in a rapid and violent stage  followed by a longer one for which no control is applied,  see \cite{Marbach,CMS,CMSZ} for earlier uses of this method.
To implement this two-scales strategy,
 we introduce a positive small scale $\e\ll1$ as in \cite{CMS} and we perform the time scaling
\beq \label{S2scaling} u^{\e}(t,x):=\e u(\e t,x)   \andf p^{\e}(t,x) :=\e^2p(\e t,x). \eeq
 Thus, we consider $(u^{\e},p^{\e})$ the solution to the following large time and slightly viscous  problem:
\begin{subequations} \label{NSA}
\begin{gather}
 \label{NSA1} \partial_t u^{\e}+u^{\e}\cdot\nabla u^{\e}-\e\Delta u^{\e}+\nabla p^{\e}= \xi^{\e} \quad \mbox{ in }(0,T/\e)\times \cO,\\
 \label{NSA2} \dive u^{\e} =\sigma^{\e} \quad \mbox{ in } (0,T/\e)\times\cO,\\
 \label{NSA3} u^{\e}\cdot \mathbf{n}=0 \quad \mbox{ on } (0,T/\e)\times\partial\cO,\\
 \label{NSA4}\CN(u^{\e})=0 \quad \mbox{ on } (0,T/\e)\times\partial\cO,\\
 \label{NSA5} u^{\e}(0,\cdot)=\e u_* \quad \mbox{ in } \cO.
\end{gather}
\end{subequations}
Observing  the amplitude factor $\e$ in the right hand side of   \eqref{NSA5},
we  can  deduce Theorem \ref{th-app}  from the following result:
\begin{theorem}\label{th-extended-eps}
{\sl Let $T>0$ and $u_*\in H^{200}  (\cO)$ divergence free and tangent to $\partial\cO$. Then  there are some sequences,
$\left\{u^{\e}\right\}_\e,$ $\left\{\xi^{\e}\right\}_\e $ with
$$u^{\e} \in  C([0,{T}/{\e}]; H^1(\cO))\cap L^2((0,{T}/{\e}); H^2(\cO)) \andf
 \xi^{\e} \in C([0,{T}/{\e}]; H^1(\cO)),$$
 and $\left\{\sigma^{\e}\right\}_\e  $  a sequence of smooth scalar functions, for $\e$ in $(0,1)$,
 such that the mappings $\xi^{\e}$  and $\sigma^{\e} $   are supported in $\overline{\cO}\backslash\overline{\Omega}$
 as a function of $x$ and compactly supported in $(0,{T}/{\e})$ as a function of $t.$
Furthermore, \eqref{NSA} holds true
and
\begin{equation}\label{DGS}
\|u^{\e}({T}/{\e},\cdot)\|_{H^1(\cO)}=o(\e).
\end{equation}}
\end{theorem}
The proof of Theorem \ref{th-extended-eps} is actually the core of the analysis and its proof will be detailed in the subsequent sections.
Let us start to see here how it entails Theorem  \ref{th-extended}.

\begin{proposition}\label{implie3}
{\sl Theorem \ref{th-extended-eps} implies Theorem  \ref{th-extended}.}
\end{proposition}
\begin{proof}
Let $T>0$ and $u_*\in H^{200}  (\cO)$ divergence free. Then  for any $\delta >0$, according to Theorem \ref{th-extended-eps},
there is $\e >0$ and there exist $u^{\e}$ belongs to $ C([0,{T}/{\e}]; H^1(\cO))\cap L^2((0, {T}/{\e}); H^2(\cO))$, $\xi^{\e}$ belongs to $C([0,{T}/{\e}];H^1(\cO))$ and supported in $\overline{\cO}\backslash\overline{\Omega},$ $\sigma^{\e}  $  is a  smooth scalar function supported in $\overline{\cO}\backslash\overline{\Omega}$
such that \eqref{NSA} holds true  and $\|u^{\e}({T}/{\e},\cdot)\|_{H^1(\cO)} <\delta \e.$
Let us set
\begin{equation}\label{DFD}
( u,\sigma) ( t,x) :=  \frac{1}{ \e}  (u^{\e},\sigma^{\e}) \bigl(\f{t}{ \e},x\bigr) \,  \text{ and  }  \, ( p,  \xi) (t,x)) :=  \frac{ 1 }{ \e^2}  (p^{\e},    \xi^{\e}) \bigl(\f{t}{ \e},x\bigr)  .
 \end{equation}
Then $u$ belongs to $ C([0,T];H^1(\cO))\cap L^2((0,T);H^2(\cO))$, $\xi$ and $\sigma $ are
 compactly supported in $(0,T) \times \overline{\cO}\backslash\overline{\Omega}$ so that
 \eqref{NSC} holds true and $\|u(T,\cdot)\|_{H^1(\cO)} <\delta.$

\end{proof}

\subsection{An auxiliary Euler  solution due to the return method}
 \label{Premium}

When $\varepsilon$ is small,  it is  expected that the analysis of the system   \eqref{NSA}  may be built on
the small-time  global exact controllability of Euler equations.
We therefore consider the  counterpart of the system   \eqref{NSA}
 where the viscosity term has been dropped out. This involves  the incompressible Euler equations.
   For these equations  it is natural to prescribe the condition $ u^{\e} \cdot \mathbf{n}  = 0$ on   an impermeable wall, and only this one.
The natural inviscid counterpart of  \eqref{NSA}  is therefore:
\begin{subequations} \label{NSA-inv}
\begin{gather}
 \label{NSA1-inv} \partial_t u^{\e}+u^{\e}\cdot\nabla u^{\e}+\nabla p^{\e}=\xi^{\e} \quad \mbox{ in }(0,T/\e)\times \cO,\\
 \label{NSA2-inv} \dive u^{\e} =\sigma^{\e} \quad \mbox{ in } (0,T/\e)\times\cO,\\
 \label{NSA3-inv} u^{\e}\cdot \mathbf{n}=0 \quad \mbox{ on } (0,T/\e)\times\partial\cO,\\
 \label{NSA4-inv}  u^{\e}(0,\cdot)=\e u_* \quad \mbox{ in } \cO.
\end{gather}
\end{subequations}

Considering an asymptotic expansion of the form  $u^\e  = \varepsilon u^1 +  o( \varepsilon)  $
would amount to considering the linearized Euler equations around the null state, an equation which is not controllable, unless the initial data $u_*$  is the gradient of a harmonic function, which is not the case in general.
 In order to overcome this difficulty, we are going to use Coron's return method
  to take profit of the nonlinearity by forcing the amplitude of the solution thanks to the control.
  Indeed next result asserts that it is possible to guarantee the existence of a controlled solution to the Euler system with
variations of order $O(1)$ on time interval of order $O(1)$, say $(0,T)$ (but observe that the allotted time in \eqref{DGS} is $T/\varepsilon$), vanishing at both ends of the time interval.
\begin{lemma}\label{lmu0}
{\sl There exists a solution $(u^0,p^0,\nu^0,\sigma^0 )\in C^{\infty}([0,T]\times\overline{\cO}; \mathbb{R}^3\times\mathbb{R}\times\mathbb{R}^3\times\mathbb{R})$ to the system:
\begin{subequations}
\label{euler0}
\begin{gather}
\label{u0}\partial_t u^0+u^0\cdot\nabla u^0+\nabla p^0=\nu^0 \quad \mbox{ in }(0,T)\times\cO,\\
\label{divu0}\dive u^0=\sigma^0  \quad \mbox{ in }(0,T)\times\cO,\\
\label{u0n}u^0\cdot \mathbf{n}=0 \quad \mbox{ on }(0,T)\times\partial\cO,\\
\label{u00}u^0(0,\cdot)=0 \quad \mbox{ in }\cO,\\
u^0(T,\cdot)=0 \quad \mbox{ in }\cO,
\end{gather}
\end{subequations}
such that the flow $\Phi^0$ defined by $\partial_s \Phi^0(t,s,x)=u^0(s,\Phi^0(t,s,x))$ and $\Phi^0(t,t,x)=x$ satisfies
\begin{equation} \label{flush}
\forall\, x\,\in \overline{\cO},\,\exists\, t_x\in (0,T),\quad \Phi^0(0,t_x,x)\in \overline{\cO}\setminus\overline{\Omega}.
\end{equation}
Moreover, $u^0$ can be chosen such that:
  \begin{equation}\label{curlu0}
\nabla\times u^0=0\quad\mbox{ in }[0,T]\times \overline{\cO}.
\end{equation}
In addition, $\nu^0$ and $\sigma^0 $ are supported in $\overline{\cO}\backslash\overline{\Omega}$, $(u^0,p^0,\nu^0,\sigma^0 )$ are compactly supported in $(0,T)$. In the sequel, when we need it, we will implicitly extend them by zero after T.}
\end{lemma}

Lemma \ref{lmu0} is the key argument of many papers concerning the small-time
global exact controllability of Euler equations, cf. ~\cite{MR1233425} for 2D
simply connected domains, ~\cite{MR1380673} for general 2D domains
 when $\Sigma$ intersects all connected components of $\partial\Omega$,
\cite{MR1485616}  for 3D simply connected domains, in~\cite{MR1745685}
 for general domains  when $\Sigma$ intersects all connected components of $\partial\Omega$.
 Let us also refer to \cite{MR1860818} and to
  \cite[Lemma 2]{CMS}.

With this particular auxiliary Euler solution in hands, Coron's return method  consists in
looking for solutions to  \eqref{NSA-inv} admitting  asymptotic expansions of the form:
$ u^\e  = u^0 + \varepsilon u^1 +  o( \varepsilon) $ and
$ p^\e  = p^0 + \varepsilon p^1 +  o( \varepsilon) $,
with some controls $\xi^{\e}$ and $\sigma^{\e}$ also admitting  asymptotic expansions of the same form:
$ \xi^{\e} = \nu^0 + \varepsilon \nu^1 +  o( \varepsilon)   $ and
$ \sigma^\e  = \sigma^0 + \varepsilon \sigma^1 +  o( \varepsilon) $.
Indeed by gathering the terms of order $O(\varepsilon)$, we are led to the following equations for $(u^1 , p^1 )$:
\begin{equation*}
\label{linu}
    \begin{cases}
\partial_t u^1 +  u^0 \cdot \nabla u^1 +   u^1 \cdot \nabla  u^0 + \nabla p^1  = \nu^1 \quad  \mbox{ in } (0,T)\times\cO,\
\\  \dive u^1   =  \sigma^1    \quad  \textrm{in }  (0,T)\times\cO,\
\\  u^1 \cdot \mathbf{n}   = 0  \quad \textrm{on }  (0,T)\times\partial\cO,
\\  u^1 \rvert_{t = 0}  = u_0  \quad  \textrm{in } \cO.
   \end{cases}
\end{equation*}
This is  the linearisation of the Euler equations around $u^0$, and the fact that the vector field $u^0$ satisfies
\eqref{flush} is a crucial gain with respect to the null state.

In the sequel we will use such equations only with zero control on the divergence (corresponding to setting $ \sigma^1 =0$)
 but also with a  source term $f$ supported in the whole domain $\cO$ in the first equation.
We therefore consider  the following type linearized Euler system:
\begin{equation}\label{Au1}
\begin{cases}
\partial_t u+u^0\cdot\nabla u+u\cdot\nabla u^0+\nabla p=\nu +f \quad t\geq 0,x\in\cO,\\
\dive u=0 \quad t\geq 0,x\in\cO,\\
u\cdot \mathbf{n}=0 \quad t\geq 0,x\in\partial\cO,\\
u(0,\cdot)=u_0 \quad t=0,x\in \cO,
\end{cases}
\end{equation}
where $f$ is a given source term whereas $\nu$ is a control force to be chosen  supported in $\overline{\cO}\backslash\overline{\Omega}$.

\begin{lemma}\label{np}
{\sl Let  $k,p\in \mathbb{N_+}.$
Let $u_0 \in H^p(\cO)$ with $\dive u_0=0$ and $u_0\cdot \mathbf{n}=0$ on $\partial\cO$. Let $f \in C^k_{\gamma} (\mathbb{R}_+;H^p(\cO))$ (see Definition \ref{DC}) and $\nabla\times f$ is supported in $[0,T]$ as a function of time $t$.
Then there are  $\nu(t,x)$ in $C^k(\mathbb{R}_+;H^{p-1}(\cO))$, supported in $\overline{\cO}\backslash\overline{\Omega}$ as a function of  $x$ and supported in $[0,T]$ as a function of  time $t$, and $u$ in $C^k(\mathbb{R}_+;H^p(\cO))$, supported in $[0,T]$, such that
 \eqref{Au1} holds true.
Moreover the unique pressure $p,$ for which the integral condition  $\int_{\cO}p \ dx=0$ is satisfied at any time, is in $C^{k-1}_{\gamma}(\mathbb{R}_+;H^p(\cO))$.}
\end{lemma}

\begin{remark}
Though we do not require $f$ to be supported in $[0,T]$, when $t\geq T$, since $f$ is curl-free, $f$ can be represented as a part of the
 pressure term and has decay. In this case, it will be used to solve $u^4$ below.
\end{remark}

\begin{proof}
The existence and uniqueness of a solution in  $C^k(\mathbb{R}_+;H^p(\cO))$
 to the system (\ref{Au1}) makes no debate, the  point is here to choose an appropriate control function $\nu$, supported in $\overline{\cO}\backslash\overline{\Omega}$ as a function  of $x$, such that the solution $u$ of (\ref{Au1}) vanishes when $t\geq T.$ We can prove the Lemma by the argument in Lemma 3 of  \cite{CMS} and Duhamel formula. For sake of completeness let us quickly recall the key observation that  $\omega :=\nabla\times u$  satisfies
\begin{equation}
\begin{cases}\label{omega}
\partial_t \omega +u^0\cdot\nabla\omega-\omega\cdot\nabla u^0+(\dive u^0) \omega=\nabla\times \nu+\nabla\times f,\\
\omega(0,\cdot)=\nabla\times u_0.
\end{cases}
\end{equation}
By Duhamel formula, we wish to find a solution
\begin{equation}\label{omg}
\omega(t,x)=\omega_1(t,x)+\int_0^t\omega_2(s,t,x)ds,
\end{equation}
where $\omega_1$ and $\omega_2$ satisfy
\begin{equation}\label{omg1}
\begin{cases}
\partial_t \omega_1 +u^0\cdot\nabla\omega_1-\omega_1\cdot\nabla u^0+(\dive u^0) \omega_1=\nabla\times \nu_1,\qquad t\geq0,\\
\omega_1(0,\cdot)=\nabla\times u_0,\qquad t=0,
\end{cases}
\end{equation}
and
\begin{equation}\label{omg2}
\begin{cases}
\partial_s \omega_2 +u^0\cdot\nabla\omega_2-\omega_2\cdot\nabla u^0+(\dive u^0) \omega_2=\nabla\times \nu_2,\qquad s\geq t,\\
\omega_2(t,\cdot)=\nabla\times f(t,\cdot), \qquad s=t.
\end{cases}
\end{equation}
By the argument in  Lemma 3 of  \cite{CMS} we can find control functions $\nu_1$, $\nu_2$ and solutions $\omega_1, \omega_2$ of \eqref{omg1} and \eqref{omg2}. We take $\nu=\nu_1+\int_0^t \gamma_2(s,t,x)ds,$ and define $\omega$ by \eqref{omg}. Then $\omega$ is a solution of \eqref{omega}.
Since $u^0\in H^p(\cO)$, $f\in C^k_{\gamma}(\mathbb{R}_+;H^p(\cO))$ and $\nabla\times f$ is supported in $[0,T]$, we can check from the proof of Lemma 3 of  \cite{CMS} that $\nu \in C^{k}(\mathbb{R}_+;H^{p-1}(\cO))$ and is supported in $\overline{\cO}\backslash\overline{\Omega}$ as a function  of $x$ and is supported in $[0,T]$ as a function of time $t$, $\omega\in C^{k}(\mathbb{R}_+;H^{p-1}(\cO))$ and is supported in $[0,T]$. Since $u$ satisfies
$\nabla\times u=\omega$ in $\cO$,
$\dive u=0$, in  $\cO$ and
 $u\cdot \mathbf{n}=0$ on $\partial\cO$,
it is in $ C^k(\mathbb{R}_+;H^p(\cO))$ and supported in $[0,T]$.
By the first equation of (\ref{Au1}) and the Poincar\'e inequality we obtain the part of Lemma \ref{np} concerning the pressure.
\end{proof}

\subsection{Boundary layer and multi-scale asymptotic expansion}
\label{BL}

Since only the impermeability condition is considered in   \eqref{NSA-inv}, a corrector
has to be added to  the Euler equation
to guarantee  the  Navier slip-with-friction boundary condition  \eqref{NSA4}.
  The role of this corrector is  to accurately describe the  behaviour of the fluid close to the boundary in a layer which vanishes as $\e$ goes to $0$.
For the Navier conditions,  in the uncontrolled setting,  it was highlighted
in  \cite{iftimie} that the thickness of this boundary layer is
$\mathcal{O}(\sqrt{\varepsilon})$ and the
the amplitude of the corrector term is also $\mathcal{O}(\sqrt{\varepsilon})$.
 Moreover, a  multiscale  asymptotic expansion of the solutions to the uncontrolled Navier-Stokes equations in the small viscosity limit
  involving a boundary layer term $v,$
which involves an extra variable describing  the fast variations of the fluid velocity in the normal  direction near the boundary, is given. This corrector $v$ is given as a solution to an initial boundary value problem with a boundary condition with respect to this extra variable, that is, in a informal way, an  asymptotic expansion of the form
\begin{equation}
 \label{eq.eq}
 u^\e \sim u^0(t,x)
 + \sqrt{\varepsilon} v\left(t,x, \varphi(x)/ \sqrt{\varepsilon}\right) .
    \end{equation}
Indeed the boundary layer corrector is described by a smooth vector field $v$
 expressed in terms both
of the slow space variable $x \in \cO$ and a fast scalar variable
$z = \varphi(x)/\sqrt{\varepsilon}$, where $v(t,x,z)$
satisfies an equation of the form:
\begin{equation}
 \label{eq.v}
  \partial_t v + (u^0 \cdot \nabla) v  - \partial_{zz} v  =  0,
\end{equation}
for $x$ in $\bar{\cO}$ and $z$ in $ \R_+$, with the following boundary condition at $z=0$:
\begin{equation}
 \label{eq.bv}
   \partial_z v(t, x, 0)  =  2\CN(u^0)(t,x)  .
   \end{equation}
The interest to prescribe  \eqref{eq.bv} is that  the velocity vector field given by  \eqref{eq.eq}
satisfies the Navier condition  \eqref{NSA4},  up to an error term  of order $o(1)$, due to the slow derivatives of $v$.
Indeed it is more convenient to consider an evolution equation for $v$ which is slightly more complicated than  \eqref{eq.v}, and which in particular contains some extra-terms which are of  lower order but allow to propagate the pointwise orthogonality condition
\begin{equation}
  \label{ortoto}
  v(t,x,z) \cdot \mathbf{n}(x) = 0 ,
\end{equation}
 including  the inside domain, not only on the boundary,   from the initial and boundary data to positive times.
  For this type of linear hyperbolic-parabolic (focusing on $t,x$ or $t,z$) equation,
the Cauchy theory is now well-understood, see \cite{f0,f00,g}.

The analysis in \cite{iftimie} was performed for times of order $O(1)$, and
in general  this type of  multiscale  asymptotic expansions
   fails to describe the vanishing viscosity limit of the  Navier-Stokes equation
for large times of order $O(1/\varepsilon)$, even in the case  where the Euler solution stays smooth for all times.
However since the Euler solution $ u^0$ at stake here vanishes after the time $T$,
 the equations \eqref{eq.v} and \eqref{eq.bv}, for $t \geq T$, reduce to
\begin{equation} \label{eq.v.after}
    \partial_t v - \partial_{zz} v   = 0,
\textrm{ for } z \in \R_+ , \textrm{ and  }
    \partial_z v(t, x, 0)  = 0 ,
\end{equation}
where the dependence in the slow variable $x$ only appears through the  ``initial''  data $ \overline{v}(x,z) :=  v(T,x,z)$.
This heat system  dissipates
towards the null state for large times.
However  the decay
at the final time $t = T/\varepsilon$ is only given by
\begin{equation} \label{eq.relax.v.naive}
 \left\| \sqrt{\varepsilon}
 v\left(\frac{T}{\varepsilon},\cdot,\frac{\varphi(\cdot)}
 {\sqrt{\varepsilon}}\right) \right\|_{L^2(\cO)}
 = \mathcal{O}\left(\varepsilon \right) ,
\end{equation}
which is, unfortunately,  not sufficient in view of  the wished estimate \eqref{DGS} and of the tentative expansion  \eqref{eq.eq}.

\subsection{Well-prepared dissipation method}
\label{wpdm}

This difficulty was already presented in \cite{CMS,CMS-proc}, and there to overcome this difficulty, the authors make use of the  well-prepared dissipation method, which was first
 introduced
in~\cite{Marbach} in the case of the 1D  Burgers equation.
The idea is  to enhance the natural dissipation of the boundary layer   after the time $T$ by an appropriate control before, that is in
 guaranteeing that $\overline{v}$ satisfies a finite number of
vanishing moment conditions for $k \in \N$ of the form:
\begin{equation} \label{eq.moments.k}
 \forall x \in \cO, \quad \int_{\R_+} z^k \overline{v}(x,z) \, dz = 0,
\end{equation}
so that  the estimate \eqref{eq.relax.v.naive} holds true but with $o\left(\varepsilon \right)$ in the right hand side.
By linearity the  moments of $\overline{v}$  in  left hand side of  \eqref{eq.moments.k} can be decomposed as the sum of an addend due to the free evolution of $v$ and of an addend due to the control. Indeed due to the properties of the vector field  $u^0$, see \eqref{flush}, it is possible to generate some moments outside, and to convect inside the physical original domain in the time interval $[0,T]$. This allows  to ensure the  condition  \eqref{eq.moments.k} for all $x$ in $\cO$.

\subsection{Backflow}
\label{sec-Backflow}

Thanks to the orthogonality condition  \eqref{ortoto}, the divergence of
the vector field $(t,x)\mapsto v\left(t,x, \varphi(x)/ \sqrt{\varepsilon}\right)$ is not singular in $\e$.
Still it is not zero, there is an error term  of order $O(1)$, due to the slow derivatives of $v$.
To compensate this part,
we set
\begin{equation}
 \label{def-norm}
w(t,x,z) := -\int_{z}^{\infty}\dive v(t,x,z^{\prime})dz^{\prime} ,
    \end{equation}
and consider instead of the expansion  \eqref{eq.eq}
 the refined asymptotic expansion
\begin{equation}
 \label{eq.eq.eq}
 u^\e \sim u^0(t,x)
 + \sqrt{\varepsilon} v\left(t,x, \varphi(x)/ \sqrt{\varepsilon}\right)
 + \e w\left(t,x, \varphi(x)/ \sqrt{\varepsilon}\right) \mathbf n.
    \end{equation}
This expansion has the advantage over \eqref{eq.eq}  to satisfy
\eqref{NSA2-inv} (observe that the right-hand-side has to be zero in $\Omega$ because of the support condition on $\sigma^\e$) up to an error of order
$O(\e)$.
The new term, the last one in  \eqref{eq.eq.eq},   corresponds to a boundary layer on the normal velocity.
The choice to integrate from infinity in  \eqref{def-norm} is precisely to guarantee that $w$ vanishes as $z$ goes to infinity.
Then the new issue is that $w(t,x,0)$ is not zero
so that the right-hand-side of \eqref{eq.eq.eq} cannot satisfy the impermeability condtition
\eqref{NSA3}.
Then a new correction is considered by the mean of  what we call a backflow velocity.
As $w$ will be constructed with the integral condition
\begin{equation*}
\int_{\partial\cO} w(t,x,0)dx=0.
\end{equation*}
there is a solution $\phi$ to the following Neumann problem:
\begin{equation*}
\begin{cases}
\Delta \phi=0\quad \mbox{ in } \cO,\\
\partial_\mathbf{n}\phi= -w(\cdot,\cdot,0) \quad \mbox{ on } \partial\cO.
\end{cases}
\end{equation*}

Thanks to (\ref{curlu0}), we observe that the so-called backflow velocity $\nabla \phi$ satisfies
\begin{equation}\label{BF->E}
\begin{cases}
\partial_t \nabla \phi+u^0\cdot\nabla \nabla \phi+\nabla \phi\cdot\nabla u^0+\nabla \Big(- \partial_t  \phi  - u^0\cdot\nabla \phi \Big)= 0, \quad t\geq 0,x\in\cO,\\
\dive \nabla \phi =0, \quad t\geq 0,x\in\cO,\\
(\nabla \phi) \cdot \mathbf{n}= -w(\cdot,\cdot,0) , \quad t\geq 0,x\in\partial\cO,
\end{cases}
\end{equation}
that is  $\nabla \phi$ satisfies the Euler equations linearized around $u^0$.
Then the
 asymptotic expansion
\begin{equation}
 \label{eq.eq.eq.eq}
 u^\e \sim u^0 (t,x)
 + \sqrt{\varepsilon} v\left(t,x, \varphi(x)/ \sqrt{\varepsilon}\right)
 + \e \Big( w\left(t,x, \varphi(x)/ \sqrt{\varepsilon}\right) \mathbf n + \nabla \phi (t,x) \Big),
    \end{equation}
 is better than the  asymptotic expansion  \eqref{eq.eq.eq} in the sense that
 the impermeability condition
\eqref{NSA3} is now satisfied up to error term $o(\e)$.

\subsection{Approximate solutions}
\label{sec-appr}

Indeed by expanding further the  asymptotic expansion, in particular expanding the velocity into an expansion of the form
 \beq
\label{previou}
\begin{split}
&u^{\e}_a (t,x) :=
u^0  (t,x)
+ \sqrt{\e}    v^1   (t,x,\varphi(x)/ \sqrt{\varepsilon})\\
 +\sum_{j=2}^4\e^{\frac{j}{2}}  \big( u^j (t,x)+&  v^j \left(t,x, \varphi(x)/ \sqrt{\varepsilon}\right)  + \nabla \phi^j (t,x) +   w^j  \left(t,x, \varphi(x)/ \sqrt{\varepsilon}\right)  \mathbf{n}(x) \big),
\end{split}
\eeq
 with some profiles satisfying some PDEs of the previous types but with extra forcing terms due to error terms associated with the profiles which are already determined,
we will be able to  construct some approximate solutions
$u^{\e}_a$,   $p^{\e}_a$ to the system \eqref{NSA}
associated with some control forces $ \xi^{\e} $ and $\sigma^0$ (on the divergence the control given by Lemma \ref{lmu0}  will be sufficient).

These solutions are approximate in the sense that
\begin{subequations}
\label{eqa}
\begin{gather}
\label{eqa2}
\partial_t u^{\e}_a-\e\Delta u^{\e}_a+u^{\e}_a\cdot\nabla u^{\e}_a+\na p^{\e}_a=\xi^\e + \e^2F \quad \mbox{ in }\cO, \\
\label{eqa2div}
\dive u^{\e}_a= \sigma^0  + \e^2H \quad \mbox{ in }\cO,\\
\label{eqa3}
u^{\e}_a\cdot \mathbf{n}=0 \quad \mbox{ on }\partial\cO,\\
\label{eqa4}
\CN(u^{\e}_a)=  \e^2G\quad \mbox{ on }\partial\cO,\\
\label{eqa5}
u^{\e}_a|_{t=0}=\e u_*-\e^2R_0 \quad \mbox{ in }\cO,
\end{gather}
\end{subequations}
where $H$ , $G$, $F $  and $R_0$ are error terms which satisfy some uniform bounds in some appropriate spaces which we now define.
Let us introduce a cut-off function $\chi\in C^{\infty}_0(\R^3)$ such that $\chi=0$ when $|\varphi|\geq \delta_0$ and $\chi=1$ when $|\varphi|<\frac{\delta_0}{2}$, where $\delta_0$ is selected in Section \ref{sec-ext}, and the vector fields set
\begin{align*}
\mathfrak{W} :=  \Bigl\{
&w^0:=\varphi \mathbf{n},w^1:=\bigl(0,-\partial_3\varphi,\partial_2\varphi\bigr)^\top,\
w^2:=\bigl(\partial_3\varphi,0,-\partial_1\varphi\bigr)^\top,\\
&w^3:=\bigl(-\partial_2\varphi,\partial_1\varphi,0\bigr)^\top,\
w^4:=\bigl(\partial_3(x_3(1-\chi)),0,-\partial_1(x_3(1-\chi))\bigr)^\top,\\
&\hspace{3.95cm}
w_0^5:=\bigl(\partial_2(x_1(1-\chi)),-\partial_1(x_1(1-\chi)),0\bigr)^\top \Bigr\}.
\end{align*}
It is easy to observe that $w^j$ are tangential to $\partial\cO$, $0\leq j\leq 5$. Moreover, for $1\leq j\leq 5$, $w^j\cdot \mathbf{n}=0$ in $\mathcal{V}_{\delta_0/2}$ and $\dive w^j=0$ in $\cO$. Now we define the tangential derivatives
\beq\label{defZ}
Z_j :=  w^j\cdot\nabla\ \ \mbox{for}\ 0\leq j\leq 5 \andf Z^{\alpha} :=  Z_0^{\alpha_0}\cdots Z_5^{\alpha_5} \ \mbox{for}\ \alpha=(\alpha_0,\cdots,\alpha_5).
\eeq

Let us observe that
\begin{eqnarray}
\label{Z1}&\nabla Z_j=Z_j\nabla+ \nabla w^j\cdot \nabla,\\
\label{Z2}&\Delta Z_j=Z_j\Delta +2\nabla w^j:\nabla^2+\Delta w^j \cdot \nabla .
\end{eqnarray}
Generally, for $|\alpha|=m\in \mathbb{N}_+,$ we can use Leibniz formula to find that
\begin{eqnarray}
\label{Z3}[\Delta, Z^{\alpha}]=\sum_{|\beta|,|\gamma|\leq m-1}(c_{\beta}\nabla^2Z^{\beta}+c_{\gamma}\nabla Z^{\gamma}),
\end{eqnarray}
for some smooth functions $c_{\beta}$ and $c_{\gamma}$ depended only on the vector field $\mathfrak{W}.$

Let us also observe that, for $1\leq i,j\leq 5,$
\begin{equation}\label{cm}
\text{ the commutators }[\partial_\mathbf{n},Z_{i}],[Z_0,Z_i],[Z_i,Z_j]\text{ are tangential derivatives}.
\end{equation}
Indeed,
$[\partial_\mathbf{n},Z_{i}]=(\mathbf{n}\cdot\nabla) w^{i}\cdot\nabla -(w^{i}\cdot\nabla) \mathbf{n}\cdot\nabla$, and, on one hand
 $(w^{i}\cdot\nabla \mathbf{n})\cdot\nabla$ is a tangential derivative since $w^{i}\cdot\nabla \mathbf{n}\cdot \mathbf{n}=0$ in $\mathcal{V}_{\delta_0}$, while on the other hand,
  due to $w^{i}\cdot \mathbf{n}=0$ in $\mathcal{V}_{\delta_0/2} $ and $\nabla \mathbf{n}$ is symmetric, we have
   $$\bn\cdot\nabla w^{i}\cdot \mathbf{n}=-\mathbf{n}\cdot\nabla \mathbf{n}\cdot w^{i}=-w^{i}\cdot\nabla \mathbf{n}\cdot \mathbf{n}=0$$ in $\mathcal{V}_{\delta_0/2}$, so that $\mathbf{n}\cdot\nabla w^{i}\cdot\nabla$ is also a tangential derivative. Whereas notice that
   for $1\leq i\leq 5,$  $w^i\cdot \nabla\varphi=w^i\cdot \mathbf{n}=0,$ we find that $[Z_0,Z_i]=\varphi[\partial_{\mathbf{n}},Z_i]$ is also
   a tangential derivative. Finally, for $1\leq i,j\leq 5,$ there holds
    $$[Z_i,Z_j]=(w^i\cdot\nabla w^j-w^j\cdot\nabla w^i)\cdot \nabla.$$
   Since $w^i\cdot \mathbf{n}=w^j\cdot\mathbf{n}=0$ and $\nabla \mathbf{n}$ is symmetric, we have
   $$w^i\cdot\nabla w^j\cdot \bn-w^j\cdot\nabla w^i\cdot \bn=-w^i\cdot \nabla\mathbf{n}\cdot w^j+w^j\cdot \nabla\mathbf{n}\cdot w^i=0.$$
 Thus $[Z_i,Z_j]$ is a tangential derivative and \eqref{cm} holds true.

We define
the Sobolev conormal spaces
  \begin{equation*}
H^m_{\rm co}(\cO) :=\left\{ u\in L^2(\cO):Z^{\alpha}u\in L^2(\cO),|\alpha|\leq m \right\}
\end{equation*}
with norm
  \begin{equation}\label{scn}
\|u\|_m :=\Bigl(\sum_{|\alpha|\leq m}\|Z^{\alpha}u\|^2_{L^2}\Bigr)^{\frac{1}{2}}.
\end{equation}
In the same way, we set
  \begin{equation*}
\|u\|_{k,\infty} :=\sum_{|\alpha|\leq k}\|Z^{\alpha}u\|_{L^{\infty}}
\end{equation*}
and we say $u\in W_{co}^{k,\infty}$ if $\|u\|_{k,\infty}$ is finite.

\begin{theorem}\label{uj}
{\sl Let $\gamma>1,$ $k,p,s,q\in \mathbb{N}_+$ with $k\geq 2, p\geq 8,s,q\geq 4$. Assume $u_*$ is smooth enough, say it satisfies \eqref{u*} in Section \ref{cp}. Then there exist
$u^{\e}_a$,   $p^{\e}_a$ and $\xi^{\e}$ satisfy \eqref{eqa2}-\eqref{eqa5} with $F, G , H$ and $R_0$ satisfying, for $0\leq j\leq k ,$
$p_1+p_2\leq p-3,p_2\leq s-2,m\leq p-3,$
\begin{gather}
\label{FH1} \bigl\|\partial_t^jZ^{p_1}(\sqrt{\e}\partial_\mathbf{n})^{p_2}
\begin{pmatrix} F \\ H \end{pmatrix}
\bigr\|_{L^2(\cO)}\lesssim\e^{\frac{1}{4}}\t^{-\gamma},\\
\label{FH2}\bigl\|\partial_t^j Z^{p_1}(\sqrt{\e}\partial_\mathbf{n})^{p_2}
\begin{pmatrix} F \\ H \end{pmatrix}
\bigr\|_{L^{\infty}(\cO)}\lesssim \t^{-\gamma},\\
\label{HG}\|H\|_{H^{m}(\partial\cO)}+\|\partial_t^jG\|_{H^{p-1}(\cO)} \lesssim\t^{-\gamma},\\
\label{R0}\e^{-\frac{1}{4}}\|Z^{p_1}(\sqrt{\e}\partial_{\mathbf{n}})^{p_2}R_0\|_{L^2(\cO)}+\|Z^{p_1}(\sqrt{\e}\partial_{\mathbf{n}})^{p_2}R_0\|_{L^{\infty}(\cO)}\lesssim \e^{-\frac{1}{2}},
\end{gather}

Moreover $u^{\e}_a$ satisfies,
\begin{gather}
\label{Ra}\| u^{\e}_a\|_{W^{1,\infty}(\cO)}+\|\nabla u^{\e}_a\|_{m,\infty}+ \sqrt{\e}\|\nabla^2 u^{\e}_a\|_{m-1,\infty}\lesssim \t^{-\gamma},\\
\label{ua0}\|u^{\e}_a-u^0\|_{m,\infty}+ \sqrt{\e}\|\nabla(u^0_a-u^0)\|_{m,\infty}   \lesssim\sqrt{\e}\t^{-\gamma},\\
\label{uja}\|u^{\e}_a({T}/{\e},\cdot)\|_{H^1(\cO)}=o(\e).
\end{gather}}
\end{theorem}

The proof of Theorem \ref{uj} will be presented  in Section \ref{sec-approx}.

\subsection{Remainder estimate}
\label{Sec-RE}

It follows from the well-posedness of he Navier-Stokes equations with Navier boundary conditions (for fixed $\e$)
that
 for every $\e$ in $(0,1)$, there is $T^\e \in (0, { T }/{ \e} ]$ and a solution $(u^{\e},p^{\e})$ to
\eqref{NSA} with  $\xi^{\e}$ given by Theorem \ref{uj} and  $\sigma^{\e} := \sigma^0$, for each $\e$, where $\sigma^0$ is given by Lemma \ref{lmu0}.

We define a family of vector fields $R$, neglecting an index for the dependence on $\e$ for sake of levity, by
\begin{equation}
\label{DEF-R}
u^{\e}=u^{\e}_a+\e^2 R .
\end{equation}
The latter $R$ stands for ``remainder " as we hope to be able to find such a vector field with a nice behaviour  in $\varepsilon$.
Indeed we will prove in Section \ref{sec-Remainder} the following   \textit{a priori } estimate:
\begin{equation}
\label{APRIORI-R}
\e^2 \sup_{t \in (0,T^\e)} \|R(t,\cdot)\|_{H^1(\cO)}\lesssim \e^{\frac{5}{4}}.
\end{equation}
This entails that  $T^\e =  \frac{T}{\e}$ and, with \eqref{uja},
that \eqref{DGS} holds true. This concludes
  the scheme of proof  of Theorem \ref{th-extended-eps}, and then according to Proposition \ref{implie3}, Proposition \ref{implie2} and Proposition \ref{implie1}, this also  concludes
the scheme of proof of Theorem \ref{th}.
To complete the proof of Theorem  \ref{th-extended-eps}
it remains to prove the
two main intermediate results which are Theorem \ref{uj}  and the \textit{a priori} estimate
\eqref{APRIORI-R}. In Section \ref{sec-aux}, we will study
 an auxiliary problem associated with the boundary layer on the tangential velocity.
It will be instrumental in  the proof of Theorem \ref{uj} which will be given in
Section \ref{sec-approx}.


\section{Well-prepared dissipation of tangential boundary layers with forcing}
\label{sec-aux}

We set
\begin{eqnarray}
u^0_{\flat}(t,x) :=\frac{u^0(t,x)\cdot \mathbf{n}(x)}{\varphi(x)}\qquad\mbox{ in }\mathbb{R}_+\times\cO,
\end{eqnarray}
where  $u^0$  is given by Lemma \ref{lmu0}
and we observe  that $u^0_{\flat}$ is smooth in $\overline{\cO}$.
Let $B^0 =B^0 (t,x)$  be a smooth field of $3 \times 3$ matrices such that for any $v$ in $\mathbb{R}^3$,
 \begin{equation} \label{defB}
 B^0 v:=v\cdot\nabla u^0+(u^0\cdot \nabla \mathbf{n}\cdot v) \mathbf{n}-(v\cdot\nabla u^0\cdot \mathbf{n})\mathbf{n} .
\end{equation}
The key property associated with $B^0$ is that for a smooth vector field $v(t,x)$,
 \begin{equation} \label{defB-k}
(u^0\cdot\nabla v+B^0 v )\cdot \mathbf{n} = u^0\cdot\nabla  ( v \cdot \mathbf{n}) \quad \mbox{ in }\mathcal{V}_{\delta_0}.
\end{equation}
We are interested in this section by the following type of constrained initial-boundary value problem:
\begin{equation}\label{u}
\begin{cases}
\partial_t v+u^0\cdot\nabla v+B^0 v-u^0_{\flat}z\partial_z v-\partial_z^2 v=\xi+f, \quad t\geq 0,x\in \cO,z\in \mathbb{R}_+,\\
\partial_z v|_{z=0}=g,\quad t\geq 0,x\in \cO,\\
v\cdot \mathbf{n}=0,\quad t\geq 0,x\in \cO,z\in \mathbb{R}_+,\\
v|_{t=0}=v_0,\quad x\in \cO,z\in \mathbb{R}_+,
\end{cases}
\end{equation}
where $f$ and $g$ are given source terms whereas $\xi$ is a control force to be chosen.
Problem like \eqref{u} will be useful to construct such boundary layer correctors of the tangential velocity  as that described in Section \ref{BL}.
As already mentioned, the Cauchy theory for  this type of linear hyperbolic-parabolic (respectively in $t,x$ and in $t,z$) equation
 is now well-understood, see \cite{f0,f00,g},
 and our concern will rather be the large time asymptotics and in particular the implementation of the
 well-prepared dissipation method alluded in
Section \ref{wpdm} in the presence of source terms. This will be useful in the next section  in the course of constructing the
higher order terms
$v^j$ for $j \geq 2$ alluded in \eqref{previou}.

Let us introduce the following weighted Sobolev spaces.
\begin{definition}
For $z\in \mathbb{R}$, we denote $\langle z\rangle:=\sqrt{1+z^2}$
and for $s$ and $q\in \mathbb{N},$ we set
\begin{equation*}
H^{s}_{q}(\mathbb{R_+}):=\Bigl\{\ f\in H^s(\mathbb{R_+}): \ \sum_{j=0}^s\int_{\mathbb{R_+}}\langle z\rangle^{2q}|\partial_z^jf(z)|^2dz<+\infty\ \Bigr\},
\end{equation*}
endowed with it natural associated norm.
 In the same way we define $H^{s}_{q}(\mathbb{R})$ and the norm
\begin{equation*}
 \|f\|_{H^{s}_{q}(\mathbb{R})} := \Bigl(\sum_{j=0}^s \int_{\mathbb{R}}\langle z\rangle^{2q}|\partial_z^jf(z)|^2dz \Bigr)^\frac12 .
\end{equation*}
 \end{definition}
Observe that by  the Plancherel theorem, we have the following  equivalence of norms:
\begin{equation}
  \label{equinorm}
\|f\|_{H^{s}_{q}(\mathbb{R})} \sim  \sum_{j=0}^{q} \Bigl(\int_{\mathbb{R}} \langle \zeta\rangle^{2s} | \partial_{\zeta}^j \hat{f}(\zeta)|^2 d\zeta\Bigr)^\frac12 ,
\end{equation}
where $\hat{f}$ denotes the Fourier transform  of $f$.
\begin{definition}\label{DC}
Let $k\in\mathbb{N},\gamma>0$ and $X$  a Banach space with norm $\|\cdot\|_X$.
We define the space $C^k_{\gamma}(\mathbb{R}_+;X)$ of the functions  $f\in C^k(\mathbb{R}_+;X)$ such that
$$\|f\|_{C^k_{\gamma}(\mathbb{R}_+;X)} :=\sup_{t\geq 0,0\leq j\leq k}\bigl(\|\partial_t^jf(t)\|_X\t^{\gamma}\bigr)<+\infty,$$ where
$$C^k(\mathbb{R}_+;X) :=\bigl\{\ f:\partial_t^jf\in C(\mathbb{R}_+;X), \ 0\leq j\leq k\ \bigr\}.$$
\end{definition}

Let $\mathcal{S}(\mathbb{R})$  the Schwartz space of smooth functions on $\mathbb{R}$ whose derivatives are rapidly decreasing.
Let us denote by $\mathcal{S}(\mathbb{R}_+)$ the set of the  restrictions to $\mathbb{R}_+$
of the  functions of $\mathcal{S}(\mathbb{R})$.

The goal of this section is to prove the following result, where the notation $[x]$ designates the floor integer part  of a real number $x$.
\begin{proposition}\label{lmu}
{\sl Let $\gamma>0$ and $s,q,k,p \in\mathbb{N}$  with $k\geq 1$. Set  $n :=[\frac{q}{2}+\gamma]$,
\begin{gather}
  \label{G1}
\tilde{\gamma}:= 2n+3, \quad \tilde{s}:=   s+2k+2n, \quad    \tilde{q}:=   2n+3,
 \\   \label{G2}  k' :=[\frac{s+1}{2}]+k+n ,  \quad     \tilde{k}:= k+k'-1,
 \quad     \tilde{p}:= p+k'+1    .
\end{gather}
Let
$$f\in C^{\tilde{k}}_{\tilde{\gamma}}(\mathbb{R}_+; H^{\tilde{p} } (\cO ; H^{\tilde{s}}_{\tilde{q}}  (\mathbb{R}_+) ))
   \text{ and }
g\in C^{\tilde{k}}_{\tilde{\gamma}}(\mathbb{R}_+;H^{\tilde{p}}(\cO)) ,$$ such that
$f(t,x,z)$ and $ g(t,x) $ are supported in $ \mathcal{V}_{\delta} $ as a function of  $ x $ and such that
$f(t,x,z)\cdot \mathbf{n}(x)=g(t,x)\cdot \mathbf{n}(x)=0$, for any $ t\geq 0,x\in \cO$ and $z\in \mathbb{R}_+$.
Let
\begin{equation}
v_0(x,z)=A(0,x,z)\in H^{p+2}(\cO;C_0^{\infty}(\overline{\mathbb{R}_+})),
\end{equation}
where $A(t,x,z)$ will be defined in (\ref{DefA}) soon.

Then there are
$$\xi\in C^{k-1}(\mathbb{R}_+;H^p (\cO;\mathcal{S}(\mathbb{R}_+)))
 \text{ and }
v \in  C^{k}_{\gamma}(\mathbb{R}_+;H^p (\cO; H^s_q (\mathbb{R}_+))),$$
 such that (\ref{u})  holds true.
Moreover  there is a continuous function $\tilde{S}:\mathbb{R}_+\rightarrow\mathbb{R}_+$, such that for any positive $\delta,$  $\delta\leq \tilde{S}(\delta)$, and
 $\xi$
  is supported in $(\overline{\cO}\backslash\overline{\Omega})\cap\mathcal{V}_{\tilde{S}(\delta)}$ as a function of  $x$ and is compactly supported in $(0,T)$ as a function of  time $t$, and satisfies
 $\xi(t,x,z)\cdot \mathbf{n}(x)=0$, for all $ t \in (0,T)$, $x\in (\overline{\cO}\backslash\overline{\Omega})\cap\mathcal{V}_{\tilde{S}(\delta)} $ and $z\in \mathbb{R}_+$,
  and  $v$ is supported in $\mathcal{V}_{\tilde{S}(\delta)}$ as a function of  $x$ .
  Moreover, if $f$ and g are both supported away from $t=0$ as a function of time $t$, then so does $v$.
 }
\end{proposition}

The first key observation towards the proof of Proposition \ref{lmu}
 is that  for $t\geq T$, we have $u^0=0, u^0_{\flat}=0,$ $B^0 =0$ and we look for
a control $\xi$ which is  compactly supported in $(0,T)$, so the equations for $v$ reduces to
\begin{equation*}
\begin{cases}
\partial_t v  - \partial_z^2 v= f, \quad t\geq T,x\in \cO,z\in \mathbb{R}_+,\\
\partial_z v|_{z=0}=g,\quad t\geq T,x\in \cO,\\
v\cdot \mathbf{n}=0,\quad t\geq T,x\in \cO,z\in \mathbb{R}_+,
\end{cases}
\end{equation*}
with an ``initial" data at $t=T$ which has no reason to be zero.
To prepare the part of the proof of Proposition \ref{lmu} regarding the decay in time,
we first single out some well-prepared dissipation conditions for the heat equation on the full line (in space) with non-zero ``initial" data at $t=T$
and non-zero source term:
\begin{equation*}
\begin{cases}
\partial_t v-\partial_z^2 v=f, \quad t\geq T,x\in \cO,z\in \mathbb{R},\\
v|_{t=T}=v(T,\cdot,\cdot),\quad x\in \cO,z\in \mathbb{R}.
\end{cases}
\end{equation*}
For $n\in \mathbb{N}$ and  $x\in \mathbb{R},$  we set
\begin{equation}
\label{defsn}
s_n(x):=\sum_{k=0}^n \frac{x^k}{k!}.
\end{equation}
\begin{lemma}\label{lmfc}
{\sl Let $\gamma>0$ and $k,s,q,n\in \mathbb{N}$  and
\begin{equation}
\label{flam}
 n\geq \frac{q}{2}+\gamma-1.
 \end{equation}
Let $\tilde{\gamma}$, $\tilde{s}$, and $\tilde{q}$ be as in   \eqref{G1}.
Let $v_0 \in H^{s+2k}_{\tilde{q}  }     (\mathbb{R})$ and $f\in C^0_{\tilde{\gamma}}(\mathbb{R}_+;H^{\tilde{s}}_{\tilde{q}}(\mathbb{R}))$ when $k=0$ and $   f \in C^{k-1}_{\tilde{\gamma}}(\mathbb{R}_+;H^{\tilde{s}}_{\tilde{q}  }     (\mathbb{R}))$ when $k\geq 1$, such that
\begin{eqnarray}\label{Au0}
\left.\Big(\partial^j_{\zeta}\big(\hat{v}_0(\zeta)+\int_0^{\infty}s_n(\tau\zeta^2)\hat{f}(\tau,\zeta)d\tau\big)\Big)\right|_{\zeta=0}=0,\quad    \text{  for } 0\leq j\leq 2n+1,
\end{eqnarray}
Then the following Cauchy problem
\begin{equation*}
\begin{cases}
\partial_t v-\partial_z^2 v=f, \quad t\geq 0, z\in \mathbb{R},\\
v|_{t=0}=v_0,\quad z\in \mathbb{R}.
\end{cases}
\end{equation*}
has a unique solution $v\in
 C^{k}_{\gamma}(\mathbb{R}_+;H^{s}_{q}(\mathbb{R})).$}
\end{lemma}

\begin{proof}
We first observe that it is sufficient to deal with the case where $k=0$, since the general case follows by using that  for $0\leq i\leq k$,   for $z$ in $\mathbb{R}$ and $t\geq 0$,
$\partial_t^{i}v=\partial_z^2\partial_t^{i-1}v+\partial_t^{i-1}f$.

The Fourier transform $\hat{v}(t,\cdot)$ of $v(t,\cdot)$ is given,
for $ t\geq 0$ and $\zeta$ in $\R$, by
\begin{eqnarray} \label{dej}
\hat{v}(t,\zeta)=e^{-t\zeta^2}\Big( \hat{v}_0(\zeta)+\int_0^t e^{\tau\zeta^2}\hat{f}(\tau,\zeta)d\tau  \Big).
\end{eqnarray}

Let us  observe that
\begin{equation}
\label{le1}
\forall j\in \mathbb{N}, \,   \exists C_j > 0   \text{ such that } \,   \forall t>0,\,  \forall \zeta\in \mathbb{R}, \,  |\partial_{\zeta}^{j}(e^{-t\zeta^2})|\leq C_j \langle t\rangle^{\frac{j}{2}}e^{-\frac{3}{4}t\zeta^2} .
\end{equation}

Now we decompose the proof of   Lemma \ref{lmfc}
 into the following  two steps:
\smallskip

\noindent \textbf{Step 1:} we first prove that, for $0\leq t\leq 1$, $\|u(t,\cdot)\|_{H^{s}_{q}(\mathbb{R})}$ is bounded.
Indeed, for $0\leq t \leq 1$ and $s<\tilde{s},$ $q<\tilde{q},  $  it follows from  \eqref{equinorm}, \eqref{dej},  the Leibniz formula and \eqref{le1} that
  \begin{eqnarray*}
&&\|v(t,\cdot)\|_{H^{s}_{q}(\mathbb{R})}\lesssim\sum_{j=0}^{q}\|\langle \zeta\rangle^s\partial_{\zeta}^j \hat{v}(t,\zeta)\|_{L^2_{\zeta}}\\
&&\lesssim \sum_{j=0}^{q}\sum_{j_1+j_2=j}\Big(\|\langle\zeta\rangle^s\partial_{\zeta}^{j_1}(e^{-t\zeta^2})\partial_{\zeta}^{j_2}\hat{v}_0(\zeta)\|_{L^2_{\zeta}}
+\|\langle\zeta\rangle^s\int_0^{t}\partial_{\zeta}^{j_1}(e^{-(t-\tau)\zeta^2})\partial_{\zeta}^{j_2}\hat{f}(\tau,\zeta)d\tau\|_{L^2_{\zeta}}\Big)\\
&&\lesssim  \|v_0\|_{H^{s}_{\tilde{q}}     (\mathbb{R})} +\|f\|_{C^0_{\tilde{\gamma}}(\mathbb{R}_+;H^{\tilde{s}}_{\tilde{q}}(\mathbb{R}))}.
\end{eqnarray*}

Thus for $0\leq t\leq 1$, $\|v(t,\cdot)\|_{H^{s}_{q}(\mathbb{R})}$ is bounded.
\smallskip

\noindent \textbf{Step 2:} It remains to prove that there exists $C>0$ such that  for $t\geq 1$,  $\| v(t,\cdot)\|_{H^{s}_{q}}  \leq C\t^{-\gamma}$.
Indeed, for $t\geq 1$,  by  \eqref{dej}, we write
\begin{equation}
\label{daistyle}
\hat{v}(t,\zeta)=\sum_{i=1}^4I_i(t,\zeta),
\end{equation}
 where
\begin{eqnarray*}
&I_1(t,\zeta):=e^{-t\zeta^2}\Big( \hat{v}_0(\zeta)+\int_0^{+\infty} s_n(\tau\zeta^2)\hat{f}(\tau,\zeta)d\tau  \Big),\quad &I_2(t,\zeta):=-e^{-t\zeta^2}\int_{\frac{t}{4}}^{+\infty} s_n(\tau\zeta^2)\hat{f}(\tau,\zeta)d\tau ,\\
&I_3(t,\zeta):=e^{-t\zeta^2}\int_0^{\frac{t}{4}}\bigl(e^{\tau\zeta^2}-s_n(\tau\zeta^2)\bigr)\hat{f}(\tau,\zeta)d\tau,\quad &I_4(t,\zeta):=\int_{\frac{t}{4}}^{t} e^{-(t-\tau)\zeta^2}\hat{f}(\tau,\zeta)d\tau.
\end{eqnarray*}
Thanks to  \eqref{le1},
 to conclude this second step, it is sufficient to show that, for $1 \leqslant i \leqslant 4$,  $0\leq j\leq q$ and $t\geq 1,$
 $$\| \langle  \zeta\rangle^s\partial_{\zeta}^j  I_i(t,\zeta)\|_{L^{2}_{\zeta}}\lesssim \langle t\rangle^{-\gamma}.$$
We observe that for $t\geq 1$,
\begin{equation}\label{t>1}
t\leq \t\leq \sqrt{2}t.
\end{equation}

\no $\bullet$ \underline{Estimate of $I_1$}\vspace{0.2cm}

Since $u_0$ and $f$ satisfies (\ref{Au0}), we have, for $0\leq j_2\leq q$,  by the Taylor formula,
\begin{eqnarray*}
&&|\partial_{\zeta}^{j_2}(\hat{v_0}(\zeta)+\int_0^{+\infty}s_n(\tau\zeta^2)\hat{f}(\tau,\zeta)d\tau)|\\
&&\lesssim |\zeta|^{2n+2-j_2}\|\partial_{\zeta}^{2n+2}\big(\hat{v_0}(\zeta)+\int_0^{+\infty}s_n(\tau\zeta^2)\hat{f}(\tau,\zeta)d\tau\big)\|_{L^{\infty}_{\zeta}}\\
&&\lesssim|\zeta|^{2n+2-j_2}\|\partial_{\zeta}^{2n+2}\big(\hat{v_0}(\zeta)+\int_0^{+\infty}s_n(\tau\zeta^2)\hat{f}(\tau,\zeta)d\tau\big)\|_{H^1_{\zeta}}\\
&&\lesssim|\zeta|^{2n+2-j_2}\bigl(\|v_0\|_{H^{0}_{\tilde{q}}   (\mathbb{R})}    +\|f\|_{C^0_{\tilde{\gamma}}(\mathbb{R}_+;H^{\tilde{s}}_{\tilde{q}}     (\mathbb{R}))}\bigr)\\
&&\leq C|\zeta|^{2n+2-j_2} .
\end{eqnarray*}
This together with the Leibniz formula, \eqref{le1} and (\ref{t>1}) implies that
 for $0\leq j\leq q$ and $t\geq 1,$
\begin{equation*}
\|\langle\zeta\rangle^s\partial_{\zeta}^jI_1(t,\zeta)\|_{L^2_{\zeta}}
\lesssim\sum_{j_1+j_2=j} \|\langle \zeta\rangle^s e^{-\frac{3}{4}t\zeta^2}\langle t\rangle^{\frac{j_1}{2}}|\zeta|^{2n+2-j_2}\|_{L^2_{\zeta}}\lesssim t^{-(n+\frac{5}{4}-\frac{j}{2})}.
\end{equation*}
Thus, thanks to  \eqref{flam}, we achieve
\begin{equation}
\label{daistyle1} \|\langle\zeta\rangle^s\partial_{\zeta}^jI_1(t,\zeta)\|_{L^2_{\zeta}}
\leq C\t^{-\gamma} .
\end{equation}

\no $\bullet$ \underline{Estimate of $I_2$}\vspace{0.2cm}

 By the Leibniz formula and \eqref{le1}, for $0\leq j\leq q,$ we find
\begin{eqnarray*}
\|\langle \zeta\rangle^s\partial_{\zeta}^jI_2(t,\zeta) \|_{L_{\zeta}^2}
&\lesssim& \sum_{j_1+j_2+j_3=j}\|\langle \zeta\rangle^s\int_{\frac{t}{4}}^{\infty} \partial_{\zeta}^{j_1}(e^{-t\zeta^2})\partial_{\zeta}^{j_2}(s_n(\tau\zeta^2))\partial_{\zeta}^{j_3}\hat{f}(\tau,\zeta)d\tau\|_{L^2_{\zeta}}\\
&\lesssim& \sum_{j_1+j_2+j_3=j}\|\int_{\frac{t}{4}}^{\infty} \t^{\frac{j_1}{2}}e^{-\frac{3}{4}t\zeta^2}\langle\tau\rangle^n\langle \zeta\rangle^{s+2n}|\partial_{\zeta}^{j_2}\hat{f}(\tau,\zeta)|d\tau\|_{L^2_{\zeta}}\\
&\lesssim& \int_{\frac{t}{4}}^{\infty} \langle t\rangle^{\frac{q}{2}}\langle \tau\rangle^{n}\|f(\tau,\cdot)\|_{H^{s+2n}_q}d\tau .
\end{eqnarray*}
Since $\|f(\tau,\cdot)\|_{H^{s+2n}_{q}}\lesssim\langle \tau\rangle^{-(2n+3)}$, by using \eqref{flam}, we deduce that
\begin{equation}
\label{daistyle2}
\|\langle \zeta\rangle^s\partial_{\zeta}^jI_2(t,\zeta) \|_{L_{\zeta}^2}
\leq C\t^{-\gamma}.
\end{equation}

\no $\bullet$ \underline{Estimate of $I_3$}\vspace{0.2cm}

 By Taylor's expansion and by induction on $j$, we prove that
for all $j\in \mathbb{N}$, there exists $C_{j,n} > 0 $ such that  for all $\tau>0$, for all $\zeta\in \mathbb{R}$,
 $$|\partial_{\zeta}^j(e^{\tau\zeta^2}-s_n(\tau\zeta^2))|\leq C_{j,n}\tau^{n+1}|\zeta|^{2n+2-j} e^{(2-\frac{1}{j+1})\tau\zeta^2} .$$
Then,  for $0\leq j\leq q,$ by the Leibniz formula, one has
\begin{eqnarray*}
&&\|\langle \zeta\rangle^s\partial_{\zeta}^jI_3(t,\zeta) \|_{L_{\zeta}^2}\\
&&\lesssim \sum_{j_1+j_2+j_3=j}\|\langle \zeta\rangle^s\int_0^{\frac{t}{4}}\partial_{\zeta}^{j_1}(e^{-t\zeta^2})\partial_{\zeta}^{j_2}(e^{\tau\zeta^2}-s_n(\tau\zeta^2))
\partial_{\zeta}^{j_3}\hat{f}(\tau,\zeta)d\tau\|_{L_{\zeta}^2}\\
&&\lesssim \sum_{j_1+j_2+j_3=j}\|\langle \zeta\rangle^s\int_0^{\frac{t}{4}}\langle t\rangle^{\frac{j_1}{2}}e^{-\frac{3}{4}t\zeta^2}\tau^{n+1}|\zeta|^{2n+2-j_2}e^{2\tau\zeta^2}|\partial_{\zeta}^{j_3}\hat{f}(\tau,\zeta)|d\tau\|_{L^2_{\zeta}}\\
&&\lesssim \sum_{j_1+j_2+j_3=j}t^{\frac{j_1+j_2}{2}-n-1}\|\int_0^{\frac{t}{4}}\langle \zeta\rangle^{s}\tau^{n+1}|\partial_{\zeta}^{j_3}\hat{f}(\tau,\zeta)|d\tau\|_{L^2_{\zeta}}\\
&&\lesssim t^{\frac{q}{2}-n-1}\int_0^{\frac{t}{4}}\tau^{n+1}\|f(\tau,\cdot)\|_{H^{s}_{q}}d\tau.
\end{eqnarray*}
Since $\|f(\tau,\cdot)\|_{H^s_q}\lesssim \langle\tau\rangle^{-\tilde{\gamma}}$ and (\ref{t>1}), we obtain
\begin{equation}
\label{daistyle3}
\|\langle \zeta\rangle^s\partial_{\zeta}^jI_3(t,\zeta) \|_{L_{\zeta}^2}\leq C\langle t\rangle^{\frac{q}{2}-n-1}\leq C\t^{-\gamma}.
\end{equation}

\no $\bullet$ \underline{Estimate of $I_4$}\vspace{0.2cm}

 By  \eqref{le1}, we find, for $0\leq j\leq q$,
\begin{eqnarray*}
\|\langle \zeta\rangle^s\partial_{\zeta}^jI_4(t,\zeta) \|_{L_{\zeta}^2}&\lesssim& \sum_{j_1+j_2=j}\|\int_{\frac{t}{4}}^t \partial_{\zeta}^{j_1}(e^{-(t-\tau)\zeta^2})\langle \zeta\rangle^s\partial_{\zeta}^{j_2}\hat{f}(\tau,\zeta)d\tau\|_{L^2_{\zeta}}\\
&\lesssim& \sum_{j_1+j_2=j}\|\int_{\frac{t}{4}}^t \langle t-\tau\rangle^{\frac{j_1}{2}}e^{-\frac{3(t-\tau)}{4}\zeta^2}\langle \zeta\rangle^s|\partial_{\zeta}^{j_2}\hat{f}(\tau,\zeta)|d\tau\|_{L^2_{\zeta}}\\
&\lesssim& \int_{\frac{t}{4}}^t \langle \tau\rangle^{\frac{q}{2}}\|f(\tau,\cdot)\|_{H^{s}_{q}}d\tau .
\end{eqnarray*}
Since $\|f(t,\cdot)\|_{H^{s}_{q}}\lesssim\langle \tau\rangle^{-(2n+3)}$, we infer
\begin{equation}
\label{daistyle4}
\|\langle \zeta\rangle^s\partial_{\zeta}^jI_4(\tau,\zeta) \|_{L_{\zeta}^2}\lesssim \langle t\rangle^{\frac{q}{2}-2n-2}\leq \t^{-\gamma}.
\end{equation}

By combining  the estimates, \eqref{daistyle},  \eqref{daistyle1}, \eqref{daistyle2},  \eqref{daistyle3} and \eqref{daistyle4},  we deduce that there exists $C>0$ such that  for $t\geq 1$,  $\| v(t,\cdot)\|_{H^{s}_{q}}  \leq C\t^{-\gamma}$.

Finally by combining step 1 with step 2, we conclude  that $v$ belongs to $ C^0_{\gamma}(\mathbb{R}_+;H^{s}_{q}(\mathbb{R})).$
\end{proof}

We  now turn  to the following counterpart for the whole line $z \in \mathbb{R}$ of the  initial-boundary value problem \eqref{u}:
\begin{equation}\label{Vz}
\begin{cases}
\partial_t V+u^0\cdot\nabla V+B^0 V-u^0_{\flat}z\partial_z V-\partial_z^2 V= \Xi+F, \quad t\geq 0,x\in \cO,z\in \mathbb{R},\\
V|_{t=0}=0,\quad x\in \cO,z\in \mathbb{R} .
\end{cases}
\end{equation}
We recall that $B^0 $ is defined in  \eqref{defB}.

\begin{lemma}\label{lmU}
{\sl Let $\gamma>0,$ $k,p,s,q,n\in\mathbb{N},k\geq 1$ satisfying $n\geq \frac{q}{2}+\gamma-1.$ Let $\tilde{\gamma},\tilde{s}, \tilde{q}$ be as in (\ref{G1})  and $\delta>0$ be a small constant.
Let
\begin{equation}
\label{Fdata}
F\in C^{k-1}_{\tilde{\gamma}}(\mathbb{R}_+;H^{p+1} (\cO;   H^{\tilde{s}}_{\tilde{q}}     (\mathbb{R})),
\end{equation}
with $F(t,x,z) $ being supported in $ \mathcal{V}_{\delta}$ as a function of  $ x$ and $F(t,x,z)\cdot \mathbf{n}(x)=0$, for all $ t\geq 0$, $x\in \cO$ and $z\in \mathbb{R}$.

Then there are
$$\Xi(t,x,z)\in C^{k-1}(\mathbb{R}_+;H^p(\cO;\mathcal{S}(\mathbb{R})))
 \text{ and }   V \in C^{k}_{\gamma}(\mathbb{R}_+;H^p  (\cO;   H^s_q  (\mathbb{R}))) ,$$
 such that (\ref{Vz})  holds true, and there is a continuous function $\tilde{S}:\mathbb{R}_+\rightarrow\mathbb{R}_+$, such that for any positive $\delta,$  $\delta\leq \tilde{S}(\delta)$, and
 $\Xi$
  is supported in $(\overline{\cO}\backslash\overline{\Omega})\cap\mathcal{V}_{\tilde{S}(\delta)}$ as a function of  $x$ and is compactly supported in $(0,T)$ as a function of  time $t$, and satisfies
 $\Xi(t,x,z)\cdot \mathbf{n}(x)=0$, for all $ t \in (0,T)$, $x\in (\overline{\cO}\backslash\overline{\Omega})\cap\mathcal{V}_{\tilde{S}(\delta)} $ and $z\in \mathbb{R}$,
  and  $V$ is supported in $\mathcal{V}_{\tilde{S}(\delta)}$ as a function of  $x$ and satisfies  $V(t,x,z)\cdot \mathbf{n}(x)=0$, for all $ t\geq 0$, $x\in \cO$ and $z\in \mathbb{R}$.}

  Moreover, if $F$ is supported away from $t=0$ as a function of time $t$, then so does $V$.

\end{lemma}

\begin{proof}
For $0\leq j\leq 2n+1$ and $x$ in $\cO$, let
$$\gamma_j(x):=\partial_{\zeta}^j\int_0^{\infty}s_n(\tau\zeta^2)\hat{F}(T+\tau,x,\zeta)d\tau|_{\zeta=0},$$
where $\hat{F}(t,x,\cdot)$ is the  partial Fourier transform
 of $F(t,x,z)$ with respect to the $z$ variable.  We use $\zeta$ as dual variable of $z$ by the partial  Fourier transform.
 We also recall that $s_n$ is defined in  \eqref{defsn}.
By \eqref{Fdata}, for $0\leq j\leq 2n+1$,
 $\gamma_j\in H^{p+1}(\cO).$
 We look for  a control profile $\Xi$, with the properties mentioned in the statement of Lemma \ref{lmU},  such that there is a solution  $V$ in  $C^{k} (\mathbb{R}_+;H^p  (\cO;   H^s_q  (\mathbb{R})))$  to (\ref{Vz})  satisfying
\begin{equation}\label{UT}
(\partial_{\zeta}^j\hat{V}(T,x,\zeta)+\gamma_j(x))|_{\zeta=0}=0,\quad \text{ for } 0\leq j\leq 2n+1  \text{ and } x \in \overline{\cO},
\end{equation}
where $\hat{V}(t,x,\cdot)$ is the partial  Fourier transform of $V(t,x,\cdot)$.
Then, for  $t\geq T$,  as $u^0=0$, $u^0_{\flat}=0$ and $B^0 =0$,
the first equation in \eqref{Vz} reduces to
\begin{equation} \label{ }
\partial_t V-\partial_z^2 V=F, \quad x\in \cO,z\in \mathbb{R} .
\end{equation}
Therefore it would follow from Lemma \ref{lmfc} that $V \in C^{k}_{\gamma}(\mathbb{R}_+;H^p  (\cO;   H^s_q  (\mathbb{R})))$.

Indeed for a given control profile $\Xi$, with the properties mentioned in the statement of Lemma \ref{lmU},
     the existence of a solution  $V$ in $C^{k} (\mathbb{R}_+;H^p  (\cO;   H^s_q  (\mathbb{R})))$ to (\ref{Vz}),  supported in a neighborhood of the boundary as a function of $x$ and satisfying  $V(t,x,z)\cdot \mathbf{n}(x)=0$, for all $ t\geq 0$, $x\in \cO$ and $z\in \mathbb{R}$,   can be proved along the same lines as  \cite[Proposition $5$] {iftimie}. We therefore focus on the existence of a control profile $\Xi$ for which the corresponding solution $V$ to (\ref{Vz}) satisfies the
      conditions (\ref{UT}).
In this perspective we first observe  that the Cauchy problem \eqref{Vz} for $V$
translates into the following one for $\hat{V}$:
\begin{equation}  \label{Uhat}
\begin{cases}
\partial_t\hat{V}+u^0\cdot\nabla\hat{V}+(B^0 +\zeta^2-u^0_{\flat})\hat{V}-u^0_{\flat}\zeta\partial_{\zeta}  \hat{V}=\hat{\Xi}+\hat{F},\\
\hat{V}|_{t=0}=0 .
\end{cases}
\end{equation}
Let
\begin{equation*}
H(x,\zeta) :=\sum_{j=0}^{2n+1}\gamma_j(x)\frac{\zeta^j}{j!}\chi_1(\zeta),
\end{equation*}
where $\chi_1$ in $C^{\infty}_0(\mathbb{R})$ is a cut-off function satisfying $\chi_1(\zeta)=1$ when $|\zeta|\leq 1$ and $\chi_1(\zeta)=0$ when $|\zeta|\geq 2$,
 so that $H\in H^{p+1}(\cO;C^{\infty}_0(\mathbb{R}))$ and
 \begin{equation} \label{ConD}
\partial_{\zeta}^j H(x,\zeta)|_{\zeta=0}=\gamma_j(x)    \text{ for } 0\leq j\leq 2n+1  \text{ and } x \in \overline{\cO}.
\end{equation}

Let
\begin{equation}
\label{defff}
\tilde{F}:=\hat{F}+u^0\cdot\nabla H+(B^0 +\zeta^2-u^0_{\flat})H-u^0_{\flat}\zeta\partial_{\zeta} H .
\end{equation}

By \eqref{Fdata}, for $0\leq j\leq 2n+1$,
$\partial_{\zeta}^j\tilde{F}|_{\zeta=0}\in C^{k-1}_{\tilde{\gamma}}(\mathbb{R}_+; H^{p}(\cO)) $.

Using \eqref{flush},  we can prove the existence of  $\Xi$ with the properties mentioned in the statement of Lemma \ref{lmU},
such that for $0\leq j\leq 2n+1$, the unique solution $Q_j$ to
\begin{equation} \label{CauchY}
\begin{cases}
\partial_t Q_j+u^0\cdot\nabla Q_j+B^0 Q_j-(j+1)u^0_{\flat}Q_j=-j(j-1)Q_{j-2}+\partial_{\zeta}^j \hat{\Xi}|_{\zeta=0}+\partial_{\zeta}^j\tilde{F}|_{\zeta=0},\\
Q_j|_{t=0}=\gamma_j(x) ,
\end{cases}
\end{equation}
where $\hat{\Xi}(t,x,\cdot)$ is the Fourier transform
 of $\Xi(t,x,\cdot)$,
satisfies
 \begin{equation} \label{ConDQQ}
Q_j(T,x)=0   ,   \text{ for }    0\leq j\leq 2n+1  \text{ and }   x\in \overline{\cO}.
\end{equation}
We refer here to \cite[Lemma 7]{CMS}, see also the discussion in Section \ref{wpdm}.
By differentiating \eqref{Uhat}, by  \eqref{ConD}
and by using the uniqueness of the Cauchy problem \eqref{CauchY}, we observe that
  the solution  $V$ to (\ref{Vz}), for the  control profile $\Xi$ mentioned above,  satisfies
 \begin{equation} \label{ConDQ}
  Q_j(t,x)=\partial_{\zeta}^j \hat{V}(t,x,\zeta)|_{\zeta=0} +  \gamma_j(x),   \text{ for } 0\leq j\leq 2n+1 , \ t \in \R_+   \text{ and } x \in \overline{\cO}.
\end{equation}
  By  combining  \eqref{ConDQQ} and \eqref{ConDQ}, we conclude that
 (\ref{UT}) is  satisfied. From the construction of $\Xi$ and $Q_j$ we can see that, if $F$ vanishes near $t=0$, so does $V.$

 Finally, thanks to the argument in  \cite[Section 3.4]{CMS}, there is  a continuous function $\tilde{S}:\mathbb{R}_+\rightarrow\mathbb{R}_+$, such that for any positive $\delta,$  $\delta\leq \tilde{S}(\delta)$,  and  $V$ is supported in $\mathcal{V}_{\tilde{S}(\delta)}$.
 We can choose $\delta$ small enough such that $\tilde{S}(\delta)<\delta_0$. (Recall that $\delta_0$ is defined  in Section \ref{sec-ext}).
 \end{proof}

Now we are in a position to complete the proof of Proposition \ref{lmu}.

\begin{proof}[Proof of Proposition \ref{lmu}]
Let
\begin{equation}
  \label{defgj}
g_1:=g, \quad g_{j+1}:=\partial_tg_j+u^0\cdot\nabla g_j+B^0 g_j-(2j-1)u^0_{\flat}g_j-(\partial_z^{2j-1}f)|_{z=0^+}   \text{ for } 1\leq j< k'.
\end{equation}
 It is clear that $g_j$ is supported in $\mathcal{V}_{\delta}$ as a function of  $x$,  $g_j\cdot \mathbf{n}=0$ and $g_j\in C^{\tilde{k}+1-j}_{\tilde{\gamma}}(\mathbb{R}_+;H^{\tilde{p}+1-j}(\cO))$ for $1\leq j\leq k'$.

For $z\geq 0$, we denote
\begin{equation}\label{DefA}
A(t,x,z) :=\sum_{j=1}^{k'}g_j(t,x)\frac{z^{2j-1}}{(2j-1)!}\chi_1(z),
\end{equation}
where $\chi_1\in C_0^{\infty}(\mathbb{R})$ is an even cut-off function as in the proof of Lemma \ref{lmU}.
 One can check that
 $$A\in C^{k}_{\tilde{\gamma}}(\mathbb{R}_+;H^{p+2}(\cO;C_0^{\infty}(\overline{\mathbb{R}_+}))),$$
 and satisfies
\begin{equation}
  \label{legend}
 \partial_z^{2j-1}A|_{z=0^+}=g_j  \quad \text{ for } 1\leq j\leq k'.
 \end{equation}

Let
\begin{equation}
 \label{bureau}
F:=f-(\partial_t A+u^0\cdot\nabla A+B^0 A-u^0_{\flat}z\partial_zA-\partial_z^2 A).
  \end{equation}
 It is easy to check that
 $$F\in C^{k-1}_{\tilde{\gamma}}(\mathbb{R}_+;H^{p+1}     (\cO ;  H^{\tilde{s}}_{\tilde{q}}     (\mathbb{R}_+))).$$
By combining \eqref{defgj},   \eqref{legend} and  \eqref{bureau}, we observe  that $\partial_z^{2j-1} F|_{z=0}=0$ for $1\leq j< k'.$
Thus, extending $F$ by $F(t,x,z):=F(t,x,|z|)$, and by the definition of $k'$, we have
 $$F\in C^{k-1}_{\tilde{\gamma}}(\mathbb{R}_+;H^{p+1}    (\cO ;    H^{\tilde{s}}_{\tilde{q}}     (\mathbb{R}))),$$
  which is supported in $\mathcal{V}_{\delta}$ as a function of $x$. Thus we can use Lemma \ref{lmU} to find $\Xi$ and  $V$, such that, in particular,  (\ref{Vz})  holds true.
  Let
\begin{equation} \label{deuf}
v(t,x,z) :=V(t,x,z) + A(t,x,z), \quad t\geq 0,x\in \cO,z\in \mathbb{R}_+.
\end{equation}

  Then
  $v$ satisfies all the properties listed in Proposition \ref{lmu}. In particular it follows from (\ref{Vz}), \eqref{defgj},   \eqref{legend} and  \eqref{deuf}
  that  (\ref{u})  holds true, with $v_0=A(0,x,z)\in H^{p+2}(\cO;C_0^{\infty}(\overline{\mathbb{R}_+}))$. In particular, if f and g are both supported away from $t=0$ as a function of time $t$, the so do $A,V$ and $v$, and $v_0=0$.
\end{proof}

\section{Proof of Theorem \ref{uj}}
\label{sec-approx}

let us first  introduce a Lemma which handles multiplication in space $C^{k}_{\gamma}(\mathbb{R}_+;H^p(\cO;H^s_q(\mathbb{R}_+)))$.

\begin{lemma}\label{UV} Let $\gamma>0, k,p,s,q \in \mathbb{N}_+$ with $p\geq 4$ and $s\geq 2.$
Let $U\in C^{k}_{\gamma}(\mathbb{R}_+;H^p(\cO))$ and $V,\tilde{V}\in C^{k}_{\gamma}(\mathbb{R}_+;H^p(\mathcal{O};H^s_q(\mathbb{R}_+)))$ be scalar functions,
 then, one has
\begin{gather}\label{eq-UV}
UV,\tilde{V}V\in C^{k}_{\gamma}(\mathbb{R}_+;H^p(\mathcal{O};H^s_q(\mathbb{R}_+))).
\end{gather}
\end{lemma}
\begin{proof}
By Definition \ref{DC} and Sobolev imbedding, for $0\leq j\leq k,0\leq |\alpha|\leq p-2$ and $0\leq \beta\leq s-1$,
\begin{equation*}
\partial_t^j\partial_x^{\alpha}U\in L^{\infty}(\mathbb{R}_+\times\cO)\quad \text{ and }\quad\partial_t^j\partial_x^{\alpha}\partial_z^{\beta}V\in L^{\infty}(\mathbb{R}_+\times\cO\times\mathbb{R}_+).
\end{equation*}
Note that when $p\geq 4$ and $s\geq 2,$
\begin{equation*}
[\frac{p}{2}]\leq p-2\ \ \text{ and }\ \  [\frac{s}{2}]\leq s-1.
\end{equation*}
Then we can easily check \eqref{eq-UV} by definition.
\end{proof}

\subsection{Construction of profiles}\label{cp}

Recall that $u^0$ is given by Lemma \ref{lmu0} which is smooth, curl-free and compactly supported in $(0,T)$ as a function of time $t$. Now we construct an approximate solution of form \eqref{previou}. Plug \eqref{previou} into \eqref{eqa}, and we can find the equation for $u^i$ and $v^i$. For the equation of $v^i$, profiles $v^j,u^j$ with $j<i$ will play roles as source terms. We use Proposition \ref{lmu} to find profile $v^j$. But there will be some regularity loss. Thanks to Lemma \ref{lmfc}, we need more regularity of the source term to gain decay of the solution.

Let $\gamma>1, k,p,s,q\in\mathbb{N}_+$ and set $n:=[\frac{q}{2}+\gamma]$. We define the mapping $\frak{\Gamma}$ by setting $\frak{\Gamma}(\gamma,k,p,s,q) :=(\tilde{\gamma},\tilde k,\tilde p,\tilde s,\tilde q)$, where  $\tilde{\gamma},\tilde k,\tilde p,\tilde s,\tilde q$ are given by (\ref{G1}) and (\ref{G2}).

From now on, we fix $\gamma>1, k,p,s,q\in \mathbb{N}_+$ with $k\geq 2,p\geq 8, s,q \geq 4,$ we denote
\begin{eqnarray*}
(\gamma_4,k_4,p_4,s_4,q_4)&:=&(\gamma,k,p,s,q),\\
(\gamma_i,k_i,p_i-1,s_i-1,q_i-2)&:=&\frak{\Gamma}(\gamma_{i+1},k_{i+1},p_{i+1},s_{i+1},q_{i+1})\quad\mbox{ for }1\leq i\leq 3.
\end{eqnarray*}
 We observe that, for $0\leq i\leq 3$,
\begin{eqnarray*}
&&n_{i+1}=[\frac{q_{i+1}}{2}+\gamma_{i+1}]\geq 3,\quad k^{\prime}_{i+1}=[\frac{s_{i+1}+1}{2}]+k_{i+1}+n_{i+1}\geq 7,\\
&&\gamma_i=2n_{i+1}+3\geq q_{i+1}+2\gamma_{i+1}+1\geq \gamma_{i+1}+6,\\
&&k_{i}=k_{i+1}+k_{i+1}^{\prime}-1\geq k_{i+1}+6,\\
&&p_{i}=p_{i+1}+k^{\prime}_{i+1}+1\geq p_{i+1}+8,\\
&&s_{i}=s_{i+1}+2k_{i+1}+2n_{i+1}\geq s_{i+1}+10,\\
&&q_{i}=2n_{i+1}+3\geq q_{i+1}+2\gamma_{i+1}+1\geq q_{i+1}+3.
\end{eqnarray*}

Let
\begin{eqnarray*}
\delta_1:=\tilde{S}(\delta)\quad\text{ and }\quad \delta_i:=\tilde{S}(\delta_{i-1})\  \text{ for }2\leq i\leq 4.
\end{eqnarray*}
Recall that $\tilde{S}:\mathbb{R}_+\rightarrow\mathbb{R}_+$ is a continuous function satisfying $\tilde{S}(0)=0$ and $\tilde{S}(\delta)\geq \delta$ for any $\delta>0,$ we can choose and fix a small $\delta>0$ such that $\delta_4<\delta_0,$ where $\delta_0$ is defined in Section \ref{sec-ext}.

We assume that the initial data $u_*$ satisfies
\begin{equation}\label{u*}
u_*\in H^{p_1-1}(\mathcal{O}).
\end{equation}

\no $\bullet$ \underline{Main  velocity boundary layer}\vspace{0.2cm}

Let  $\chi_2 $ a cut-off function such that $\chi_2(x)=1$ when $x\in \mathcal{V}_{\delta/2}$, and $\chi_2(x)=0$ when $x\in \cO\backslash\mathcal{V}_{\delta}.$
Set
\begin{equation}\label{g1}
g^1 :=2\CN(u^0)\chi_2(x).
\end{equation}
Then $g^1$ is
 in $C^{\infty}(\mathbb{R}_+\times\overline{\cO})$, is  supported in $\mathcal{V}_{\delta}$ as a function of $x$, is compactly supported in $(0,T)$ as a function of time, and $g^1\cdot \mathbf{n}=0$.
 By  Proposition \ref{lmu}, there exist $\xi^1\in C^{k_1-1}(\mathbb{R}_+;H^{p_1}(\cO;\mathcal{S}(\mathbb{R_+}))$ and  $v^1\in C^{k_1}_{\gamma_1}(\mathbb{R}_+;H^{p_1}(\mathcal{O};(H^{s_1}_{q_1}(\mathbb{R}_+)))$ such that
\begin{equation}\label{u11}
\begin{cases}
\partial_t v^1+u^0\cdot\nabla v^1+B^0 v^1-u^0_{\flat}z\partial_z v^1-\partial_z^2 v^1=\xi^1,\quad t\geq 0,x\in \cO,z\in \mathbb{R}_+,\\
\partial_z v^1|_{z=0}=g^1,\quad t\geq 0,x\in \cO\\
v^1|_{t=0}=0,\quad x\in \cO,z\in \mathbb{R}_+.
\end{cases}
\end{equation}
Moreover, $\xi^1$ is supported in $(\overline{\mathcal{O}}\backslash\overline{\Omega})\cap\mathcal{V}_{\delta_1}$ as a function of $x$ and is compactly supported in $(0,T)$ as a function of time $t$, and $v^1$ is supported in $\mathcal{V}_{\delta_1}$ as a function of $x$ and is supported away from $t=0$ as a function of time $t$, and $\xi^1\cdot \mathbf{n}=v^1\cdot \mathbf{n}=0,$ for any $t\geq 0, x\in \cO$ and $z\geq 0.$

\smallskip

\no $\bullet$ \underline{Main pressure boundary layer}\vspace{0.2cm}

We set
\begin{equation*}
\pi^2(t,x,z):=-\int_z^{+\infty}\bigl(-u^0\cdot\nabla \mathbf{n}\cdot v^1+v^1\cdot\nabla u^0\cdot \mathbf{n}\bigr)\,dz' .
\end{equation*}
Then $\pi^2\in C^{k_1}_{\gamma_1}(\mathbb{R}_+;H^{p_1}(\mathcal{O};H^{s_1}_{q_1-2}(\mathbb{R}_+)))$ and
\begin{equation}
\label{npi2}
\partial_z \pi^2=-u^0\cdot\nabla \mathbf{n}\cdot v^1+v^1\cdot\nabla u^0\cdot \mathbf{n}.
\end{equation}
 Moreover, $\pi^2$ is supported in $\mathcal{V}_{\delta_1}$ as a function of  $x$, and is supported away from $t=0$ as a function of time $t$.

\smallskip

\no $\bullet$ \underline{Main normal velocity boundary layer}\vspace{0.2cm}

We set
\beq
\label{w2}w^2(t,x,z) := -\int_{z}^{\infty}\dive v^1(t,x,z^{\prime})dz^{\prime}.
\eeq
Then $\partial_z w^2=\dive v^1$ and $w^2\in C^{k_1}_{\gamma_1}(\mathbb{R}_+;H^{p_1-1}(\mathcal{O};H^{s_1}_{q_1-2}(\mathbb{R}_+)))$ is supported in $\mathcal{V}_{\delta_1}$ as a function of  $x$ and its $t$ support is away from $t=0$. Similar to the proof in Section 6.1 of \cite{sueur}, we find that
\begin{equation}\label{iw2}
\int_{\partial\cO} w^2(t,x,0)dx=0.
\end{equation}

\no $\bullet$ \underline{Main  backflow  velocity}\vspace{0.2cm}

Let $\phi^2$ be a solution of the following Neumann problem:
\begin{equation}\label{phi2}
\begin{cases}
\Delta \phi^2=0\quad \text{ in }\mathcal{O},\\
\partial_{\mathbf{n}}\phi^2=-w^2(t,x,0)\quad\text{ on } \partial\mathcal{O}.
\end{cases}
\end{equation}
Thanks to (\ref{iw2}), there exists a unique solution $\phi^2\in C^{k_1}_{\gamma_1}(\mathbb{R}_+;H^{p_1}(\mathcal{O}))$ up to a constant and $\phi^2$ is supported away from $t=0$ as a function of time $t$.

\smallskip

\no $\bullet$ \underline{Linearized Euler flow}\vspace{0.2cm}

It follows from  Lemma \ref{lmu0} that  $\Delta u^0$ is supported in $\overline{\cO}\backslash\Omega$ and is smooth.
Thus, by Lemma \ref{np} and (\ref{u*}), there are $\nu^2\in C^{k_1}(\mathbb{R}_+;H^{p_1-2}(\cO))$, supported in $\overline{\cO}\setminus\overline{\Omega}$ as a function of $x$,  $u^2 \in C^{k_1}(\mathbb{R}_+;H^{p_1-1}(\cO)) $ and $p^2\in C_{\gamma_1}^{k_1-1}(\mathbb{R}_+;H^{p_1-1}(\cO))$ such that
\begin{equation}\label{u21}
\begin{cases}
\partial_t u^2+u^0\cdot\nabla u^2+ u^2\cdot\nabla u^0+\nabla p^2=\nu^2+\Delta u^0,\quad t\geq t,x\in \cO,\\
 u^2\cdot \mathbf{n}= 0 ,\quad t\geq0,x\in \partial\cO,\\
\dive  u^2=0,\quad t\geq0,x\in \cO,\\
 u^2=u_*,\quad t=0,x\in\cO.
\end{cases}
\end{equation}
Moreover,  $\nu^2,u^2$ and $p^2$ are supported in $[0,T]$ as functions of time $t$.

\smallskip

\no $\bullet$ \underline{Subprincipal tangential boundary layer}\vspace{0.2cm}

Let
\ben
\label{f2} f^2 &:=&-\bigl[v^1\cdot\nabla v^1+2n\cdot\nabla\partial_z v^1-\Delta\varphi \partial_z v^1-w^2\partial_z v^1+\nabla\pi^2\bigr]_{\tan}\\
&&-(n\cdot\nabla u^0)_{\tan}w^2-(u^0\cdot\nabla \mathbf{n})w^2, \nonumber\\
\label{g2} g^2 &:=&2\CN(v^1)|_{z=0}\chi_2(x).
\een
By Lemma \ref{UV}, we find that $f^2\in C^{k_1}_{\gamma_1}(\mathbb{R}_+;H^{p_1-1}(\mathcal{O};H^{s_1-1}_{q_1-2}(\mathbb{R}_+)))$ and $g^2\in C^{k_1}_{\gamma_1}(\mathbb{R}_+;H^{p_1-1}(\cO))$ satisfy the conditions of Proposition \ref{lmu}, that is, $f^2$ and $g^2$ are supported in $\mathcal{V}_{\delta_1}$ as functions of $x$ and are supported away from $t=0$ as functions of time $t$, and satisfy $f^2(t,x,z)\cdot \mathbf{n}(x)=g^2(t,x)\cdot \mathbf{n}(x)=0$ for any $t\geq 0,x\in \mathcal{O} $ and $z\geq 0.$
Therefore there exist $\xi^2\in C^{k_2-1}_{\gamma_2}(\mathbb{R}_+;H^{p_2}(\cO;\mathcal{S}(\mathbb{R}_+)))$ and a solution $v^2\in C^{k_2}_{\gamma_2}(\mathbb{R}_+;H^{p_2}(\mathcal{O};H^{s_2}_{q_2}(\mathbb{R}_+)) )$ to
\begin{equation}\label{u22}
\begin{cases}
\partial_t v^2+u^0\cdot\nabla v^2+B^0 v^2-u^0_{\flat}z\partial_z v^2
-\partial_z^2 v^2=\xi^2+f^2\quad \mbox{ in }\mathbb{R}_+\times\cO\times\mathbb{R}_+,\\
\partial_z v^2|_{z=0}=g^2\quad \mbox{ on }\mathbb{R}_+\times\cO\times\{z=0\},\\
v^2|_{t=0}=0\quad \mbox{ on }\cO\times\mathbb{R}_+ .
\end{cases}
\end{equation}
Furthermore, $\xi^2$ is supported in $(\overline{\mathcal{O}}\backslash\overline{\Omega})\cap\mathcal{V}_{\delta_2}$ as a function of $x$ and is compactly supported in $(0,T)$ as a function of time $t$, and $v^2$ is supported in $\mathcal{V}_{\delta_2}$ as a function of $x$ and is supported away from $t=0$ as a function of time $t$, and $\xi^2\cdot \mathbf{n}=v^2\cdot\mathbf{n}=0$.

\smallskip

\no $\bullet$ \underline{Subprincipal pressure boundary layer}\vspace{0.2cm}

We set
\begin{eqnarray*}
\pi^3(t,x,z) :=-\int_z^{+\infty}\Big(\partial_t w^2+u^0\cdot \nabla w^2-u^0\cdot\nabla \mathbf{n}\cdot v^2+ (v^2+w^2\mathbf{n})\cdot\nabla u^0\cdot \mathbf{n} \\
\nonumber -u^0_{\flat}z'\partial_z w^2+v^1\cdot\nabla v^1\cdot \mathbf{n}-\partial_z^2 w^2+\partial_\mathbf{n} \pi^2\Big)(t,x,z')dz'.
\end{eqnarray*}
Then it follows from  Lemma \ref{UV} that $\pi^3\in C^{k_2}_{\gamma_2}(\mathbb{R}_+;H^{p_2}(\mathcal{O};H^{s_2}_{q_2-2}(\mathbb{R}_+))$ and
\begin{equation}
\label{pi3}
\begin{split}
\partial_z\pi^3=&\partial_t w^2+u^0\cdot \nabla w^2-u^0\cdot\nabla \mathbf{n}\cdot v^2+ (v^2+w^2\mathbf{n})\cdot\nabla u^0\cdot \mathbf{n}\\
&-u^0_{\flat}z\partial_z w^2+v^1\cdot\nabla v^1\cdot \mathbf{n}-\partial_z^2 w^2+\partial_\mathbf{n} \pi^2.
\end{split}
\end{equation}
Moreover, $\pi^3 $ is supported in $\mathcal{V}_{\delta_2}$ as a function of $x$ and is supported away from $t=0$ as a function of time $t$.

\smallskip

\no $\bullet$ \underline{Subprincipal normal velocity boundary layer}\vspace{0.2cm}

Let
\begin{equation}
\label{w3}
w^3 (t,x,z):=-\int_z^{\infty}\dive  (v^2+w^2\mathbf{n})(t,x,z')dz'.
\end{equation}
Then $\partial_z w^3=\dive (v^2+w^2\mathbf{n})$ and $w^3\in C^{k_2}_{\gamma_2}(\mathbb{R}_+;H^{p_2-1}(\mathcal{O};H^{s_2}_{q_2-2}(\mathbb{R}_+))$ is supported in $\mathcal{V}_{\delta_2}$ as a function of  $x$ and is supported away from $t=0$ as a function of time $t$, furthermore
\begin{equation}\label{iw3}
\int_{\partial\cO} w^3(t,x,0)dx=0.
\end{equation}

\no $\bullet$ \underline{Lower order backflow velocity}\vspace{0.2cm}

Let $\phi^3$ be a solution of the following Neumann problem:
\begin{equation}\label{phi3}
\begin{cases}
\Delta \phi^3=0\quad \text{ in }\mathcal{O},\\
\partial_{\mathbf{n}}\phi^3=-w^3(t,x,0)\quad\text{ on } \partial\mathcal{O}.
\end{cases}
\end{equation}
Thanks to (\ref{iw3}), there exists a unique solution $\phi^3\in C^{k_2}_{\gamma_2}(\mathbb{R}_+;H^{p_2}(\mathcal{O}))$ up to a constant and $\phi^3$ is supported away from $t=0$ as a function of time $t$.

\smallskip

\no $\bullet$ \underline{Lower order interior flow}\vspace{0.2cm}

We take
\begin{equation}\label{u31}
u^3(t,x)=\nu^3(t,x)=0,\,\,\, p^3(t,x)=0\quad \text{ for }t\in\mathbb{R}_+,x\in \cO.
\end{equation}

\smallskip

\no $\bullet$ \underline{Lower order tangential  velocity boundary layer}\vspace{0.2cm}

Let
\beq\label{tildef3}
\begin{split}
\tilde{f}^3:=&\nabla \pi^3+v^1\cdot\nabla (u^2+\nabla \phi^2+v^2+w^2\mathbf{n})+(u^2+\nabla \phi^2+v^2+w^2\mathbf{n})\cdot\nabla v^1
\\
&-w^2\partial_z (v^2+w^2\mathbf{n})- w^3\partial_z v^1-\Delta v^1+2n\cdot\nabla \partial_z(v^2+w^2\mathbf{n})-\Delta\varphi\partial_z (v^2+w^2\mathbf{n}),
\end{split}
\eeq
and
\ben
\label{f3} f^3 &:=&-(\tilde{f}^3)_{\tan}-(n\cdot\nabla u^0)_{\tan}w^3-(u^0\cdot\nabla \mathbf{n} )w^3,\\
\label{g3} g^3 &:=&2\CN(u^2+v^2+\nabla \phi^2+w^2\mathbf{n})|_{z=0}\chi_2(x).
\een
Thanks to Lemma \ref{UV}, $f^3\in C^{k_2}_{\gamma_2}(\mathbb{R}_+;H^{p_2-1}(\mathcal{O};,H^{s_2-1}_{q_2-2}(\mathbb{R}_+))$ and $g^3\in C^{k_2}_{\gamma_2}(\mathbb{R}_+;H^{p_2-1}(\cO))$
and satisfy $f^3(t,x,z)\cdot \mathbf{n}(x)=g^3(t,x)\cdot \mathbf{n}(x)=0$ for any $t\geq 0,x\in \mathcal{O} $ and $z\geq 0.$
Moreover $f^3$ and $g^3$ are supported in $\mathcal{V}_{\delta_2}$ as functions of $x$. Then, by using Proposition \ref{lmu}, there exist
 $\xi^3\in C^{k_3-1}_{\gamma_3}(\mathbb{R}_+;H^{p_3}(\cO;\mathcal{S}(\mathbb{R}_+)))$ , $ v^3\in C^{k_3}_{\gamma_3}(\mathbb{R}_+;H^{p_3}(\mathcal{O};H^{s_3}_{q_3}(\mathbb{R}_+))$ and $v^3_0\in H^{p_3+2}(\cO;C_0^{\infty}(\overline{\mathbb{R}_+}))$ such that
\begin{equation}\label{u3}
\begin{cases}
\partial_t v^3+u^0\cdot\nabla v^3+B^0 v^3-u^0_{\flat}z\partial_z v^3
-\partial_z^2 v^3=\xi^3+f^3\quad\mbox{ in }\ \mathbb{R}_+\times\cO\times\mathbb{R}_+,\\
\partial_z v^3|_{z=0}=g^3\quad\mbox{ on }\mathbb{R}_+\times\cO\times\{z=0\},\\
 v^3|_{t=0}=v^3_0\quad\mbox{ on }\ \cO\times\mathbb{R}_+.
\end{cases}
\end{equation}
Moreover, $\xi^3$ is supported in $(\overline{\mathcal{O}}\backslash\overline{\Omega})\cap\mathcal{V}_{\delta_3}$ as a function of $x$, $v^3$ is supported in $\mathcal{V}_{\delta_3}$ as a function of $x$ and $\xi^3\cdot \mathbf{n}=v^3\cdot\mathbf{n}=0$.

\smallskip

\no $\bullet$ \underline{A lower order pressure boundary layer}\vspace{0.2cm}

We set
\beno
\begin{split}
\pi^4(t,x,z) :=-\int_z^{+\infty}\Big(& \partial_t w^3+u^0\cdot \nabla w^3-u^0\cdot \nabla \mathbf{n}\cdot v^3-u^0_{\flat}z'\partial_z w^3
\\
&+  (v^3+w^3\mathbf{n})\cdot\nabla u^0\cdot \mathbf{n}-\partial_z^2 w^3+\tilde{f}^3\cdot \mathbf{n}\Big)(t,x,z')dz'.
\end{split} \eeno
Hence $\partial_z\pi^4\in C^{k_3}_{\gamma_3}(\mathbb{R}_+;H^{p_3}(\mathcal{O};H^{s_3}_{q_3-2}\mathbb{R}_+))$ and \begin{equation}
\label{p4}
\partial_z \pi^4 := \partial_t w^3+u^0\cdot \nabla w^3-u^0\cdot \nabla \mathbf{n}\cdot v^3+  (v^3+w^3\mathbf{n})\cdot\nabla u^0\cdot \mathbf{n}-u^0_{\flat}z\partial_z w^3-\partial_z^2 w^3+\tilde{f}^3\cdot \mathbf{n}.
\end{equation}
Furthermore, $\pi^4 $ is supported in $\mathcal{V}_{\delta_3}$ as a function of $x$.

\smallskip

\no $\bullet$ \underline{A lower order normal velocity boundary layer}\vspace{0.2cm}

Set
\begin{equation}
\label{w4}
w^4(t,x,z):=-\int_z^{\infty}\dive( v^3+w^3\mathbf{n})(t,x,z')dz'.
\end{equation}
Then $\partial_z w^4=\dive(v^3+w^3\mathbf{n})$ and  $w^4$ belongs to $C^{k_3}_{\gamma_3}(\mathbb{R}_+;H^{p_3-1}(\mathcal{O};H^{s_3}_{q_3-2}(\mathbb{R}_+))$
and is supported in $\mathcal{V}_{\delta_3}$ as a function of  $x$, with $w^4|_{t=0}=w_0^4\in H^{p_3+1}(\cO;C_0^{\infty}(\overline{\mathbb{R}_+}))$.
Moreover $w^4$ satisfies
\begin{equation}\label{iw4}
\int_{\partial\cO} w^4(t,x,0)dx=0.
\end{equation}

\no $\bullet$ \underline{A lower order backflow velocity}\vspace{0.2cm}

Let $\phi^4$ be a solution of the following Neumann problem:
\begin{equation}\label{phi4}
\begin{cases}
\Delta \phi^4=0\quad \text{ in }\mathcal{O},\\
\partial_{\mathbf{n}}\phi^4=-w^4(t,x,0)\quad\text{ on } \partial\mathcal{O}.
\end{cases}
\end{equation}
Thanks to (\ref{iw4}), there exists a unique solution $\phi^4\in C^{k_3}_{\gamma_3}(\mathbb{R}_+;H^{p_3}(\mathcal{O}))$ up to a constant, with $\phi^4|_{t=0}=\phi^4_0\in H^{p_3+2}(\cO)$.

\smallskip

\no $\bullet$ \underline{A lower order interior flow}\vspace{0.2cm}

Let
\begin{equation}
\label{l4}
l^4  :=- u^2\cdot\nabla u^2+\Delta u^2\in C^{k_2}_{\gamma_2}(\mathbb{R}_+;H^{p_2-2}(\cO)) ,
\end{equation}
 and obeserve that $\curl l^4$ is supported in $[0,T]$ as a function of  time.
By Lemma \ref{np}, there are $\nu^4\in C^{k_3}(\mathbb{R}_+;H^{p_3-2}(\cO))$, supported in $\overline{\cO}\setminus\overline{\Omega}$ as a function of $x$, ${u}^4 $ in $ C^{k_3}_{\gamma_3}(\mathbb{R}_+;H^{p_3-1}(\cO))$ and $p^4\in C_{\gamma_3}^{k_3-1}(\mathbb{R}_+; H^{p_3-1}(\cO))$ such that
\begin{equation} \label{u41}
\begin{cases}
\partial_t u^4+u^0\cdot\nabla u^4+ u^4\cdot\nabla u^0+\nabla p^4=\nu^4+l^4\quad\mbox{ in }\ R_+\times\cO,\\
\dive  u^4=0\quad\mbox{ in }\ R_+\times\cO,\\
 u^4\cdot \mathbf{n}=0 \quad\mbox{ on }\ R_+\times\partial\cO,\\
 u^4|_{t=0}=0\quad\mbox{ in }\ \cO .
\end{cases}
\end{equation}
Moreover, $\xi^4,u^4$ and $p^4$ are supported in $[0,T]$ as functions of time $t$.

\smallskip

\no $\bullet$ \underline{a lower order tangential  velocity boundary layer}\vspace{0.2cm}

Set
\beq\label{def-f4-tilde}
\begin{split}
\tilde{f}^4 :=&v^1\cdot\nabla (u^3+\nabla \phi^3+v^3+w^3\mathbf{n})+(u^2+\phi^2+v^2+w^2\mathbf{n})\cdot\nabla (v^2+w^2\mathbf{n}) \\
 &+ (v^2+w^2\mathbf{n})\cdot\nabla (u^2+\phi^2+v^2+w^2\mathbf{n})+(u^3+\nabla \phi^3+v^3+w^3\mathbf{n})\cdot\nabla v^1\\
 &-w^2\partial_z (v^3+w^3\mathbf{n})-w^3\partial_z (v^2+w^2\mathbf{n})- w^4\partial_z v^1-\Delta (v^2+w^2\mathbf{n})\\
 & +2n\cdot\nabla\partial_z (v^3+w^3\mathbf{n})-\Delta\varphi\partial_z (v^3+w^3\mathbf{n})+\nabla\pi^4,
\end{split}
\eeq
and
\ben
\label{def-f4}
f^4 &:=&-\tilde{f}^4_{\tan}-(n\cdot\nabla u^0)_{\tan}w^4-(u^0\cdot\nabla \mathbf{n}) w^4 ,\\
\label{g4}
 g^4 &:=&2\CN(u^3+\nabla\phi^3+v^3+w^3\mathbf{n})|_{z=0}\chi_2(x).
 \een
 Thanks to Lemma \ref{UV}, one can check that $f^4\in C^{k_3}_{\gamma_3}(\mathbb{R}_+;H^{p_3-1}(\mathcal{O};H^{s_3-1}_{q_3-2}(\mathbb{R}_+))$ and $g^4\in C^{k_3}_{\gamma_3}(\mathbb{R}_+;H^{p_3-1}(\cO))$
and satisfy $f^4(t,x,z)\cdot \mathbf{n}(x)=g^4(t,x)\cdot \mathbf{n}(x)=0$ for any $t\geq 0,x\in \mathcal{O} $ and $z\geq 0.$
Moreover $f^4$ and $g^4$ are supported in $\mathcal{V}_{\delta_3}$ as functions of $x$. Then by using Proposition \ref{lmu},
there exist  $\xi^4\in C^{k_4-1}_{\gamma_4}(\mathbb{R}_+;H^{p_4}(\cO;\mathcal{S}(\mathbb{R}_+)))$, $ v^4 \in C^{k_4}_{\gamma_4}(\mathbb{R}_+;H^{p_4}(\mathcal{O};H^{s_4}_{q_4}(\mathbb{R}_+))$ and $v^4_0\in H^{p_4+2}(\cO;C_0^{\infty}(\overline{\mathbb{R}_+}))$ such that
\begin{equation}\label{u4}
\begin{cases}
\partial_t  v^4 +u^0\cdot\nabla  v^4 +B^0   v^4 -u^0_{\flat}z\partial_z  v^4
-\partial_z^2  v^4 =\xi^4+f^4\quad\mbox{ in }\cO,\\
\partial_z  v^4 |_{z=0}=g^4\quad\mbox{ in }\cO,\\
  v^4 |_{t=0}=v^4_0\quad\mbox{ in }\cO.
\end{cases}
\end{equation}
Moreover $\xi^4$ is supported in $(\overline{\mathcal{O}}\backslash\overline{\Omega})\cap\mathcal{V}_{\delta_4}$ as a function of $x$ and is compactly supported in $(0,T)$ as a function of time $t$, and $v^4$ is supported in $\mathcal{V}_{\delta_4}$ as a function of $x$ and $\xi^4\cdot\mathbf{n}=v^4\cdot\mathbf{n}=0$.

\smallskip

\no $\bullet$ \underline{A last pressure boundary layer}\vspace{0.2cm}

We set
\beno
\begin{split}
\pi^5(t,x,z) :=-\int_z^{+\infty}\Big(&\partial_t w^4+u^0\cdot \nabla w^4+
 (v^4+w^4\mathbf{n})\cdot\nabla u^0\cdot \mathbf{n}\\
&-u^0\cdot\nabla \mathbf{n}\cdot  v^4-u^0_{\flat}z'\partial_z w^4-\partial_z^2w^4+\tilde{f}^4\cdot \mathbf{n}\Big)(t,x,z')dz' ,
\end{split} \eeno

Then $\pi^5\in C^{k_4}_{\gamma_4}(\mathbb{R}_+;H^{p_4}(\mathcal{O};H^{s_4}_{q_4-2}(\mathbb{R}_+))$ and
\begin{equation}\label{pi5}
\partial_z \pi^5 := \partial_t w^4+u^0\cdot \nabla w^4-u^0\cdot\nabla \mathbf{n}\cdot  v^4+  (v^4+w^4\mathbf{n})\cdot\nabla u^0\cdot \mathbf{n}-u^0_{\flat}z\partial_z w^4-\partial_z^2w^4+\tilde{f}^4\cdot \mathbf{n}.
\end{equation}
Moreover, $\pi^5 $ is supported in $\mathcal{V}_{\delta_4}$ as a function of $x$.

In summary, we have now constructed
\beno
\begin{split}
u^{j}\in & C^{k_{j-1}}(\mathbb{R}_+;H^{p_{j-1}-1}(\cO)),p^j\in C^{k_{j-1}-1}(\mathbb{R}_+;H^{p_{j-1}-1}(\cO)),  \quad 2\leq j\leq 4,\\
\nu^j\in & C^{k_{j-1}}(\mathbb{R}_+;H^{p_{j-1}-2}(\cO)),\phi^{j}\in  C^{k_{j-1}}_{\gamma_{j-1}}(\mathbb{R}_+;H^{p_{j-1}}(\cO)),\quad 2\leq j\leq 4,\\
v^{j}\in & C^{k_j}_{\gamma_j}(\mathbb{R}_+;H^{p_j}(\cO;H^{s_j}_{q_j}(\mathbb{R}_+))), \pi^{j+1}\in C^{k_j}_{\gamma_j}(\mathbb{R}_+;H^{p_j}(\cO;H^{s_j}_{q_j-2}(\mathbb{R}_+))),\quad 1\leq j\leq 4,\\
w^{j}\in& C^{k_{j-1}}_{\gamma_{j-1}}(\mathbb{R}_+;H^{p_{j-1}-1}(\cO;H^{s_{j-1}}_{q_{j-1}-2}(\mathbb{R}_+))),\quad 2\leq j\leq 4,\\
\xi^{j}\in & C^{k_j-1}_{\gamma_j}(\mathbb{R}_+;H^{p_j}(\cO;\mathcal{S}(\mathbb{R}_+))),\quad 1\leq j\leq 4.
\end{split} \eeno
Moreover, $u^j,p^j,\nu^j$ and $\xi^j$ are supported in $[0,T]$ as functions of time $t$, $\nu^j$ and $\xi^j$ are supported in $\overline{\cO}\backslash\overline{\Omega}$ as functions of $x$,
$v^{j},w^{j+1},\pi^{j+1}$ are supported in $\mathcal{V}_{\delta_{j}}$ as functions of $x$ and $v^j\cdot \mathbf{n}=\xi^j\cdot\mathbf{n}=0.$
Furthermore, $v^1,v^2,\phi^2,\phi^3,w^2$ and $w^3$ are supported away from $t=0$ as a function of time $t$.

\subsection{Construction of the family of approximate solutions}
\label{cfas}

Let us start with a notation: for a profile $f(t,x,z)$, we define
\begin{equation*}
\{f\}_\e:=f\Bigl(t,x,\frac{\varphi(x)}{\sqrt{\e}}\Bigr).
\end{equation*}
We define the approximate solutions via
 \ben
\label{wc1}
u^{\e}_a &:=& u^0
+ \sqrt{\e}  \{  v^1  \}_\e
+ \sum_{j=2}^4\e^{\frac{j}{2}}  \big( u^j + \nabla \phi^j  +  \{  v^j  \}_\e+  \{  w^j  \}_\e \mathbf{n} \big),\\
\label{wc2} p^{\e}_a &:=& p^0
+ \sum_{j=2}^4 \e^{\frac{j}{2}}  \big( p^j   -  \partial_t \phi^j - u^0\cdot \nabla \phi^j  +  \{  \pi^j  \}_\e   \big),
\\
\label{wc3}  \xi^{\e} &:=& \nu^0 + \sqrt{\e}  \{  \xi^1  \}_\e
+  \sum_{j=2}^4\e^{\frac{j}{2}}  \big(   \nu^j + \{  \xi^j  \}_\e  \big) .
 \een

\subsection{Consistency estimates of the approximate solutions}

\begin{lemma}\label{V}
{\sl Let $\gamma>0, k,p,s,q \in \mathbb{N}$ with $p\geq 3$ and $s\geq 1.$
Let the  profile $V\in C^{k}_{\gamma}(\mathbb{R}_+;H^p(\mathcal{O};$ $H^s_q(\mathbb{R}_+)))$ and is supported in $\mathcal{V}_{\delta_0}.$
 Then one has
 \begin{itemize}

\item[(1)] for $0\leq j\leq k$, $p_1+p_2\leq p-1$ and $p_2\leq s$,
\begin{eqnarray}\label{V2}
\|\partial_t^jZ^{p_1}(\sqrt{\e}\partial_{\mathbf{n}})^{p_2}\{V\}_{\e}\|_{L^2(\cO)}\lesssim \e^{\frac{1}{4}}\t^{-\gamma},
\end{eqnarray}

 \item[(2)] for $0\leq j\leq k$, $p_1+p_2\leq p-2$ and $p_2\leq s-1,$
\begin{eqnarray}\label{Vif}
\|\partial_t^jZ^{p_1}(\sqrt{\e}\partial_{\mathbf{n}})^{p_2}\{V\}_{\e}\|_{L^{\infty}(\cO)}\lesssim \t^{-\gamma},
\end{eqnarray}

 \item[(3)] for $  m\leq p-1,$
\begin{eqnarray}\label{Vb}
\|\{V\}_{\e}\|_{H^m(\partial\cO)}\lesssim \t^{-\gamma}.
\end{eqnarray}
\end{itemize}
}
\end{lemma}

\begin{proof} We first observe that
\begin{gather*}
\sqrt{\e}\partial_\mathbf{n}\{V\}_{\e}=\sqrt{\e}\{\partial_\mathbf{n}V\}_\e +\{\partial_zV\}_\e ,\\
Z^0\{V\}_{\e}=\{Z^0V\}_\e +\{z\partial_zV\}_\e \ \text{ and } Z^j\{V\}_{\e}=\{Z^jV\}_{\e}\ \text{ for }\ 1\leq j\leq 5.
\end{gather*}
We can take the normal derivatives $p_2$ times, take the tangential derivatives $p_1$ times and take the time derivatives $j$ times and use
  \cite[Lemma 3]{iftimie}
 to get (\ref{V2}).
For (\ref{Vif}), we use Sobolev imbedding $H^1(\mathbb{R}_+)\hookrightarrow L^{\infty}(\mathbb{R}_+)$ for variable $z$ and $H^2(\cO)\hookrightarrow L^{\infty}(\cO) $ for variable $x$. For (\ref{Vb}), we use the trace theorem to get $H^{m+1}(\cO)\hookrightarrow H^m(\partial\cO).$
\end{proof}

Let us now turn to the justification of the consistence of the approximate solutions constructed in (\ref{wc1}-\ref{wc3}) with the system
(\ref{eqa2}-\ref{eqa5}).

\smallskip

\no $\bullet$ \underline{Consistency of \eqref{eqa2}. Definition  and estimate of $F$.}\vspace{0.2cm}

By \eqref{wc1}-\eqref{wc3}, \eqref{u0}, \eqref{BF->E}, \eqref{u11}-\eqref{w2}, \eqref{u21},  \eqref{f2}, \eqref{u22},
\eqref{pi3}, \eqref{w3}, \eqref{u31}-\eqref{f3}, \eqref{u3}, \eqref{p4}, \eqref{w4}, \eqref{l4}-\eqref{def-f4}, \eqref{u4} and \eqref{pi5}, we find that
$u^{\e}_a$ satisfies \eqref{eqa2div} with
\beq\label{defF}
\begin{split}
\hspace{1cm}F:=&-\{n\partial_z\pi^5\}_\e +\sqrt{\e}\Big\{v^1\cdot\nabla (u^4+\nabla\phi^4+v^4+w^4\mathbf{n})\\
&+(u^2+\nabla\phi^2+v^2+w^2\mathbf{n})\cdot\nabla (u^3+\nabla\phi^3+v^3+w^3\mathbf{n})\\
&-w^2\partial_z (v^4+w^4\mathbf{n})+(u^3+\nabla\phi^3+v^3+w^3\mathbf{n})\cdot \nabla (u^2+\nabla\phi^2+v^2+w^2\mathbf{n})\\
&-w^3\partial_z (v^3+w^3\mathbf{n})+(u^4+\nabla\phi^4+v^4+w^4\mathbf{n})\cdot \nabla v^1-w^4\partial_z (v^2+w^2\mathbf{n})\\
& -\Delta(u^3+\nabla\phi^3+v^3+w^3\mathbf{n})+2n\cdot\nabla\partial_z (v^4+w^4\mathbf{n})-\Delta\varphi\partial_z (v^4+w^4\mathbf{n})\Big\}_{\e}\\
& +\e\Big\{(u^2+\nabla\phi^2+v^2+w^2\mathbf{n})\cdot\nabla (u^4+\nabla\phi^4+v^4+w^4\mathbf{n})\\
&+(u^3+\nabla\phi^3+v^3+w^3\mathbf{n})\cdot\nabla (u^3+\nabla\phi^3+v^3+w^3\mathbf{n})-w^3\partial_z (v^4+w^4\mathbf{n})\\
&+(u^4+\nabla\phi^4+v^4+w^4\mathbf{n})\cdot \nabla (u^2+\nabla\phi^2+v^2+w^2\mathbf{n})-w^4\partial_z (v^3+w^3\mathbf{n})\\
&-\Delta (u^4+\nabla\phi^4+v^4+w^4\mathbf{n})\Big\}_\e \\
&+\e^{\frac{3}{2}}\Big\{(u^3+\nabla\phi^3+v^3+w^3\mathbf{n})\cdot\nabla (u^4+\nabla\phi^4+v^4+w^4\mathbf{n})\\
&+(u^4+\nabla\phi^4+v^4+w^4\mathbf{n})\cdot\nabla (u^3+\nabla\phi^3+v^3+w^3\mathbf{n})-w^4\partial_z (v^4+w^4\mathbf{n})\Big\}_\e\\
& +\e^2\Big\{(u^4+\nabla\phi^4+v^4+w^4\mathbf{n})\cdot\nabla (u^4+\nabla\phi^4+v^4+w^4\mathbf{n})\Big\}_\e  .
\end{split} \eeq
By  the constructions of $u^i,\phi^i,v^i,w^i$ and the definition of $\gamma_i,k_i,p_i,s_i,q_i,$ we have $u^i+\nabla \phi^i\in C^{k}_{\gamma}(\mathbb{R}_+;H^p(\cO))$ and $v^i+w^i\mathbf{n}\in C^{k}_{\gamma}(\mathbb{R}_+;H^p(\cO;H^s_q(\mathbb{R}_+))).$ Then (\ref{FH1}) and \eqref{FH2} for the part of $F$ is a direct consequence of Lemma \ref{UV} and Lemma \ref{V}.

\smallskip

\no $\bullet$ \underline{Consistency  of \eqref{eqa2div}. Definition and estimate of $H$.}\vspace{0.2cm}

By \eqref{divu0}, \eqref{phi2}, \eqref{u21}, (\ref{w3}), \eqref{phi3}, \eqref{u31}, (\ref{w4}), \eqref{phi4} and \eqref{u41}, we find that $u^{\e}_a$ satisfies \eqref{eqa2div} with
\begin{gather} \label{defH}
H:=  \{\dive (v^4+w^4\mathbf{n})\}_\e .
\end{gather}

By construction $\dive ( v^4+w^4\mathbf{n})\in C^{k}_{\gamma}(\mathbb{R}_+;H^{p-1}(\cO;H^s_q(\mathbb{R}_+)))$, so Lemma \ref{V} immediately leads to the estimates, (\ref{FH1}), \eqref{FH2} and (\ref{HG}) for the part of $H$.

\smallskip

\no $\bullet$ \underline{Consistency  of \eqref{eqa4}. Definition  and estimate of $G$.}\vspace{0.2cm}

By (\ref{g1}), (\ref{u11}), (\ref{g2}), (\ref{u22}), (\ref{g3}), (\ref{u3}), (\ref{g4} ) and (\ref{u4}), $u^{\e}_a$ satisfies \eqref{eqa4} with
\begin{gather}  \label{defG}
G:=\CN(u^4+\nabla\phi^4+v^4+ w^4\mathbf{n})|_{z=0} .
\end{gather}
By construction, $u^4+\nabla \phi^4\in C^{k_3}_{\gamma_3}(\mathbb{R}_+;H^{p_3-1}(\cO))$ and $v^4+w^4\mathbf{n}\in C^{k}_{\gamma}(\mathbb{R}_+;H^p(\cO;H^s_q(\mathbb{R}_+)))$. By definition of $\gamma_i,k_i,p_i,s_i,q_i,$ we find that $G\in C^{k}_{\gamma}(\mathbb{R}_+;H^{p-1}(\cO))$, which is exactly (\ref{HG}) for the part of $G$.

\smallskip

\no $\bullet$ \underline{Consistency  of \eqref{eqa3} and \eqref{eqa5}. Definition  and estimate of $R_0.$}\vspace{0.2cm}

 By \eqref{u0n}, \eqref{u00}, \eqref{u11}, \eqref{w2}, \eqref{phi2}, \eqref{u21}, \eqref{u22}, \eqref{w3}, \eqref{phi3}, \eqref{u31}, \eqref{u3}, \eqref{w4}, \eqref{phi4}, \eqref{u41} and \eqref{u4}, \eqref{eqa3} and \eqref{eqa5} are satisfied with
 \begin{equation}\label{defua0}
 R_0=-\e^{\frac{1}{2}}v^3_0-(v^4_0+\nabla\phi^4_0+w^4_0\mathbf{n}),
 \end{equation}
and \eqref{R0}) is a direct consequences of Lemma \ref{V}.

 \subsection{ Verification of \eqref{Ra}-\eqref{uja}}
Let us verify \eqref{Ra} and \eqref{ua0} first. Since  $u^0$ is smooth and has compact support in $t$,
\begin{eqnarray}\label{vf1}
\|u^0\|_{W^{1,\infty}(\cO)}+\|\nabla u^0\|_{m-1,\infty}+\|\nabla^2 u^0\|_{m-1,\infty}\lesssim \chi_{[0,T]}(t).
\end{eqnarray}
By construction, $v^j\in C^{k}_{\gamma}(\mathbb{R}_+;H^p(\cO;H^s_q(\mathbb{R}_+))),1\leq j\leq 4.$ Then it follows from (\ref{Vif}) of Lemma \ref{V} that, for $1\leq j\leq 4,m\leq p-3$,
\begin{eqnarray}\label{vf2}
\sqrt{\e}\|\{v^j\}_{\e}\|_{W^{1,\infty}}+\sqrt{\e}\|\nabla\{v^j\}_{\e}\|_{m-1,\infty}+\e\|\nabla^2\{v^j\}_{\e}\|_{m-1,\infty}\lesssim \t^{-\gamma}.
\end{eqnarray}
The same inequality holds for $w^j\mathbf{n}$ with $2\leq j\leq 4,$ since they also belong to the space $C^{k}_{\gamma}(\mathbb{R}_+;H^p(\cO;H^s_q(\mathbb{R}_+))).$ For $u^j,$ $2\leq j\leq4$, it belongs to $C^k(\mathbb{R}_+;H^p(\cO))$ and is supported in $[0,T]$. Hence Sobolev imbedding Theorem ensures that, for $2\leq j\leq 4,m\leq p-3,$
\begin{eqnarray}\label{vf3}
\|u^j\|_{W^{1,\infty}(\cO)}+\|\nabla u^j\|_{m-1,\infty}+\|\nabla^2 u^j\|_{m-1,\infty}\lesssim \chi_{[0,T]}(t).
\end{eqnarray}
For $\nabla \phi^j,$ $2\leq j\leq 4$, it belongs to $C^k_{\gamma}(\mathbb{R}_+;H^{p}(\cO))$. Then it follows form
 Sobolev imbedding Theorem that, for $2\leq j\leq 4,m\leq p-3,$
\begin{eqnarray}\label{vf4}
\|\nabla\phi^j\|_{W^{1,\infty}(\cO)}+\|\nabla^2\phi^j\|_{m-1,\infty}+\|\nabla^3\phi^j\|_{m-1,\infty}\lesssim \t^{-\gamma}.
\end{eqnarray}
Combine (\ref{vf1})-(\ref{vf4}), we have verified \eqref{Ra} and \eqref{ua0}.

\smallskip

Let us move on to \eqref{uja}.
Since  $u^0$ is smooth and $u^j\in C^{k}(\mathbb{R}_+;H^p(\cO))$ for $2\leq j\leq 4$, and they both supported in $[0,T]$, one has
\begin{eqnarray*}
\|u^0\|_{H^1(\cO)}+\|u^j\|_{H^1(\cO)}\leq \chi_{[0,T]}(t).
\end{eqnarray*}
For $\nabla \phi^j \in C^k_{\gamma}(\mathbb{R}_+;H^{p}(\cO)),$ $2\leq j\leq 4,$
\begin{eqnarray*}
\|\nabla\phi^j\|_{H^1(\cO)}\lesssim \t^{-\gamma}.
\end{eqnarray*}
For $v^j\in C^{k}_{\gamma}(\mathbb{R}_+;H^p(\cO;H^s_q(\mathbb{R}_+))),$ $1\leq j\leq 4$, it follows from Lemma \ref{V} that
\begin{eqnarray*}
\sqrt{\e}\|\{v^j\}_\e\|_{H^1(\cO)}\lesssim\e^{\frac{1}{4}}\t^{-\gamma}.
\end{eqnarray*}
The same estimates holds for $w^j\mathbf{n}.$
By gathering the above estimates, we find that
\begin{eqnarray*}
\|u^{\e}_a(t,\cdot)\|_{H^1(\cO)}\lesssim \chi_{[0,T]}(t)+\e\t^{-\gamma}+\e^{\frac{1}{4}}\t^{-\gamma}.
\end{eqnarray*}
As a result, it comes out
\begin{equation*}
\|u^{\e}_a({T}/{\e},\cdot)\|_{H^1(\cO)}\lesssim \e^{\gamma+\frac{1}{4}}.
\end{equation*}
Since $\gamma>1$, \eqref{uja} holds true.

\section{Estimates of the remainder R}
\label{sec-Remainder}

The goal of this section is to establish the  \textit{a priori} estimate
\eqref{APRIORI-R} for the remainder term $R$  defined by \eqref{DEF-R}.
We also introduce the remainder pressure term $ \pi$ such that
 $p^{\e} =  p^{\e}_a +\e^2 \pi$.
Then in view of \eqref{NSA}, \eqref{eqa} and \eqref{DEF-R}, we write
\begin{subequations} \label{R}
\begin{gather}
 \label{R1}
\partial_t R-\e\Delta R+u^{\e}\cdot\nabla R+R\cdot\nabla u^{\e}_a+\nabla\pi=-F\quad \mbox{ and } \quad    \dive R=-H
\quad  \mbox{ in }\ \R_+\times\cO, \\  \label{R2}
R\cdot \mathbf{n}=0 \quad \mbox{ and } \quad
\CN(R)=-G\quad \mbox{ on }\ \R_+\times\partial\cO,\\  \label{R3}
R|_{t=0}=R_0 \quad \mbox{ in }\cO.
\end{gather}
\end{subequations}
These equations are satisfied up to the time  $T^\e$  introduced in Section \ref{Sec-RE}.
At the end of this section, once
the  \textit{a priori} estimate \eqref{APRIORI-R} in hands, we will deduce
  that $T^\e \geq \frac{T}{\e}$.

We will start with a
$L^2$ estimate in Subsection \ref{ldeux}, then we will turn to tangential derivatives estimates in Subsection \ref{tangen}.
We will also need to handle the estimate of one normal derivative, and for that, we introduce
an appropriate substitute to the vorticity, see \eqref{defeta}, which is in the spirit of  \cite{masmoudi}.
We will see in Subsection \ref{sisi-eta} that this quantity, as the vorticity, allows to estimate one normal derivative.
The advantage of this quantity over the vorticity is that
its time evolution is easier to be investigated; this will be done in Subsection \ref{time-eta}.
The estimate of the terms involving the pressure are quite difficult and are therefore postponed to
Subsection  \ref{sub-pressure}.
An estimate of $\|R\|_{1,\infty}$  will be obtained in Subsection
\ref{lipR}. The end of  the proof of \eqref{APRIORI-R} will be given in Subsection
\ref{Sec-RE}.

\subsection{$L^2$ estimates} \label{ldeux}
From now on, we simplify $\|\cdot\|_{L^2(\cO)}$ as $\|\cdot\|$.
\begin{proposition}\label{S5prop1}
{\sl There exist a constant $C>0$, such that the remainder $R$ satisfies
\begin{equation} \label{el2}
\|R(t)\|^2+\e\int_0^t\|\nabla R\|^2dt'\leq C\e^{-\frac{1}{4}}\quad\mbox{ for }\ 0\leq t\leq T^\e .
\end{equation}}
\end{proposition}

\begin{proof}
Let $\mathbb{P}$ the Leray projection operator to the divergence free vector field,  we decompose
$R$ into  $R=\mathbb{P}R+\nabla \phi.$
Hence $\phi $ satisfies
$\Delta\phi=\dive R=-H$ in $\cO$ and
 $\partial_\mathbf{n}\phi=R\cdot \mathbf{n}=0$  on $\partial\cO$.
By elliptic regularity and (\ref{FH1}), one has
\begin{equation}\label{PH2}
\|(I-\mathbb{P})R\|_{H^1(\cO)}\lesssim \|H\|_{L^2(\cO)}\lesssim\e^{\frac{1}{4}}\t^{-\gamma}.
\end{equation}
Next we estimate $\mathbb{P}R$. Indeed by taking $L^2$ inner product of \eqref{R1} with  $\mathbb{P}R,$ we find
\beq \label{PHC}
\begin{split}
 \frac{1}{2}\f{d}{dt}\|\mathbb{P}R(t)\|^2-\e\int_{\cO}\Delta R\cdot \mathbb{P}R
+&\int_{\cO} (u^{\e}\cdot\nabla R)\cdot \mathbb{P}R\\
+&\int_{\cO}(R\cdot\nabla u^{\e}_a )\cdot \mathbb{P}R+\int_{\cO}\nabla \pi\cdot \mathbb{P}R
=-\int_{\cO}F\cdot\mathbb{P}R.
\end{split}
\eeq
Let us now estimate each term of \eqref{PHC}, from the right to the left.
\begin{itemize}
\item
Since $F$ satisfies (\ref{FH1}), we have
\begin{equation}\label{R21}
\vert \int_{\cO} F\cdot \mathbb{P}R \vert \lesssim \|F\|\|\mathbb{P}R\|\lesssim \e^{\frac{1}{4}}\t^{-\gamma}\bigl(\|\mathbb{P}R\|^2+1\bigr).
\end{equation}
\item While in view of \eqref{R2}, we get, by an integration by parts, that
 $$\int_{\cO}\nabla \pi\cdot \mathbb{P}R=0.$$

\item Let us now move to the term before in \eqref{PHC}.
We first deduce from  \eqref{PH2} that
\begin{equation}\label{PR}
\|R\|\lesssim\|\mathbb{P}R\|+\e^{\frac{1}{4}}\t^{-\gamma} ,
\end{equation}
which together  with \eqref{Ra} ensures that
\begin{equation}\label{R22}
\vert \int_{\cO}(R\cdot\nabla u^{\e}_a)\cdot \mathbb{P}R \vert  \lesssim \|\nabla u^{\e}_a\|_{L^{\infty}(\cO)}\|R\|\|\mathbb{P}R\|\lesssim \t^{-\gamma}\bigl(\|\mathbb{P}R\|^2+\e^{\frac{1}{2}}\t^{-2\gamma}\bigr).
\end{equation}
\item To deal with the third term in \eqref{PHC}, we start with using  again  the Helmholtz-Leray decomposition to deduce that
\begin{equation} \label{LAST}
\int_{\cO}(u^{\e}\cdot\nabla R)\cdot \mathbb{P}R=\int_{\cO}(u^{\e}\cdot\nabla \mathbb{P}R)\cdot \mathbb{P}R+\int_{\cO}(u^{\e}\cdot\nabla(I- \mathbb{P})R)\cdot \mathbb{P}R .
\end{equation}
Thanks to  \eqref{NSA2}, \eqref{NSA3}, and $\sigma^0 $ is supported on $[ 0,T]$,
we get, by using integration by parts,
that
\begin{equation}\label{R23}
\vert \int_{\cO}(u^{\e}\cdot\nabla \mathbb{P}R)\cdot \mathbb{P}R \vert
\lesssim
\t^{-\gamma}  \|\mathbb{P}R\|^2.
\end{equation}
Whereas  to deal with the last term in \eqref{LAST}, we first use the decomposition \eqref{DEF-R}
to obtain
\begin{equation*}
\vert \int_{\cO}(u^{\e}\cdot\nabla(I- \mathbb{P})R)\cdot \mathbb{P}R \vert \lesssim\|\nabla(I-\mathbb{P})R\|(\|u^{\e}_a\|_{L^{\infty}}\|\mathbb{P}R\|+\e^2\|R\|_{L^4}^2) .
\end{equation*}
Observing from  Korn's inequality and (\ref{PR}) that
\begin{equation}\label{RH1}
\| R\|_{H^1}
 \lesssim\|D(R)\|+\|\mathbb{P}R\|+\e^{\frac{1}{4}}\t^{-\gamma}.
\end{equation}
 Then recalling that
 $\|R\|_{L^4}\lesssim \|R\|^{\frac{1}{4}}\|\nabla R\|^{\frac{3}{4}},$
and
 using again (\ref{PR}), we find
\begin{eqnarray*}
\|R\|_{L^4}^2&\lesssim&(\|\mathbb{P}R\|+\e^{\frac{1}{4}}\t^{-\gamma})^{\frac{1}{2}} (\|D(R)\|+\|\mathbb{P}R\|+\e^{\frac{1}{4}}\t^{-\gamma})^{\frac{3}{2}}\\
&\lesssim& \lambda \|D(R)\|^2+C_{\lambda}(\|\mathbb{P}R\|^2+\e^{\frac{1}{2}}\t^{-2\gamma}),
\end{eqnarray*}
for a small constant $\lambda>0$,  where in the last step, we used Young's inequality.

As a consequence, we deduce from   \eqref{PH2} and \eqref{Ra} that
\begin{eqnarray}
\nonumber\vert \int_{\cO} (u^{\e}\cdot\nabla(I- \mathbb{P})R)\cdot \mathbb{P}R \vert
\!\!\!&\lesssim&\!\!\!\e^{\frac{1}{4}}\t^{-\gamma}(\t^{-\gamma}\|\mathbb{P}R\|+\e^2\lambda \|D(R)\|^2+\e^2C_{\lambda}(\|\mathbb{P}R\|^2+\e^{\frac{1}{2}}\t^{-2\gamma}))\\ \label{R24}
\!\!\!&\lesssim&\!\!\!\e^2\lambda\|D(R)\|^2+C_{\lambda}\e^{\frac{1}{4}}\t^{-\gamma}(\|\mathbb{P}R\|^2+1).
\end{eqnarray}
\item  For the second term of the energy equality, \eqref{PHC}, we start with the following integration by parts:
\begin{equation*}
-\int_{\cO}\Delta R\cdot \mathbb{P}R=2\int_{\cO}D(R)\cdot D(\mathbb{P}R)+2\int_{\partial\cO}(D(R)\cdot \mathbf{n})_{\tan}\cdot \mathbb{P}R .
\end{equation*}
Then, on the one hand, it follows from \eqref{PH2} that
\beq \label{R25}
\begin{split}
2\int_{\cO}D(R)\cdot D(\mathbb{P}R)=&2\|D(R)\|^2-2\int_{\cO}D(R)\cdot D((I-\mathbb{P})R)\\
\geq&\|D(R)\|^2-C\|D((I-\mathbb{P})R)\|^2\\
\geq&\|D(R)\|^2-C\e^{\frac{1}{4}}\t^{-\gamma},
\end{split}
\eeq
and, on the other hand, by using  boundary condition $\CN(R)=G$ on $\partial\cO$, one has
\begin{eqnarray*}
\nonumber\int_{\partial\cO}(D(R)\cdot \mathbf{n})_{\tan}\cdot \mathbb{P}R&=&\int_{\partial\cO}\bigl(G-(MR)_{\tan}\bigr)\cdot\mathbb{P}R\\
\nonumber&=&\int_{\cO}\dive\left(n(G-(MR)_{\tan})\cdot\mathbb{P}R\right),
\end{eqnarray*}
so that thanks to (\ref{PR}), (\ref{RH1}) and (\ref{HG}), for $\la>0,$ we get,  by applying Young's inequality, that
\begin{eqnarray}\label{R26}
 \vert  \int_{\partial\cO}(D(R)\cdot \mathbf{n})_{\tan}\cdot \mathbb{P}R  \vert \\
&\lesssim&\lambda\|D(R)\|^2+C_{\lambda}(\|\mathbb{P}R\|^2+\t^{-2\gamma}) .
\end{eqnarray}

\end{itemize}

By inserting the estimates, (\ref{R21}), (\ref{R22}), (\ref{R23}), (\ref{R24}) and (\ref{R26}), into \eqref{PHC}, we arrive at
\begin{equation*}
\frac{1}{2}\f{d}{dt}\|\mathbb{P}R(t)\|^2+\e\|D(R)\|^2\leq C\e\lambda\|D(R)\|^2+C_{\lambda}(\e+\t^{-\gamma})\|\mathbb{P}R\|^2+\e^{\frac{1}{4}}\t^{-\gamma}.
\end{equation*}
Choosing $\lambda$ small enough such that $C\lambda<\frac{1}{2}$ and note that (\ref{R0}) implies
\begin{equation}
\|\mathbb{P}R_0\|\leq\|R_0\|\lesssim \e^{-\frac{1}{4}},
\end{equation}
then we use Gronwall inequality to find that
\begin{equation*}
\|\mathbb{P}R(t)\|^2+\e\int_0^t\|D(R)\|^2dt'\leq C\e^{-\frac{1}{4}}\quad\mbox{ for }0\leq t\leq T^\e .
\end{equation*}
Together with (\ref{PH2}) and \eqref{RH1}, we thus conclude the proof of \eqref{el2}.
\end{proof}

\subsection{Tangential derivatives estimates} \label{tangen}

We now estimate the tangential derivatives of the remainder. Recall that the tangential derivatives $Z^{\alpha}$ are defined in \eqref{defZ} of Subsection \ref{sec-appr} and the conormal Sobolev  norm $\|\cdot\|_m$ is defined in \eqref{scn}. Let us start by estimating $\nabla R$ on the boundary.

\begin{lemma}\label{REDF}
{\sl Let $m\geq1. $ There holds
\begin{equation}\label{Rn}
\|\nabla R\|_{H^{m-1}(\partial\cO)}\lesssim \|R\|_{H^{m}(\partial\cO)}+\t^{-\gamma}.
\end{equation}
}
\end{lemma}
\begin{proof} Indeed
we only need to estimate $\|\partial_\mathbf{n} R\|_{H^{m-1}(\partial\cO)}$. On the one hand, we deduce from the
 boundary conditions: $\CN(R)=-G,$ $R\cdot \mathbf{n}=0$ on $\partial\cO,$
that  $$(\partial_\mathbf{n} R-\nabla \mathbf{n}\cdot R+2MR)_{\tan}=-2G\quad\mbox{ on} \ \partial\cO,$$ from which, and (\ref{HG}),
we infer
\begin{equation}\label{Rn1}
\|(\partial_\mathbf{n} R)_{\tan}\|_{H^{m-1}(\partial\cO)}\lesssim\|R\|_{H^m(\partial\cO)}+\|G\|_{H^{m-1}(\partial\cO)}\lesssim \|R\|_{H^m(\partial\cO)}+\t^{-\gamma}.
\end{equation}
On the other hand, $\dive R=-H$ gives us
\begin{equation}\label{div}
\partial_\mathbf{n} R\cdot \mathbf{n} +\sum_jc_jZ_jR=-H,
\end{equation}
for some smooth functions $c_j$, which depends only on vector field $w^j.$ Thus, by (\ref{HG}),
\begin{equation}\label{Rn2}
\|\partial_\mathbf{n} R\cdot \mathbf{n}\|_{H^{m-1}(\partial\cO)}\lesssim\|R\|_{H^m(\partial\cO)}+\|H\|_{H^{m-1}(\partial\cO)}\lesssim \|R\|_{H^m(\partial\cO)}+\t^{-\gamma}.
\end{equation}
Combining the estimate (\ref{Rn1}) with (\ref{Rn2}), we have proved the part of (\ref{Rn}) regarding
  $\|\partial_\mathbf{n} R\|_{H^{m-1}(\partial\cO)}$.
The other part of the estimate is straightforward.
\end{proof}
\begin{proposition}\label{S5prop3}
{\sl Let $1\leq m\leq p-3$ be an integer. Then there exists a constant $C_1>0$ such that for any $t\in [0,T^\e ]$,
\beq \label{Re}
\begin{split}
\f{d}{dt}\|R(t)\|_m^2+&C_1\e\|\nabla R\|^2_{m}\lesssim \e\|\nabla R\|^2_{m-1}+\sum_{|\alpha|\leq m}|\int_{\cO}Z^{\alpha}\nabla\pi\cdot  Z^{\alpha}R|+\e^{\frac{1}{4}}\t^{-\gamma}\\
&+\|R\|_{m}^2(\e+\t^{-\gamma})+\e^2(\|R\|_{1,\infty}\|\nabla R\|_{m-1}\|R\|_{m}+\|\nabla R\|_{L^{\infty}}\|R\|^2_m).
\end{split}\eeq
}
\end{proposition}

\begin{proof}
Let $1\leq \ell\leq p-3$ be an integer and $\alpha$ be a multi-index  with $|\alpha|=\ell$.
By applying $Z^{\alpha}$ to \eqref{R1}, we obtain
\beno
\begin{split}
\partial_tZ^{\alpha}R-\e\Delta Z^{\alpha}R+&u^{\e}\cdot\nabla Z^{\alpha}R+Z^{\alpha}(R\cdot\nabla u^{\e}_a)+Z^{\alpha}\nabla \pi\\
&\qquad=Z^{\alpha}F-\e[\Delta,Z^{\alpha}]R+[u^{\e}\cdot\nabla,Z^{\alpha}]R,
\end{split}\eeno
Taking $L^2$ inner product of the above equation with $ Z^{\alpha}R$ gives rise to
\beq \label{ZR}
\begin{split}
&\frac{1}{2}\f{d}{dt}\|Z^{\alpha}R(t)\|^2
-
\e \int_{\cO}\Delta Z^{\alpha}R\cdot Z^{\alpha}R
+\int_{\cO}(u^{\e}\cdot\nabla  Z^{\alpha}R)\cdot Z^{\alpha}R+\int_{\cO}Z^{\alpha}\nabla\pi\cdot  Z^{\alpha}R\\
&=
- \int_{\cO} Z^{\alpha}(R\cdot\nabla u^{\e}_a)  \cdot  Z^{\alpha}R
+ \int_{\cO} Z^{\alpha}F  \cdot  Z^{\alpha}R
 - \e \int_{\cO}  [\Delta,Z^{\alpha}]R  \cdot  Z^{\alpha}R
\\
&\quad+\int_{\cO}  [u^{\e}\cdot\nabla,Z^{\alpha}]R \cdot  Z^{\alpha}R.
\end{split}\eeq

In what follows, we shall handle term by term above in \eqref{ZR}.

\begin{itemize}

\item We start with estimating the second term in \eqref{ZR}, which relies on the following lemma:
\begin{lemma}\label{J1} Let $1\leq |\alpha|\leq m$. There exist constants $C_1,C>0$ such that
\begin{equation}
\label{ZR1}
-\int_{\cO}\Delta Z^{\alpha}R\cdot Z^{\alpha}R\geq C_1\|\nabla Z^\alpha  R\|^2-C\|R\|_m^2-C\t^{-2\gamma}.
\end{equation}
\end{lemma}

We postpone its proof to the end of this subsection.

\item For the third term of (\ref{ZR}),  since $\dive u^{\e}=\sigma^0 $ in $\cO,$  $u^\e \cdot \mathbf{n} = 0$ on $\partial \cO$, and
$\sigma^0 $ is supported on $[ 0,T],$
 we get, by  using integration by parts, that
\begin{equation}\label{ZR2}
\vert \int_{\cO}(u^{\e}\cdot\nabla  Z^{\alpha}R)\cdot Z^{\alpha}R \vert  \lesssim \t^{-\gamma}\|R\|_m^2.
\end{equation}
 .

\item
By using the Leibniz formula and   \eqref{Ra}, we find
\beq \label{ZR3}
  \begin{split}
\vert \int_{\cO} Z^{\alpha}(R\cdot \nabla u^{\e}_a)\cdot Z^{\alpha}R \vert \lesssim& \sum_{\alpha_1+\alpha_2=\alpha}|\int_{\cO}Z^{\alpha_1}R\cdot Z^{\alpha_2}\nabla u^{\e}_a\cdot Z^{\alpha}R|\\
\lesssim& \t^{-\gamma}\|R\|^2_m .
\end{split} \eeq

 \item (\ref{FH1}) ensures that
\begin{equation}\label{ZR4}
\vert\int_{\cO}Z^{\alpha}F\cdot Z^{\alpha}R \vert
\lesssim \e^{\frac{1}{4}}\t^{-\gamma}\bigl(\|R\|_m^2+1\bigr).
\end{equation}

\item Thanks to \eqref{Z3} and $Z^\alpha(R\cdot \bn)=0$ on $\p\cO,$ we get, by using integration by parts, that
\beno
\begin{split}
\vert\int_{\cO}\e[\Delta,Z^{\alpha}]R\cdot Z^{\alpha}R\vert\lesssim&|\int_{\cO}\sum_{|\beta|,|\gamma|\leq m-1}\e(c_{\beta}\nabla^2Z^{\beta}R+c_{\gamma}\nabla Z^{\gamma}R)\cdot Z^{\alpha}R|\\
\lesssim&\e(\|\nabla R\|_m+\|R\|_m )\|\nabla R\|_{m-1}+\e\|\nabla R\|_{H^{m-1}(\partial\cO)}\|R\|_{H^{m}(\partial\cO)}.
\end{split} \eeno
Whereas due to trace Theorem (see (87) of \cite{masmoudi} for instance) that
\beq\label{trace}
\|R\|_{H^m(\partial\cO)}^2
\lesssim \|R\|_m^2+\|R\|_m\|\nabla R\|_m,
\eeq
and Lemma \ref{REDF}, for any $\la>0,$ we infer
\beq \label{NZ}
\begin{split}
\|\nabla R\|_{H^{m-1}(\partial\cO)}\|R\|_{H^{m}(\partial\cO)} \lesssim&\|R\|_{H^m(\partial\cO)}^2+\t^{-2\gamma}\\
\lesssim&\|R\|_m^2+\|R\|_m\|\nabla R\|_m+\t^{-2\gamma}\\
\lesssim&\lambda\|\nabla R\|_m^2+C_{\lambda}\|R\|^2_m+\t^{-2\gamma}.
\end{split} \eeq

As a result, it comes out
\begin{eqnarray}\label{ZR5}
\hspace{1cm}
\e\vert\int_{\cO}[\Delta,Z^{\alpha}]R\cdot Z^{\alpha}R\vert \leq \lambda\e\|\nabla R\|_m^2+C_{\lambda}\e(\|\nabla R\|_{m-1}^2+\|R\|_m^2+\t^{-2\gamma}),
\end{eqnarray}

\item
For the last term, we use the decomposition \eqref{DEF-R} to get
\beno
\begin{split}
\int_{\cO}[u^{\e}\cdot\nabla,Z^{\alpha}]R\cdot  Z^{\alpha}R =&\int_{\cO}[u^0\cdot\nabla,Z^{\alpha}]R\cdot Z^{\alpha}R\\
&+\int_{\cO} [(u^{\e}_a-u^0)\cdot\nabla,Z^{\alpha}]R\cdot  Z^{\alpha}R+ \e^2 \int_{\cO}[R\cdot\nabla, Z^{\alpha}]R\cdot Z^{\alpha} R.
\end{split} \eeno
We write
\begin{equation}
\label{u0Z}u^0\cdot \nabla=\sum_j c_jZ_j+(u^0\cdot \mathbf{n})\partial_\mathbf{n}=\sum_j c_jZ_j+u^0_{\flat}Z_0,
\end{equation}
for some smooth functions $c_j.$

Thanks to \eqref{cm}, we can easily show by induction that $[Z_j,Z^{\alpha}],0\leq j\leq 5,$ is a tangential derivative of order $m$. Note that $u^0$ is supported in $[0,T]$, we have
\begin{equation} \label{psg1}
\vert \int_{\cO}[u^0\cdot\nabla,Z^{\alpha}]R\cdot Z^{\alpha}R \vert \lesssim \chi_{[0,T]}(t)\|R\|_{m}^2.
\end{equation}
While applying  the Leibniz formula yields
\begin{eqnarray*}
[(u^{\e}_a-u^0)\cdot\nabla,Z^{\alpha}]R&=&\sum_{\alpha_1+\alpha_2=\alpha,\alpha_1\neq0}c_{\alpha_1} Z^{\alpha_1}(u^{\e}_a-u^0)Z^{\alpha_2}\nabla R+(u^{\e}_a-u^0)[\nabla,Z^{\alpha}]R,
\end{eqnarray*}
for some smooth functions $c_{\alpha_1}$ depended on the vector field $\mathfrak{W}.$

It follows from \eqref{ua0} that
\begin{equation*}
\|[(u^{\e}_a-u^0)\cdot\nabla,Z^{\alpha}]R\|\lesssim\sqrt{\e}\t^{-\gamma}\|\nabla R\|_{m-1},
\end{equation*}
which implies
\begin{equation} \label{psg2}
\vert \int_{\cO} [(u^{\e}_a-u^0)\cdot\nabla,Z^{\alpha}]R\cdot  Z^{\alpha}R \vert \lesssim\sqrt{\e}\t^{-\gamma}\|\nabla R\|_{m-1}\|R\|_{m}\lesssim \e\|\nabla R\|_{m-1}^2+\t^{-2\gamma}\|R\|_m^2.
\end{equation}
Applying  the Leibniz formula once again gives
\begin{equation*}
[R\cdot\nabla, Z^{\alpha}]R=\sum_{\alpha_1+\alpha_2=\alpha,\alpha_1\neq 0}c_{\alpha_1}Z^{\alpha_1}R\cdot Z^{\alpha_2}\nabla R+R[\nabla,Z^{\alpha}]R.
\end{equation*}
Yet it follows from  generalized Sobolev-Gagliardo-Nirenberg-Morse inequality that
\begin{equation*}
\|[R\cdot\nabla, Z^{\alpha}]R\|
\lesssim\|R\|_{1,\infty}\|\nabla R\|_{m-1}+\|\nabla R\|_{L^{\infty}}\|R\|_m,
\end{equation*}
so that we infer
\begin{equation}\label{psg3}
\vert \int_{\cO}[ R\cdot\nabla,Z^{\alpha}]R\cdot  Z^{\alpha}R \vert \lesssim \|R\|_{1,\infty}\|\nabla R\|_{m-1}\|R\|_{m}+\|\nabla R\|_{L^{\infty}}\|R\|^2_m.
\end{equation}
Combining \eqref{psg1}, \eqref{psg2} with \eqref{psg3},
we arrive at
\beq  \label{ZR6}
\begin{split}
\vert  \int_{\cO}[u^{\e}\cdot\nabla,Z^{\alpha}]R\cdot  Z^{\alpha}R \vert  \lesssim  & \e\|\nabla R\|_{m-1}^2+\t^{-2\gamma}\|R\|_m^2
\\  &+ \e^2(\|R\|_{1,\infty}\|\nabla R\|_{m-1}\|R\|_{m}+\|\nabla R\|_{L^{\infty}}\|R\|^2_m ).
\end{split} \eeq

\end{itemize}

By  inserting the estimates, \eqref{ZR1}, \eqref{ZR2}, \eqref{ZR3}, \eqref{ZR4},  \eqref{ZR5}
and  \eqref{ZR6},
 into (\ref{ZR}),  and  then by summing  up \eqref{el2} with the resulting inequality over all the multi-indices
$\alpha$    with $1\leq |\alpha|\leq m$, finally choosing $\lambda$ to be sufficiently small, we arrive to  (\ref{Re}).
\end{proof}

Let us now present the proof of  Lemma \ref{J1}.

\begin{proof}[Proof of  Lemma \ref{J1} ]
 We first get, by using integration by parts, that
\begin{equation}\label{D1}
-\int_{\cO}\Delta Z^{\alpha}R\cdot Z^{\alpha}R=2\int_{\cO}|D(Z^{\alpha}R)|^2+2\int_{\partial\cO}(D(Z^{\alpha}R)\cdot \mathbf{n})_{\tan}\cdot Z^{\alpha}R.
\end{equation}
It follows from  Korn's inequality that
\begin{equation}\label{D2}
\|D(Z^{\alpha}R)\|^2\geq C_1\|\nabla Z^{\alpha}R\|^2-C_2\|Z^{\alpha}R\|^2.
\end{equation}
As $\CN(R)=G$ on $\partial\cO$, we have
\begin{equation*}
(D(Z^{\alpha}R)\cdot \mathbf{n})_{\tan}=-(MZ^{\alpha}R)_{\tan}+Z^{\alpha}G+[\CN,Z^{\alpha}]R,
\end{equation*}
so that there holds
\beq \label{D3}
  \begin{split}
\int_{\partial\cO}(D(Z^{\alpha}R)\cdot \mathbf{n})_{\tan}\cdot Z^{\alpha}R =&\int_{\partial\cO}(Z^{\alpha}G-(MZ^{\alpha}R)_{\tan}) \cdot Z^{\alpha} R\\
& +\int_{\partial\cO} [\CN,Z^{\alpha}]R \cdot Z^{\alpha}R.
\end{split} \eeq
We are going to estimate each term of the right hand side of \eqref{D3}.

On the one hand, by virtue of (\ref{HG}) and   for any  $\lambda>0,$  we get, by
applying Young's inequality, that
\begin{eqnarray*}\label{D4}
\hspace{1cm}\bigl|\int_{\partial\cO}(Z^{\alpha}G-(MZ^{\alpha}R)_{\tan})\cdot Z^{\alpha}R\bigr|&=
&\bigl|\int_{\cO}\dive\bigl(n(Z^{\alpha}G-(MZ^{\alpha}R)_{\tan})\cdot Z^{\alpha}R\bigr)\bigr|\\
\nonumber&\lesssim &\|Z^{\alpha}G\|_{H^1}\|Z^{\alpha}R\|+(\|Z^{\alpha}G\|+\|Z^{\alpha}R\|)\|Z^{\alpha}R\|_{H^1}\\
\nonumber&\lesssim &\lambda \|\nabla Z^{\alpha}R\|^2+C_{\lambda}(\|Z^{\alpha}G\|_{H^1}^2+\|Z^{\alpha}R\|^2)\\
\nonumber&\lesssim &\lambda \|\nabla Z^{\alpha}R\|^2+C_{\lambda}(\t^{-2\gamma}+\|Z^{\alpha}R\|^2).
\end{eqnarray*}

On other hand, we deduce from \eqref{NZ} that
\begin{eqnarray*}
\bigl|\int_{\partial\cO}[N,Z^{\alpha}]R\cdot Z^{\alpha}R\bigr|&\lesssim&\|\nabla R\|_{H^{m-1}(\partial\cO)}\|R\|_{H^{m}(\partial\cO)}\\
\nonumber&\leq &\lambda\|\nabla R\|_m^2+C_{\lambda}\bigl(\|R\|^2_m+\t^{-2\gamma}\bigr).
\end{eqnarray*}

By substituting the above inequalities into \eqref{D3}, we achieve
\beq \label{Dm1}
\bigl|\int_{\partial\cO}(D(Z^{\alpha}R)\cdot \mathbf{n})_{\tan}\cdot Z^{\alpha}R\bigr|\leq 2\lambda\|\nabla R\|_m^2+2C_{\lambda}\bigl(\|R\|^2_m+\t^{-2\gamma}\bigr).
\eeq

By inserting (\ref{D2}) and (\ref{Dm1}) into \eqref{D1} and choosing  $\lambda$ to be sufficiently small, we arrive at
\eqref{ZR1}.
\end{proof}

\subsection{An appropriate substitute to the vorticity}
\label{sisi-eta}

We observe that the right hand side of (\ref{Re}) involves
 $\|\nabla R\|_{m-1}$ and $ \|\sqrt{\e}\nabla R\|_{\infty}$, so that we need to estimate at least one normal derivative of $R$.
 We define
\begin{equation}
\label{defeta}
\eta :=\sqrt{\e}(\CN(R)+G)\chi(x),
\end{equation}
where $\chi$ is a cut-off function defined in Section 2. From the definition of $\eta$, we know that $\eta=0$ on the boundary $\partial \cO$.
Observe that this property is not satisfied by the vorticity $\curl R$; this is indeed the reason why we would  rather use $\eta$ following \cite{masmoudi} than $\curl R$.

\begin{lemma}\label{EQUIV}
{\sl Let $m\geq1.$
The following equivalences hold true:
\begin{align}\label{eta}
\|\eta\|_{m-1}+\|R\|_{m}+\sqrt{\e}\t^{-\gamma}&\approx \|\sqrt{\e}\nabla R\|_{m-1}+\|R\|_{m}+\sqrt{\e}\t^{-\gamma},\\
\label{etaL}\|\eta\|_{L^{\infty}}+\|R\|_{1,\infty}+\sqrt{\e}\t^{-\gamma}&\approx \|\sqrt{\e}\nabla R\|_{\infty}+\|R\|_{1,\infty}+\sqrt{\e}\t^{-\gamma}.
\end{align}}
\end{lemma}
\begin{proof}
 Let us focus on the proof of \eqref{eta}.
We first deduce from the definitions  \eqref{def-dep} and \eqref{defeta}, and the estimate of $G$ in \eqref{HG} that
\beno
\|\eta\|_{m-1} \lesssim \|\sqrt{\e} \nabla R\|_{m-1} + \|R\|_m +\sqrt{\e} \t^{-\gamma},
\eeno
which implies
\beq \label{violon1}
\|\eta\|_{m-1}+\|R\|_{m}+\sqrt{\e}\t^{-\gamma}\lesssim \|\sqrt{\e}\nabla R\|_{m-1}+\|R\|_{m}+\sqrt{\e}\t^{-\gamma}.
\eeq

To prove the other side of the inequality \eqref{violon1}, we introduce
\beq \label{DefP}
\Pi f :=f_{\mbox{tan}}.
\eeq
Then notice that
$$ D(R)n=\f12\left(\p_\bn R+\na R\cdot n\right)_{\mbox{tan}} \andf \left(\na R_jn_j\right)_{\mbox{tan}}=(\na R_j)_{\mbox{tan}} n_j,
$$
we have
\begin{align*}
\sqrt{\e}\|\chi\Pi\partial_\mathbf{n}R\|_{m-1} &\lesssim \|\eta\|_{m-1}+\|R\|_m +\sqrt{\e}\t^{-\gamma} .
\end{align*}

While by definitions of $\chi$ and of the norm $\|\cdot\|_{m}$, one has
\begin{equation*} \label{violon3}
\|(1-\chi)\Pi\partial_\mathbf{n} R\|_{m-1} \lesssim \|R\|_m.
\end{equation*}
And it follows from (\ref{div}) and (\ref{FH1}) that
\begin{equation*}  \label{violon4}
\|\partial_\mathbf{n} R\cdot \mathbf{n}\|_{m-1}\lesssim
\|R\|_m+\e^{\frac{1}{4}}\t^{-\gamma}.
\end{equation*}
As a consequence, we obtain
\beq \label{violon2}
\begin{split}
\|\p_\bn R\|_{m-1}\lesssim & \|\chi\Pi\partial_\mathbf{n}R\|_{m-1}+\|(1-\chi)\Pi\partial_\mathbf{n} R\|_{m-1}
+\|\partial_\mathbf{n} R\cdot \mathbf{n}\|_{m-1}\\
\lesssim &\|\eta\|_{m-1}+\|R\|_m +\sqrt{\e}\t^{-\gamma}.
\end{split}
\eeq
\eqref{violon2} shows that the other side of the inequality \eqref{violon1} holds,
which leads to \eqref{eta}.
The equivalence \eqref{etaL} can be proved  along the same line.
 \end{proof}

By virtue of (\ref{eta}) and (\ref{etaL}), we can rewrite (\ref{Re}) as
\beq \label{Re2}
\begin{split}
\f{d}{dt}\|R(t)\|_m^2+C_1\e\|\nabla R\|_m^2 \lesssim&\e\|\nabla R\|_{m-1}^2+\sum_{|\alpha|\leq m}|\int_{\cO}Z^{\alpha}\nabla \pi\cdot Z^{\alpha}R|+\e^{\frac{1}{4}}\t^{-\gamma}\\
&+(\|\eta\|_{m-1}^2+\|R\|_m^2)\bigl(\e+\t^{-\gamma}+\e^2(\|\eta\|_{L^{\infty}}^2+\|R\|_{1,\infty}^2)\bigr).
\end{split} \eeq

\subsection{Time evolution of the auxiliary quantity}
\label{time-eta}

Let us now estimate the time evolution of $\|\eta(t)\|_{m-1}$, which appears in the right hand side of \eqref{Re2}.
\begin{proposition}
\label{prop-etae}
{\sl Let $1\leq m\leq p-3$. Then there exist a constant $C_1>0$ such that for any $t\in [0,T^\e ]$,
\beq\label{etae}
  \begin{split}
  \f{d}{dt}\|\eta&(t)\|_{m-1}^2+C_1\e\|\nabla \eta\|_{m-1}^2\\
 \lesssim &
\e\|\nabla \eta\|_{m-2}^2+\e^{\frac{1}{4}}\t^{-\gamma}+
\sum_{|\beta|\leq m-1}\sqrt{\e}\bigl|\int_{\cO}Z^{\beta}\left(\chi \CN(\nabla\pi)\right)\cdot Z^{\beta}\eta\bigr|\\
&+\bigl(\e+\t^{-\gamma}+\e^2(\|\eta\|_{L^{\infty}}^2+\|R\|_{1,\infty}^2)\bigr)\left(\|\eta\|_{m-1}^2+\|R\|_m^2\right).
\end{split} \eeq
Here the term $\e\|\nabla \eta\|_{m-2}^2$ does not appear on the right-hand side of \eqref{etae} when $m=1.$
}
\end{proposition}

\begin{proof}
In view of (\ref{R}), $\eta$ satisfies
\beno
\begin{split}
\partial_t \eta-\e\Delta\eta+u^{\e}\cdot\nabla\eta=&-\sqrt{\e}\chi\CN(F+\nabla\pi+R\cdot\nabla u^{\e}_a)+\sqrt{\e}\left(\partial_t-\e\Delta+u^{\e}\cdot\nabla\right)(\chi G)\hspace{0.4cm}\\
&-\e^{\f32}[\Delta,\chi \CN]R+\sqrt{\e}[u^0\cdot\nabla, \chi \CN]R+\sqrt{\e}[(u^{\e}-u^0)\cdot\nabla,\chi \CN]R.
\end{split} \eeno
Applying $Z^{\beta}$ with $|\beta|=m-1$ to the above equation yields
  \begin{eqnarray*}
\hspace{1cm}\partial_tZ^{\beta}\eta-\e\Delta Z^{\beta}\eta+u^{\e}\cdot\nabla Z^{\beta}\eta\!\!\!&=&\!\!\!-\sqrt{\e}Z^{\beta}\left(\chi\CN(F+\nabla\pi+R\cdot\nabla u^{\e}_a)\right)\\
\nonumber &&\!\!\!+\sqrt{\e}Z^{\beta}(\partial_t-\e\Delta+u^{\e}\cdot\nabla)(\chi G)-\e^{\f32}[\Delta,Z^{\beta}(\chi\CN)]R\\
\nonumber &&\!\!\!+\sqrt{\e}[u^0\cdot\nabla,Z^{\beta}(\chi \CN)]R+\sqrt{\e}[(u^{\e}-u^0)\cdot\nabla,Z^{\beta}(\chi \CN)]R.
\end{eqnarray*}
Note that $\eta=0$ on $\p\cO$ and $Z^{\beta}$ is tangential derivative, we have  $Z^{\beta}\eta=0$ on $\p\cO$.
Then we get, by taking $L^2$ inner product of the above equation with $Z^{\beta}\eta,$ that
\begin{equation}\label{etaa}
\frac{1}{2}\f{d}{dt}\|Z^{\beta}\eta(t)\|^2+2\e\|D(Z^{\beta} \eta)\|^2\lesssim \sum_{i=1}^8 |I_i|,
\end{equation}
where
\begin{eqnarray*}
&\displaystyle I_1:=\int_{\cO}u^{\e}\cdot\nabla Z^{\beta}\eta\cdot Z^{\beta}\eta, \quad
&I_2:=\sqrt{\e}\int_{\cO}Z^{\beta}(\chi \CN(F))\cdot Z^{\beta}\eta,\\
&\displaystyle I_3:=\sqrt{\e}\int_{\cO}Z^{\beta}(\chi \CN(\nabla\pi))\cdot Z^{\beta}\eta, \quad
&I_4:=\sqrt{\e}\int_{\cO}Z^{\beta}\left(\chi\CN(R\cdot\nabla u^{\e}_a)\right)\cdot Z^{\beta}\eta,\\
&\displaystyle I_5:=\sqrt{\e}\int_{\cO}(\partial_t-\e\Delta+u^{\e}\cdot\nabla)(Z^{\beta}(\chi G))\cdot Z^{\beta}\eta,\quad
&I_6:=\e^{\f32}\int_{\cO}[\Delta,Z^{\beta}(\chi \CN)]R\cdot Z^{\beta}\eta,\\
&\displaystyle I_7:=\sqrt{\e}\int_{\cO}[u^0\cdot\nabla,Z^{\beta}(\chi \CN)]R\cdot Z^{\beta}\eta,\quad
&I_8:=\sqrt{\e}\int_{\cO}[(u^{\e}-u^0)\cdot\nabla,Z^{\beta}(\chi \CN)]R\cdot Z^{\beta}\eta.\\
\end{eqnarray*}

First, regarding the second term in the left hand side of \eqref{etaa}, we observe from Korn's inequality  that
\begin{equation}
  \label{KORN}
\|D(Z^{\beta}\eta)\|\geq C_1\|\nabla Z^{\beta} \eta\|-C_2\|Z^{\beta}\eta\|.
\end{equation}

Let us now handle term by term in the right-hand side of \eqref{etaa}.

\smallskip

\no $\bullet$ \underline{Estimate of $I_1.$}\vspace{0.2cm}

Since $u^\e$ satisfies \eqref{NSA2}, \eqref{NSA3}, and $\sigma^{\e}=\sigma^{0}$ is supported in $[0,T]$, by using integration by parts, we find
\begin{equation}
\label{I1e}
\vert I_1 \vert \lesssim\t^{-\gamma}\|\eta\|^2_{m-1}.
\end{equation}

\no $\bullet$ \underline{Estimate of $I_2.$}\vspace{0.2cm}

 By virtue of (\ref{FH1}), we get, by applying the Cauchy-Schwarz inequality, that
\begin{equation}
\label{I2e} \vert I_2 \vert\lesssim \e^{\frac{1}{4}}\t^{-\gamma}\bigl(1+\|\eta\|_{m-1}^2\bigr).
\end{equation}

\smallskip

\no $\bullet$ \underline{Estimate of $I_3.$}\vspace{0.2cm}

 We simply estimate $I_3$ by the third term on the right-hand side of
\eqref{etae}. We remark that we do not try to get rid of the pressure at this step. Indeed this  delicate issue will be postponed to Subsection
\ref{sub-pressure}.

\smallskip

\no $\bullet$ \underline{Estimate of $I_4.$}\vspace{0.2cm}

 Recalling \eqref{DefP} and in view of \eqref{def-dep}, we write
\begin{equation*}
\sqrt{\e}\chi\CN(R\cdot\nabla u^{\e}_a)=\sqrt{\e}\chi\Pi\Bigl(\frac{1}{2}(\partial_\mathbf{n} (R\cdot\nabla u^{\e}_a)+\nabla (R\cdot\nabla u^{\e}_a)\cdot \mathbf{n})+M(R\cdot\nabla u_{a}^{\e})\Bigr).
\end{equation*}
Since $M$ is a smooth matrix-valued function and $m\leq p-3$, we get,
by applying  Leibniz formula and \eqref{Ra}, that
\begin{equation*}
\|\sqrt{\e}\chi\CN(R\cdot\nabla u^{\e}_a)\|_{m-1}\lesssim \sqrt{\e}\t^{-\gamma}\|\nabla R\|_{m-1}+\t^{-\gamma}\|R\|_{m-1},
\end{equation*}
which together with (\ref{eta}) ensures that
\begin{equation}
\label{I4e}\vert I_4 \vert\lesssim \t^{-\gamma}\bigl(\|R\|_{m}^2+\|\eta\|_{m-1}^2\bigr)+\sqrt{\e}\t^{-3\gamma}.
\end{equation}

\no $\bullet$ \underline{Estimate of $I_5.$}\vspace{0.2cm}

We split it into two terms
\beno
\begin{split}
I_5=I_{51}+I_{52} \with
&I_{51} :=\sqrt{\e}\int_{\cO}(\partial_t-\e\Delta+u^{\e}_a\cdot\nabla)\left(Z^{\beta}(\chi G)\right)\cdot Z^{\beta}\eta ,\\
&I_{52} :=\e^{\f52}\int_{\cO} R\cdot\nabla (Z^{\beta}(\chi G))\cdot  Z^{\beta}\eta.
\end{split}
\eeno
Thanks to (\ref{HG}) and (\ref{Ra}), $\chi$ is a smooth function, for $m\leq p-3$, we infer
\beno
 \begin{split}
\vert I_{51} \vert \lesssim & \sqrt{\e}\bigl(\|\partial_t G\|_{m-1}+\e\|\nabla^2G\|_{m-1}+\|u^{\e}_a\|_{L^{\infty}(\cO)}\|\nabla G\|_{m-1}\bigr)\|\eta\|_{m-1}  \\
\lesssim&\sqrt{\e}\t^{-\gamma}\|\eta\|_{m-1},
\end{split} \eeno
and
\begin{gather*}
\label{I52e}|I_{52}| \lesssim \e^{\frac{5}{2}}\|R\|\|\nabla G\|_{m-1}\|\eta\|_{m-1}\lesssim \e^{\frac{5}{2}}\t^{-\gamma}\|R\|_{m}\|\eta\|_{m-1},
\end{gather*}
so that we achieve
\beq \label{I51e}
\vert I_{51}\lesssim \sqrt{\e}\t^{-\gamma}\bigl(1+\e^2\|R\|_{m}\bigr)\|\eta\|_{m-1}.
\eeq

\no $\bullet$ \underline{Estimate of $I_6.$}\vspace{0.2cm}

In view of  \eqref{R2}, we decompose
 $I_6=I_{61}+I_{62}$ with
 \begin{equation*}
I_{61} := \e
\int_{\cO}[\Delta,Z^{\beta}](\eta-\sqrt{\e}\chi G)\cdot Z^{\beta}\eta \   \text{ and } \
I_{62} :=\e^{\f32}\int_{\cO} Z^{\beta}[\Delta, \chi \CN]R\cdot Z^{\beta}\eta.
\end{equation*}
On the one hand, thanks to (\ref{Z3}), we write
 \begin{eqnarray*}
\int_{\cO}[\e\Delta,Z^{\beta}](\eta-\sqrt{\e}\chi G)\cdot Z^{\beta}\eta&=&\e\int_{\cO}\sum_{|\beta_1|,|\beta_2|<m-1}\bigl(c_{\beta_1}\nabla^2 Z^{\beta_1}\eta+c_{\beta_2}\nabla Z^{\beta_2}\eta\bigr)\cdot Z^{\beta}\eta\\
&&-\e^{\f32}\int_{\cO}[\Delta,Z^{\beta}](\chi G)\cdot Z^{\beta}\eta ,
\end{eqnarray*}
where $c_{\beta_1},c_{\beta_2}$ are smooth functions depend only on vector field $\mathfrak{W}$.
Due to $Z^{\beta}\eta=0$ on $\p\cO,$ by using integration by parts and $\eqref{HG}$, we infer
\beq \label{I61e}
  \begin{split}
\vert I_{61}\vert\lesssim& \e\|\nabla \eta\|_{m-1}\|\nabla\eta\|_{m-2}+\e\|\nabla\eta\|_{m-2}\|\eta\|_{m-1}+\e^{\frac{3}{2}}\t^{-\gamma}\|\eta\|_{m-1}\\
\leq&\lambda\e\|\nabla\eta\|_{m-1}^2+C_{\lambda}\e(\|\nabla\eta\|_{m-2}^2+\|\eta\|_{m-1}^2)+C\e^{\frac{3}{2}}\t^{-\gamma},
\end{split} \eeq
where $\lambda>0$ is a small constant.

On the other hand, we write
  \begin{eqnarray*}
[\Delta, \chi \CN]R&=&\Delta\chi\CN(R)+2\sqrt{\e}\nabla\chi:\nabla\CN(R)\\
&&+\chi(\Delta\Pi)\bigl(D(R)\cdot \mathbf{n}+MR\bigr)+2\sqrt{\e}\chi(\nabla\Pi):\nabla\bigl(D(R)\cdot \mathbf{n}+MR\bigr)\\
&&+\chi\Pi \bigl(D(R)\cdot \Delta \bn+2\nabla D(R):\nabla \mathbf{n}+(\Delta M) R+2\nabla M:\nabla R\bigr).
\end{eqnarray*}
Corresponding to  the second and the forth term above, we  use integration by parts in $I_{62}.$ Then  we  deduce from (\ref{eta})  that
\beq\label{I62e}
  \begin{split}
\vert I_{62} \vert \lesssim& \e^{\frac{3}{2}}\|\nabla R\|_{m-1}\bigl(\|\nabla\eta\|_{m-1}+\|\eta\|_{m-1}\bigr)\\
\leq &\lambda\e\|\nabla\eta\|_{m-1}^2+C_{\lambda}\e\bigl(\|\eta\|_{m-1}^2+\|R\|_{m}^2\bigr)+C_{\lambda}\e^2\t^{-2\gamma}.
\end{split} \eeq

\no $\bullet$ \underline{Estimate of $I_7.$}\vspace{0.2cm}

 We write
\begin{equation*}
[u^0\cdot\nabla , Z^{\beta}(\chi \CN)]R=[u^0\cdot\nabla, Z^{\beta}](\chi\CN(R))+Z^{\beta}[u^0\cdot\nabla,\chi \CN]R.
\end{equation*}
It follows from \eqref{u0Z} that $u^0\cdot \nabla $ is a tangential derivative. So that thanks to the  observation \eqref{cm},
 we find that $[u^0\cdot \nabla,Z^{\beta}]$ is an operator of linear combination of  tangential derivatives of order $m-1.$
Then due to the  fact that $u^0$ is supported in $[0,T]$, we infer
\begin{equation*}
\sqrt{\e}\|[u^0\cdot\nabla , Z^{\beta}(\chi \CN)]R\|\lesssim \sqrt{\e} \chi_{[0,T]}(t)\|\nabla R\|_{m-1}\lesssim \t^{-\gamma}\|\nabla R\|_{m-1},
\end{equation*}
which together with \eqref{eta} implies
\begin{equation} \label{I7e}
\vert I_7 \vert\lesssim \t^{-\gamma}\bigl(\|R\|_m^2+\|\eta\|_{m-1}^2\bigr)+\e\t^{-3\gamma}.
\end{equation}

\no $\bullet$ \underline{Estimate of $I_8.$}\vspace{0.2cm}

We first decompose  $I_8$ as
$$I_8=\sum_{i=1}^6 I_{8i},$$
 with
  \begin{eqnarray*}
&\displaystyle I_{81}:=\int_{\cO}[(u^{\e}_a-u^0)\cdot\nabla,Z^{\beta}]\eta\cdot Z^{\beta}\eta,
&I_{82}:=-\sqrt{\e}\int_{\cO}[(u^{\e}_a-u^0)\cdot\nabla,Z^{\beta}](\chi G)\cdot Z^{\beta}\eta\\
&\displaystyle I_{83}:=\sqrt{\e}\int_{\cO}Z^{\beta}[(u^{\e}_a-u^0)\cdot\nabla,\chi \CN]R\cdot Z^{\beta}\eta
&I_{84}:=\e^2\int_{\cO}[ R\cdot\nabla,Z^{\beta}]\eta\cdot Z^{\beta}\eta\\
&\displaystyle I_{85}:=-\e^{\f52}\int_{\cO}[ R\cdot\nabla,Z^{\beta}](\chi G)\cdot Z^{\beta}\eta
&I_{86}:=\e^{\f52}\int_{\cO} Z^{\beta}[R\cdot\nabla,\chi \CN]R\cdot Z^{\beta}\eta
\end{eqnarray*}

Next we deal with all the terms above.

\begin{itemize}

\item \textit{Estimate of $I_{81}$}.
In view of (\ref{Z1}), we write
\begin{equation*}
[(u^{\e}_a-u^0)\cdot\nabla,Z^{\beta}]\eta=\sum_{\beta_1+\beta_2=\beta,\beta_1\neq 0}c_{\beta_1}Z^{\beta_1}(u^{\e}_a-u^0)\cdot Z^{\beta_2}\nabla \eta+(u^{\e}_a-u^0)\cdot[\nabla,Z^{\beta}]\eta,
\end{equation*}
where $c_{\beta_1},c_{\beta_2}$ are smooth functions depend only on vector filed $\mathfrak{W}$. Then thanks to \eqref{ua0}, we infer
\begin{equation}
\label{I81}
\vert I_{81} \vert \lesssim \sqrt{\e}\t^{-\gamma}\|\nabla\eta\|_{m-2}\|\eta\|_{m-1}\lesssim\e\|\nabla\eta\|_{m-2}^2+\t^{-2\gamma}\|\eta\|_{m-1}^2.
\end{equation}
\item  \textit{Estimate of $I_{82}$}.
It follows from  (\ref{HG}) and (\ref{ua0}) that
\begin{equation}
\label{I82}
\vert I_{82} \vert \lesssim\sqrt{\e}\|u^{\e}_a-u^0\|_{m-1,\infty}\|G\|_{H^{m-1}}\|\eta\|_{m-1}\lesssim\e^{\frac{3}{4}}\t^{-2\gamma}\|\eta\|_{m-1}.
\end{equation}
 \item \textit{Estimate of $I_{83}$}.
In view of \eqref{def-dep}, we write
  \begin{eqnarray*}
[(u^{\e}_a-u^0)\cdot\nabla,\chi \CN]R=[(u^{\e}_a-u^0)\cdot\nabla,\chi \Pi]\bigl(D(R)\cdot \mathbf{n}+MR\bigr)\hspace{2cm}\\
+\chi \Pi\big((u^{\e}_a-u^0)\cdot\nabla(D(R)\cdot \mathbf{n})-D((u^{\e}_a-u^0)\cdot\nabla R)\cdot \mathbf{n})+((u^{\e}_a-u^0)\cdot\nabla M)R\big).
\end{eqnarray*}
Notice that  the second order derivatives of $R$ vanish on  the right hand side above, we deduce that
\begin{eqnarray*}
\vert I_{83} \vert \lesssim \sqrt{\e}\bigl(\|u^{\e}_a-u^0\|_{m-1,\infty}+\|\nabla(u^{\e}_a-u^0)\|_{m-1,\infty}\bigr)\|\nabla R\|_{m-1}\|\eta\|_{m-1},
\end{eqnarray*}
which together with (\ref{ua0}) and (\ref{eta}) ensures that
\begin{eqnarray}
\label{I83} \vert I_{83} \vert\lesssim \sqrt{\e}\t^{-\gamma}\|\nabla R\|_{m-1}\|\eta\|_{m-1}\lesssim\t^{-\gamma}(\|\eta\|_{m-1}^2+\|R\|_{m}^2)+\e\t^{-3\gamma}.
\end{eqnarray}
\item \textit{Estimate of $I_{84}$}.
In view of \eqref{Z1}, we write
  \begin{eqnarray*}
I_{84}=\e^2\int_{\cO}\Bigl(\sum_{\beta_1+\beta_2=\beta,\beta_1\neq 0}c_{\beta_1}Z^{\beta_1}R\cdot Z^{\beta_2}\nabla\eta+R\cdot[\nabla,Z^{\beta}]\eta\Bigr)\cdot Z^{\beta}\eta.
\end{eqnarray*}
We remark that if we use  directly use the generalized Sobolev-Gagliardo-Nirenberg-Morse inequality above,  there  appears the term, $\|\nabla\eta\|_{L^{\infty}},$ which we do not have the estimate. To overcome this difficulty, we use  integrations by parts to transfer the $\nabla$ on terms like $Z^{\beta_2}\nabla\eta$ into other terms. Notice that $Z^{\beta}\eta=0$ on $\partial\cO$, no boundary term appears
during this process. Then by applying the generalized Sobolev-Gagliardo-Nirenberg-Morse inequality, we find
\beno
\begin{split}
\vert I_{84} \vert \lesssim& \e^2\bigl(\|ZR\|_{L^{\infty}}\|\eta\|_{m-2}+\|R\|_{m-1}\|\eta\|_{L^{\infty}}\bigr)\|\nabla\eta\|_{m-1}\\
&+\e^2\bigl(\|\nabla R\|_{L^{\infty}}\|\eta\|_{m-1}+\|\nabla R\|_{m-1}\|\eta\|_{L^{\infty}}\bigr)\|\eta\|_{m-1}\\
&+\e^2\|R\|_{L^{\infty}}\|\nabla\eta\|_{m-1}\|\eta\|_{m-1},
\end{split} \eeno
from which, (\ref{eta}) and (\ref{etaL}), we infer
\begin{gather}
\label{I84}|I_{84}|\leq \lambda \e\|\nabla\eta\|_{m-1}^2+ C_{\lambda}\e(\|\eta\|_{m-1}^2+\|R\|_{m}^2)(\e\|\eta\|_{L^{\infty}}^2+\e\|R\|_{1,\infty}^2+1)+C\e^2\t^{-\gamma}.
\end{gather}
\item \textit{Estimate of $I_{85}$}.
Along the same line to the estimate of $I_{84},$ we write
  \begin{eqnarray*}
[R\cdot\nabla,Z^{\beta}](\chi G)=\sum_{\beta_1+\beta_2=\beta,\beta_1\neq 0}c_{\beta_1}Z^{\beta_1}R\cdot Z^{\beta_2}\nabla(\chi G)\cdot Z^{\beta}\eta+R\cdot[\nabla,Z^{\beta}](\chi G)\cdot Z^{\beta}\eta,
\end{eqnarray*}
 from which, \eqref{HG} and $m\leq p-3,$ we infer
\begin{equation}\label{I85}
\vert I_{85} \vert \lesssim \e^{\frac{5}{2}}\t^{-\gamma}\|R\|_{m-1}\|\eta\|_{m-1}.
\end{equation}
\item \textit{Estimate of $I_{86}$}.
In view of  \eqref{def-dep}, we write
\beno
\begin{split}
[R\cdot\nabla,\chi \CN]R=&[R\cdot\nabla,\chi \Pi](D(R)\cdot \mathbf{n}+MR)\\
&+\chi \Pi\big(R\cdot\nabla (D(R)\cdot \mathbf{n})-D(R\cdot\nabla R)\cdot \mathbf{n}+(R\cdot\nabla M)R\big).
\end{split}\eeno
Notice that  the second order derivatives of $R$ vanish on the right-hand side above, we deduce from the generalized Sobolev-Gagliardo-Nirenberg-Morse inequality, that
\begin{eqnarray*}
\e^{\f52}\|[R\cdot\nabla,\chi \CN]R\|_{m-1}\!\!\!&\lesssim&\!\!\!\e^{\frac{5}{2}}\bigl(\|\nabla R\|_{L^{\infty}}+\|R\|_{L^{\infty}}\bigr)
\bigl(\|\nabla R\|_{m-1}+\|R\|_{m-1}\bigr),
\end{eqnarray*}
which together with (\ref{eta}) and (\ref{etaL}) ensures that
\begin{gather}
\label{I86} |I_{86}| \lesssim \e\bigl(\|\eta\|_{m-1}^2+\|R\|_{m}^2\bigr)\bigl(1+\e\|\eta\|_{L^{\infty}}^2+\e\|R\|_{1,\infty}^2\bigr)+\e^2\t^{-\gamma}.
\end{gather}
\end{itemize}

By summing up the estimates, (\ref{I81}-\ref{I86}), we arrive at
\beq \label{I8Es}
\begin{split}
|I_8|\leq  &\la \e \|\na\eta\|_{m-1}^2+ C\bigl(\e\|\nabla \eta\|_{m-2}^2+\e^{\frac{1}{4}}\t^{-\gamma}\bigr)\\
&+C_\la\bigl(\e+\t^{-\gamma}+\e^2(\|\eta\|_{L^{\infty}}^2+\|R\|_{1,\infty}^2)\bigr)\left(\|\eta\|_{m-1}^2+\|R\|_m^2\right).
\end{split}
\eeq

By inserting the estimates,
\eqref{I1e}, \eqref{I2e}, \eqref{I4e}, \eqref{I51e}, \eqref{I52e},  \eqref{I61e}, \eqref{I62e}, \eqref{I7e} and \eqref{I8Es},
 into (\ref{etaa})
 and summing  over the resulting inequalities with the multi-indices  $\alpha$ with $|\alpha|\leq m,$ and finally  choosing  $\lambda$
to be sufficiently small, we obtain
 (\ref{etae}). This ends the proof of Proposition \ref{prop-etae}.
\end{proof}

By summing  up (\ref{Re2}) and (\ref{etae}), we achieve
\beq\label{name-me}
\begin{split}
\f{d}{dt}\bigl(&\|R(t)\|_m^2+\|\eta(t)\|_{m-1}^2\bigr)+\e\bigl(\|\nabla R\|_m^2+\|\na\eta\|_{m-1}^2\bigr)\\
\lesssim &\e\bigl(\|\nabla R\|_{m-1}^2+\|\na\eta\|_{m-2}^2\bigr)+\e^{\frac{1}{4}}\t^{-\gamma}\\
&+\bigl(\e+\t^{-\gamma}+\e^2(\|\eta\|_{L^{\infty}}^2+\|R\|_{1,\infty}^2)\bigr)\bigl(\|\eta\|_{m-1}^2+\|R\|_m^2\bigr)\\
&+\sum_{|\alpha|\leq m}\bigl|\int_{\cO}Z^{\alpha}\nabla\pi\cdot  Z^{\alpha}R\bigr|
+\sqrt{\e}\sum_{|\beta|\leq m-1}\bigl|\int_{\cO}Z^{\beta}\chi \CN(\nabla\pi)\cdot Z^{\beta}\eta\bigr|.
\end{split} \eeq
To estimate the two integrals in \eqref{name-me},  we will have to deal with the pressure estimates in the coming subsection.

\subsection{Estimate of the pressure term} \label{sub-pressure}

In view of (\ref{R}), the pressure $\pi$ satisfies
\begin{equation*}
\begin{cases}
\Delta \pi=-\dive F-\dive (u^{\e}\cdot\nabla R+R\cdot\nabla u^{\e}_a)+\partial_t H-\e\Delta H \quad\mbox{ in }\cO, \\
\partial_\mathbf{n} \pi=-F\cdot \mathbf{n}-(u^{\e}\cdot\nabla R+R\cdot\nabla u^{\e}_a)\cdot \mathbf{n}+\e\Delta R\cdot \mathbf{n} \quad\mbox{ on }\partial \cO.
\end{cases}
\end{equation*}

We start the estimate of $\na \pi$ by the following  toy model:

\begin{lemma}\label{pi1}
{\sl Let $\pi_1$ be determined by
\beq\label{pi1p}
\begin{cases}
\Delta\pi_1=-\dive F \quad\mbox{ in }\cO,\\
\partial_\mathbf{n} \pi_1=-F\cdot \mathbf{n}\quad\mbox{ on }\partial \cO,
\end{cases}
\eeq
Then for any  non-negative integer $\ell,$ one has
\begin{eqnarray} \label{pi1-esti}
\|\nabla\pi_1\|_\ell\lesssim\|F\|_\ell.
\end{eqnarray}}
\end{lemma}

\begin{proof}
We proceed by induction on $\ell$.
By taking $L^2$ inner product of the first equation of \eqref{pi1p} and using integrations by parts, we find
\begin{equation*}
\|\nabla\pi_1\|^2 = -\int_{\cO}F\cdot\nabla \pi_1 ,
\end{equation*}
which implies
$\|\nabla\pi_1\|\leq\|F\|$; therefore \eqref{pi1-esti} holds for  $\ell=0$.

Next we assume that \eqref{pi1-esti} holds for  $\ell=m-1$
with $\ell\geq 1$. We are going to prove that  \eqref{pi1-esti} holds for  $\ell=m.$
Indeed by applying $Z^\alpha$ with $|\alpha|=m$ to \eqref{pi1p}, we get
\beno
-\D Z^\alpha\pi_1-[Z^\alpha, \D]\pi_1=Z^\alpha\dive F.
\eeno
By taking $L^2$ inner product of the above equation with $Z^\alpha \pi_1$ and using integration
by parts,  we obtain
\begin{equation} \label{decompi}
\|\nabla Z^{\alpha}\pi_1\|^2
=\sum_{i=1}^4 J_i,
\end{equation}
where
  \begin{eqnarray*}
J_1 := \int_{\partial \cO}[\partial_\mathbf{n} ,Z^{\alpha}]\pi_1 Z^{\alpha}\pi_1 ,\quad
J_2 := -\int_{\cO}[\Delta,Z^{\alpha}]\pi_1 Z^{\alpha}\pi_1 ,\\
J_3 := -\int_{\partial\cO}Z^{\alpha}(F\cdot \mathbf{n})Z^{\alpha}\pi_1 , \quad
J_4 := \int_{\cO}Z^{\alpha}\dive F Z^{\alpha}\pi_1.
\end{eqnarray*}

Let us now handle term by term the quantities above.

\smallskip

\no $\bullet$ \underline{Estimate of $J_1.$}\vspace{0.2cm}

Notice that  $Z_0=0$ on $\p\cO,$  so that if $Z^{\alpha}$ contains the tangential vector field $Z_0$, $J_1=0$.
Without loss of generality,  we may assume that $Z^{\alpha}$ is composed of $Z_i$ with $1\leq i\leq 5$.
We write, $Z^{\alpha}=Z_{i_1}Z^{\alpha_1}$, $|\alpha_1|=\ell-1$, then $$[\partial_\mathbf{n},Z^{\alpha}]=[\partial_\mathbf{n},Z_{i_1}]Z^{\alpha_1}+Z_{i_1}[\partial_\mathbf{n},Z^{\alpha_1}]. $$
 As presented in Subsection
 \ref{sec-appr},
the vector fields, $[\partial_\mathbf{n},Z_{i_1}],$ are also tangential derivatives. By induction,
$[\partial_\mathbf{n},Z^{\alpha}]$ is a tangential derivative  operator of order $m$. By trace inequality, \eqref{trace}, we infer
\begin{equation}
\label{I1}
\vert J_1 \vert \lesssim\|\pi_1\|_{H^m(\partial\cO)}^2\lesssim\|\pi_1\|_m^2+\|\pi_1\|_m\|\nabla\pi_1\|_m\lesssim\|\nabla\pi_1\|_{m-1}\|\nabla\pi_1\|_m.
\end{equation}

\no $\bullet$ \underline{Estimate of $J_2.$}\vspace{0.2cm}

To deal with the commutator, we  use \eqref{Z3} to write
\begin{eqnarray}\label{pi2z0}
\int_{\cO} [\Delta,Z^{\alpha}]\pi_1Z^{\alpha}\pi_1= \sum_{|\alpha_1|,|\alpha_2|\leq m-1}\int_{\cO}\bigl(c_{\alpha_1}\nabla^2Z^{\alpha_1}\pi_1+c_{\alpha_2}\nabla Z^{\alpha_2}\pi_1\bigr)Z^{\alpha}\pi_1,
\end{eqnarray}
where $c_{\alpha_1},c_{\alpha_2}$ are some smooth functions. Yet we do not want the second order normal derivative of $\pi_1$ to appear in
\eqref{pi2z0}. The idea is to  use integration by parts. The cost is that  boundary terms like $\int_{\partial\cO}\bn\cdot c_{\alpha_{1}}\cdot\nabla Z^{\alpha_1}\pi_2Z^{\alpha}\pi_2$ will appear. In general,  we can not guarantee that $\bn\cdot c_{\alpha_1}\cdot\nabla$ is a tangential derivative. One attempt is to use the boundary condition, $\partial_{\mathbf{n}}\pi_1=-F\cdot \mathbf{n},$ and then
the boundary terms will be bounded by $\|F\cdot \mathbf{n}\|_{H^m(\partial\cO)}.$ 
Although  Lemma \ref{V} gives $\|F\|_{H^m(\partial\cO)}\lesssim\t^{-\gamma},$ so that $\|F\cdot \mathbf{n}\|_{H^m(\partial\cO)}$
 will  gives rise to an appropriate estimate of $\pi_1.$ But when we apply similar estimate to deal with $\pi_3$,
 term like $\|R\cdot\nabla R\|_{H^m(\p\cO)}$ will appear, which is out of control.

To overcome the above mentioned difficulty, we distinguish the terms in \eqref{pi2z0} into two cases.

\begin{itemize}
\item If $Z^{\alpha}$ contains a field $Z_0$, then $Z^{\alpha}=0$ on $\p\cO.$ In this case, we use integration by parts to get

    \begin{eqnarray*}
    |J_2|\lesssim \|\nabla\pi_2\|_{m-1}\|\nabla\pi_2\|_m.
    \end{eqnarray*}

\item If $Z^{\alpha}$ does not contain any $Z_0$, we write
\begin{eqnarray*}
Z^{\alpha}=Z_{k_1}Z_{k_2}\cdots Z_{k_m} \with Z_{k_i}=w^{k_i}\cdot\nabla,\quad k_i\in \{1,2,3,4,5\},1\leq i\leq m,
\end{eqnarray*}
for $w^{k_i}$ given in  Subsection \ref{sec-appr}.

As a convention, let $Z^{\alpha_0}=Z^{\beta_{m+1}}$ be the identity operators, we denote
\begin{eqnarray*}
Z^{\alpha_i}:=Z_{k_1}\cdots Z_{k_i} \quad\text{ and }\quad Z^{\beta_i}:=Z_{k_i}\cdots Z_{k_m} \with \quad1\leq i\leq m.
\end{eqnarray*}
Then by \eqref{Z2}, we write
\begin{eqnarray*}
[\Delta,Z^{\alpha}]\pi_1&=&\sum_{i=1}^mZ^{\alpha_{i-1}}[\Delta,Z_{k_i}]Z^{\beta_{i+1}}\pi_1\\
\nonumber&=&\sum_{i=1}^m Z^{\alpha_{i-1}}\bigl(\Delta w^{k_i}\cdot\nabla Z^{\beta_{i+1}}\pi_1+2\nabla w^{k_i}:\nabla^2Z^{\beta_{i+1}}\pi_1\bigr).
\end{eqnarray*}
Notice that for $k_i\neq 0, w^{k_i}\cdot \mathbf{n}=0$ in $\mathcal{V}_{\delta_0/2}$, $|\mathbf{n}|=1$ in $\mathcal{V}_{\delta_0}$ and $\nabla n$ is symmetric, we have
\begin{eqnarray*}
\bn\cdot\nabla w^{k_i}\cdot \mathbf{n}=-\mathbf{n}\cdot\nabla \mathbf{n}\cdot w^{k_i}=-w^{k_i}\cdot \nabla \mathbf{n}\cdot \mathbf{n}=0,\quad \text{ in }\mathcal{V}_{\delta_0/2}.
\end{eqnarray*}
So that $\nabla w^{k_i}:\nabla^2$ contains at most one normal derivative and this implies
\begin{eqnarray*}
\|[\Delta,Z^{\alpha}]\pi_1\|\lesssim \|\nabla\pi_1\|_{m}.
\end{eqnarray*}
As a result, it comes out
\begin{eqnarray} \label{I2}
|J_2|\lesssim\|\nabla\pi_1\|_{m-1}\|\nabla\pi_1\|_m.
\end{eqnarray}
\end{itemize}

\no $\bullet$ \underline{Estimate of $J_3 + J_4$}\vspace{0.2cm}

Again we distinguish to the following two cases:

  \begin{itemize}

  \item
  If $Z^{\alpha}$ contains $Z_0$, then $Z^{\alpha}=0$ on $\p\cO.$ In this case $J_3=0.$ For $J_4,$ we use integration by parts to get
  \begin{eqnarray*}
J_4&=&\sum_{|\alpha_1|\leq m}\int_{\cO}c_{\alpha_1}\cdot\nabla Z^{\alpha_1} F Z^{\alpha}\pi_1\\
&=&\sum_{|\alpha_1|\leq m}\int_{\cO}Z^{\alpha_1} F(\dive c_{\alpha_1}Z^{\alpha}\pi_1+c_{\alpha_1}\cdot\nabla Z^{\alpha}\pi_1),
\end{eqnarray*}
from which, we infer
  \begin{eqnarray*}
\vert J_4 \vert
&\lesssim&\|\nabla \pi_1\|_m\|F\|_m.
\end{eqnarray*}

\item If $Z^{\alpha}$ does not contain  $Z_0$, notice that for $1\leq i\leq 5,$  $Z_i=w^i\cdot \nabla$ and $w^i\cdot \mathbf{n}=0$ in $\mathcal{V}_{\delta_0/2}$ and $\dive w^i=0,$ we get, by using integration by parts, that
  \begin{eqnarray*}
\int_{\cO}Z^{\alpha}\dive F Z^{\alpha}\pi_1&=&(-1)^m\int_{\cO}\dive FZ^{2\alpha}\pi_1\\
&=&(-1)^m\int_{\partial\cO}F\cdot \mathbf{n} Z^{2\alpha}\pi_1-(-1)^m\int_{\cO}F\cdot\nabla Z^{2\alpha}\pi_1\\
&=&\int_{\partial\cO}Z^{\alpha}(F\cdot \mathbf{n})Z^{\alpha}\pi_1+\sum_{|\alpha_1|\leq m}\int_{\cO}c_{\alpha_1}F\cdot Z^{\alpha_1}\nabla Z^{\alpha}\pi_1\\
&=&\int_{\partial\cO}Z^{\alpha}(F\cdot \mathbf{n})Z^{\alpha}\pi_1+ \sum_{|\alpha_1|\leq m}\int_{\cO}\nabla Z^{\alpha}\pi_1\cdot Z^{\alpha_1}(c_{\alpha_1}F),
\end{eqnarray*}
where $c_{\alpha_1}$ are some smooth functions depend only on the vector field in $\mathfrak{W}$.

As a consequence, we obtain
  \begin{eqnarray*}
J_3+J_4= \sum_{|\alpha_1|\leq m}\int_{\cO}\nabla Z^{\alpha}\pi_1\cdot Z^{\alpha_1}(c_{\alpha_1}F),
\end{eqnarray*}
which implies
\begin{equation}
\label{I34}
\vert J_3 + J_4 \vert \lesssim     \|\nabla \pi_1\|_m\|F\|_m.
\end{equation}
\end{itemize}

In view of \eqref{decompi}, by summarizing the estimates, \eqref{I1}, \eqref{I2} and  \eqref{I34}, we  conclude
 the proof of \eqref{pi1-esti}.
\end{proof}

\begin{proposition}
\label{pi-propi}
{\sl For $1\leq m\leq p-4$ and for $t\in [0,T^\e ]$, we have
\begin{equation}\label{pi}
\|\nabla\pi\|_m\lesssim\e^{\frac{1}{4}}\t^{-\gamma}+ \e\|\nabla R\|_m+\t^{-\gamma}\|R\|_{m}
+\e^2\bigl(\|R\|_{L^{\infty}}\|\nabla R\|_m+\|R\|_m\|\nabla R\|_{L^{\infty}}\bigr).
\end{equation}}
\end{proposition}
\begin{proof}
We  first decompose $\pi$ into four terms $\pi=\pi_1+\pi_2+\pi_3+\pi_4,$ where $\pi_1$, $\pi_2$, $\pi_3$ and $\pi_4$
 are determined respectively by \eqref{pi1p} and
\beq\label{pi2p}
\begin{cases}
\Delta\pi_2=\partial_t H \quad\mbox{ in }\cO,\\
\partial_\mathbf{n} \pi_2=0\quad\mbox{ on }\partial \cO,
\end{cases}
\eeq
\beq\label{pi3p}
\begin{cases}
\Delta\pi_3=-\dive(u^{\e}\cdot\nabla R+R\cdot\nabla u^{\e}_a) \quad\mbox{ in }\cO,\\
\partial_\mathbf{n} \pi_3= - (u^{\e}\cdot\nabla R+R\cdot\nabla u^{\e}_a)\cdot \mathbf{n}\quad\mbox{ on }\partial \cO,
\end{cases}
\eeq
and
\beq\label{pi4p}
\begin{cases}
\Delta\pi_4=-\e\Delta H \quad\mbox{ in }\cO,\\
\partial_\mathbf{n} \pi_4=\e\Delta R\cdot \mathbf{n}\quad\mbox{ on }\partial \cO.
\end{cases}
\eeq
\no $\bullet$ \underline{The estimate of $\na\pi_1.$}\vspace{0.2cm}

The estimate $\nabla \pi_1$ relies on Lemma \ref{pi1}. Indeed we deduce from Lemma \ref{pi1} and \eqref{FH1} that
\beq \label{pi1qw}
\|\nabla\pi_1\|_\ell
\lesssim\e^{\frac{1}{4}}\t^{-\gamma}\quad \mbox{for}\ \ 0\leq \ell\leq p-3.
\eeq

\no $\bullet$ \underline{The estimate of $\na\pi_2.$}\vspace{0.2cm}

We claim that  for $0\leq \ell\leq p-3,$
  \begin{eqnarray} \label{ferra}
\|\nabla\pi_2\|_\ell
\lesssim\e^{\frac{1}{4}}\t^{-\gamma}.
\end{eqnarray}

Without losing generality, we may assume that  $\int_{\cO}\pi_2=0$.  Again we proceed by induction on $\ell$.
Indeed by taking $L^2$ inner product of the $\eqref{pi2p}$ with $\pi_2$ and then
using integrations by parts and  the Poincar\'{e} inequality, we find
\begin{eqnarray*}
\|\nabla\pi_2\|^2=-\int_{\cO}(\Delta\pi_2)\pi_2=-\int_{\cO}(\partial_tH)\pi_2\lesssim\|\partial_t H\|\|\nabla\pi_2\|,
\end{eqnarray*}
which together with \eqref{FH1} yields  \eqref{ferra} for $\ell=0$.

Next let us assume that \eqref{ferra} holds for $\ell\leq m-1\leq p-4$,
we are going to prove that  \eqref{ferra} holds for $\ell=m.$  In order to do it, we apply $Z^\alpha$ with $|\alpha| \leq m$ to \eqref{pi2p}
and then taking $L^2$ inner product of the resulting equation with $Z^\alpha\pi_2$ and using integration
by parts, we obtain
\beq \label{pi2d}
  \begin{split}
\|\nabla Z^{\alpha}\pi_2\|^2=&\int_{\partial\cO}(\partial_{\mathbf{n}}Z^{\alpha}\pi_2) Z^{\alpha}\pi_2-\int_{\cO}(\Delta Z^{\alpha}\pi_2) Z^{\alpha}\pi_2\\
=&\int_{\partial\cO}[\partial_{\mathbf{n}},Z^{\alpha}]\pi_2 Z^{\alpha}\pi_2-\int_{\cO}(Z^{\alpha}\partial_t H)Z^{\alpha}\pi_2
+\int_{\cO} [\Delta,Z^{\alpha}]\pi_2Z^{\alpha}\pi_2,
\end{split} \eeq
where we used  $\partial_\mathbf{n} \pi_2=0$ on $\p\cO,$  so that
 $Z^{\alpha} \partial_\mathbf{n} \pi_2=0$ on $\p\cO.$

As  the estimate of $J_1$ in the proof of Lemma \ref{pi1}, if $Z^{\alpha}$ contains $Z_0$,
 the first term of the right hand side of \eqref{pi2d} disappear. Otherwise, $[\partial_n,Z^\alpha]$ is a tangential differential
 operator of order $m.$ Then we get, by applying the trace inequality \eqref{trace}, that
\begin{equation}
\label{pi21}\vert\int_{\partial\cO}[\partial_{\mathbf{n}},Z^{\alpha}]\pi_2 Z^{\alpha}\pi_2 \vert \lesssim\|\pi_2\|_{H^m(\partial\cO)}^2\lesssim\|\pi_2\|_m^2+\|\pi_1\|_m\|\nabla\pi_2\|_m\lesssim\|\nabla\pi_2\|_{m-1}\|\nabla\pi_2\|_m.
\end{equation}
While it follows from \eqref{FH1} that
\begin{eqnarray}\label{pi22}
|\int_{\cO}(Z^{\alpha}\partial_t H)Z^{\alpha}\pi_2|\lesssim \e^{\frac{1}{4}}\t^{-\gamma}\|\pi_2\|_{m}.
\end{eqnarray}

For the last term in the right hand side of \eqref{pi2d}, we deduce along the same line to  that of  $J_2$ in the proof of Lemma \ref{pi1} that
\begin{equation}\label{pi23}
|\int_{\cO} [\Delta,Z^{\alpha}]\pi_2Z^{\alpha}\pi_2|\lesssim \|\nabla\pi_2\|_{m-1}\|\pi_2\|_{m}.
\end{equation}


On the other hand, it follows from the boundary condition $\partial_{\mathbf{n}}\pi_2=0$ that
\begin{eqnarray*}
\|\nabla\pi_2\|_{H^{m-1}(\cO)}\approx\|\pi_2\|_{H^m(\cO)}.
\end{eqnarray*}
Then by inserting the estimates (\ref{pi21}),\eqref{pi22} and \eqref{pi23} into \eqref{pi2d} and then summing up the
resulting inequalities for $|\alpha|\leq m,$ we obtain
\beno
\|\nabla \pi_2\|_m^2\leq C\bigl(\|\nabla\pi_2\|_{m-1}\|\na\pi_2\|_{m}+ \e^{\frac{1}{4}}\t^{-\gamma}\|\na\pi_2\|_{m-1}\bigr),
\eeno
which together with the inductive assumption ensures \eqref{ferra} for $\ell=m.$ This proves \eqref{ferra}.

\smallskip
\no $\bullet$ \underline{The estimate of $\na\pi_3.$}\vspace{0.2cm}

Due to $\dive u^{\e}=\sigma^0 $ and $\dive R=-H,$ we write
\beno
\dive(u^{\e}\cdot\nabla R)=\dive\bigl(R\cdot\na u^\e-H u^\e-\sigma^0 R\bigr).
\eeno
While due to $u^{\e}\cdot \mathbf{n}=R\cdot \mathbf{n}=0$ on $\p\cO$ and $\nabla \mathbf{n}$ is symmetric,
one has
 $$(u^{\e}\cdot\nabla R)\cdot \mathbf{n}=- (u^{\e}\cdot\nabla \mathbf{n})\cdot R
=-(R\cdot\nabla \mathbf{n})\cdot u^{\e}=(u^\e\cdot\nabla \mathbf{n})\cdot R.$$
In view of \eqref{pi3p}, $\pi_3$ verifies
\begin{equation*}
\begin{cases}
\Delta\pi_3
=-\dive\bigl(R\cdot\nabla (u^{\e}+u^{\e}_a)-\sigma^0  R-Hu^\e\bigr)\quad\mbox{ in }\cO,\\
\partial_\mathbf{n}\pi_3=-R\cdot\nabla (u^{\e}+u^{\e}_a)\cdot \mathbf{n}\quad\mbox{ on }\partial\cO .
\end{cases}
\end{equation*}
From Lemma \ref{pi1} and the generalized Sobolev-Gagliardo-Nirenberg-Morse inequality, we infer that
\beno
  \begin{split}
\|\nabla\pi_3\|_m \lesssim&\|R\cdot\nabla (u^{\e}+u^{\e}_a)-\sigma^0  R-Hu^\e\|_m\\
\lesssim&\|R\|_m\|\nabla u^{\e}_a\|_{m,\infty}+ \e^2\bigl(\|R\|_{L^{\infty}}\|\nabla R\|_m+\|R\|_m\|\nabla R\|_{L^{\infty}}\bigr)\\
&+\|\sigma^0 \|_{m,\infty}\|R\|_m+\|H\|_{m}\|u^{\e}_a\|_{m,\infty}+\e^2\|H\|_{m,\infty}\|R\|_m,
\end{split} \eeno
which together \eqref{FH1}, \eqref{FH2}, \eqref{Ra}, and the fact that $\sigma^0$ is smooth and supported in $[0,T],$
ensures that for  $m\leq p-3$,
 \beq
\label{pi3-esti}
\|\nabla\pi_3\|_m \lesssim\t^{-\gamma}\|R\|_m+\e^{\frac{1}{4}}\t^{-\gamma}+
\e^2\bigl(\|R\|_{L^{\infty}}\|\nabla R\|_m+\|R\|_m\|\nabla R\|_{L^{\infty}}\bigr).
\eeq

\no $\bullet$ \underline{The estimate of $\na\pi_4.$}\vspace{0.2cm}

In view of \eqref{pi4p}, we write
\begin{equation*}
\Delta (\pi_4+\e H)=0 \mbox{ in }\cO\   \text{ and } \
\partial_\mathbf{n}(\pi_4+\e H)=-\e\Delta R\cdot \mathbf{n}+\e\partial_\mathbf{n} H \mbox{ on }\partial\cO,
\end{equation*}
from which, we deduce that  for $m\geq 1$
  \begin{eqnarray*}
\|\nabla(\pi_4+\e H)\|_m\lesssim\e\|\Delta R\cdot \mathbf{n}-\partial_\mathbf{n} H\|_{H^{m-\frac{1}{2}}(\partial\cO)}.
\end{eqnarray*}
yet it follows from $(\ref{FH1})$ and trace theorem that for $m\leq p-4,$
\begin{gather*}
\e\|\Delta H\|_m\lesssim \e^{\frac{1}{4}}\t^{-\gamma},\\
\e\|\partial_{\mathbf{n}}H\|_{H^{m-\frac{1}{2}}(\cO)}\lesssim \e\|\nabla^2 H\|_{m}\lesssim\e^{\frac{1}{4}}\t^{-\gamma}.
\end{gather*}
As a result, it comes out
  \begin{eqnarray*}
\|\nabla\pi_4\|_m\lesssim\e\|\Delta R\cdot \mathbf{n}\|_{H^{m-\frac{1}{2}}(\partial\cO)}+\e^{\frac{1}{4}}\t^{-\gamma}.
\end{eqnarray*}
The term $\|\Delta R\cdot \mathbf{n}\|_{H^{m-\frac{1}{2}}(\partial\cO)}$ above can be handled  exactly
 as that in Proposition 19 of \cite{masmoudi} so that
\begin{eqnarray*}
\|\Delta R\cdot \mathbf{n}\|_{H^{m-\frac{1}{2}}(\partial\cO)}\lesssim \|\nabla R\|_m,
\end{eqnarray*}
Then we obtain for $1\leq m\leq p-4,$
  \begin{eqnarray}
\label{pi4-esti}
\|\nabla\pi_4\|_m \lesssim \e\|\nabla R\|_{m}+\e^{\frac{1}{4}}\t^{-\gamma}.
\end{eqnarray}

By summarizing the estimates \eqref{pi1qw},  \eqref{ferra}, \eqref{pi3-esti} and \eqref{pi4-esti},  we arrive at (\ref{pi}).
This completes the proof of Proposition
 \ref{pi-propi}.
\end{proof}

With Proposition \ref{pi-propi},
we now turn to the
estimate of the two integrals involving the pressure term in  \eqref{name-me}.
\begin{corollary}\label{pre-Reta}
{\sl Let $2\leq m\leq p-4.$ Then for $\alpha,\beta$ satisfying $|\alpha|\leq m, |\beta|\leq m-1,$ and any $\lambda>0$
there exists $C_\lambda$ so that
\beq\label{pre1} \begin{split}
\vert  \int_{\cO}Z^{\alpha}\nabla\pi\cdot  Z^{\alpha}R  \vert\leq &\lambda\e\|\nabla R\|_m^2+C\e^{\frac{1}{4}}\t^{-\gamma}\\
&+C_{\lambda}\bigl(\e+\t^{-\gamma}+\e^2(\|\eta\|_{L^{\infty}}^2+\|R\|_{1,\infty}^2)\bigr)\|R\|_m^2,
\end{split} \eeq and
\beq \label{pre2}\begin{split}
\sqrt{\e} \vert \int_{\cO}Z^{\beta}\chi \CN(\nabla\pi)\cdot Z^{\beta}\eta  \vert & \leq \lambda\e\|\nabla\eta\|_{m-1}^2+C_{\lambda}
\bigl( \e^{\frac{1}{2}}\t^{-2\gamma}+\e^4\t^{-2\gamma}\|R\|^2_{L^{\infty}}\bigr)\\
&+C_{\lambda}\bigl(\e+\t^{-2\gamma}+\e^3(\|\eta\|^2_{L^{\infty}}+\|R\|^2_{1,\infty})\bigr)(\|R\|_m^2+\|\eta\|_{m-1}^2).
\end{split} \eeq
 }
\end{corollary}
\begin{proof}
Thanks to (\ref{pi}), for any $\lambda>0,$  we get, by applying  Young's inequality, that
\begin{eqnarray*}
\vert  \int_{\cO}Z^{\alpha}\nabla\pi\cdot  Z^{\alpha}R  \vert &\leq & \lambda\e\|\nabla R\|_m^2+\e^{\frac{1}{4}}\t^{-\gamma}+C_{\lambda}\left(\e+\t^{-\gamma}\right)\|R\|_{m}^2\\
\nonumber&&+\left(C_{\lambda}\e^3\|R\|_{L^{\infty}}^2+\e^2\|\nabla R\|_{L^{\infty}}\right)\|R\|_m^2,
\end{eqnarray*}
which together with \eqref{etaL} ensures \eqref{pre1}.

On the other hand, due to $\eta=0$ on $\p\cO,$
 by using integration by parts and Young's inequality, we find that for any $\lambda>0,$
  \begin{eqnarray}  \label{star}
\sqrt{\e} \vert \int_{\cO}Z^{\beta}\chi \CN(\nabla\pi)\cdot Z^{\beta}\eta \vert  
 \leq \lambda\e\|\nabla\eta\|_{m-1}^2+C_{\lambda}\|\nabla\pi\|_{m-1}^2 ,
\end{eqnarray}
Yet it follows from (\ref{pi}), (\ref{eta}) and (\ref{etaL}) that
  \begin{eqnarray*}
\|\nabla\pi\|_{m-1}^2
&\lesssim& \e^{\frac{1}{2}}\t^{-2\gamma}+\e^4\t^{-2\gamma}\|R\|^2_{L^{\infty}}\\ \nonumber&&+\bigl(\e+\t^{-2\gamma}+\e^3(\|\eta\|^2_{L^{\infty}}+\|R\|^2_{1,\infty})\bigr)(\|R\|_m^2+\|\eta\|_{m-1}^2) .
\end{eqnarray*}
Substituting  the above estimate  into  \eqref{star} leads to \eqref{pre2}.
\end{proof}

By inserting the estimates  \eqref{pre1} and \eqref{pre2} into \eqref{name-me} and choosing $\lambda$ to be
sufficiently small, we deduce that
 for $2\leq m\leq p-4$ and for $t\in [0,T^\e ]$,
\beq \label{Reta}
  \begin{split}
\f{d}{dt}\bigl(&\|R(t)\|_m^2+\|\eta(t)\|_{m-1}^2)+\e\bigl(\|\nabla R\|_m^2+\|\na\eta\|_{m-1}^2\bigr)\\
&\lesssim \e\bigl(\|\nabla R\|_{m-1}^2+\|\na\eta\|_{m-2}^2\bigr)+\e^{\frac{1}{4}}\t^{-\gamma}+\e^4\t^{-2\gamma}\|R\|^2_{L^{\infty}}\\
&\quad+\bigl(\e+\t^{-\gamma}+\e^2(\|\eta\|_{L^{\infty}}^2+\|R\|_{1,\infty}^2)\bigr)(\|\eta\|_{m-1}^2+\|R\|_m^2).
\end{split} \eeq

In order to close the  estimate of \eqref{Reta},  we still need the estimate of $\|R\|_{1,\infty}$ and $\|\eta\|_{L^{\infty}}$,
which will be the content of the next section.


\subsection{Estimate of $\|R\|_{1,\infty}$ and $\|\eta\|_{L^{\infty}}$}
\label{lipR}

\begin{proposition}\label{p1}
{\sl  Let $m>3$ be an integer. Then one has
  \beq \label{p1ad}
\e\|{R}(t)\|^2_{1,\infty}\leq C\bigl(
\| {R}(t)\|^2_m +\|\eta(t)\|_{m-1}^2+\e \t^{-2\gamma}\bigr).
\eeq}
\end{proposition}
\begin{proof}
We first deduce from  Proposition 20 of
\cite{masmoudi} that for $m_0>1$,
  \begin{eqnarray*}
\|{R}(t)\|^2_{L^{\infty}} \leq C\bigl(\|\partial_\mathbf{n}{R}(t)\|_{m_0}\|{R}(t)\|_{m_0}+\|{R}(t)\|_{m_0}^2\bigr),
\end{eqnarray*}
which together with (\ref{eta}) implies
\beq \label{S5eq86}
\begin{split}
\e\|{R}(t)\|^2_{L^{\infty}}\leq & C\bigl(\e\|\partial_\mathbf{n}{R}(t)\|_{m_0}
\|{R}(t)\|_{m_0}+\e\|{R}(t)\|_{m_0}^2\bigr)\\
 \leq & C\bigl( \|\eta(t)\|_{m-1}^2+\| {R}(t)\|^2_m+\e \t^{-2\gamma}\bigr)\quad\mbox{if}\  m \geq m_0+1.
\end{split}
\eeq
Along the same line, we can prove similar  estimate for $\|ZR\|_{L^{\infty}}$ if $m \geq m_0+2.$
\end{proof}

In order to estimate $\|\eta\|_{L^{\infty}},$ we introduce
  \begin{equation}\label{teta}
\tilde{\eta} :=\sqrt{\e}\nabla \wedge R .
\end{equation}
\begin{lemma}\label{equi-eta}
{\sl  Let $\eta$ and $\tilde{\eta} $ be determined respectively by \eqref{defeta} and \eqref{teta}. Then one has
 \begin{equation*}
\|\eta\|_{L^{\infty}}+\|R\|_{1,\infty}+\sqrt{\e}\t^{-\gamma}\approx\|\tilde{\eta}\|_{L^{\infty}}+\|R\|_{1,\infty}+\sqrt{\e}\t^{-\gamma}.
\end{equation*}}
\end{lemma}
\begin{proof}
On the one hand, it follows from (\ref{etaL}) that
  \begin{equation*}
\|\tilde{\eta}\|_{L^{\infty}}\lesssim\sqrt{\e}\|\nabla R\|_{L^{\infty}}\lesssim \|\eta\|_{L^{\infty}}+\|R\|_{1,\infty}+\sqrt{\e}\t^{-\gamma},
\end{equation*}
which implies
\beq \label{teta1}
\|\tilde{\eta}\|_{L^{\infty}}+\|R\|_{1,\infty}+\sqrt{\e}\t^{-\gamma} \lesssim \|\eta\|_{L^{\infty}}+\|R\|_{1,\infty}+\sqrt{\e}\t^{-\gamma}.
\eeq

On the other hand,
due to
$n\wedge(\nabla\wedge R)=\nabla R\cdot \mathbf{n}-\partial_\mathbf{n} R,$ we have
  \begin{equation*}
\sqrt{\e}\|\partial_\mathbf{n} R\|_{L^{\infty}}\lesssim \|\tilde{\eta}\|_{L^{\infty}}+\sqrt{\e}\|ZR\|_{L^{\infty}}+\sqrt{\e}\|\partial_\mathbf{n} R\cdot \mathbf{n}\|_{L^{\infty}}.
\end{equation*}
Yet it follows from (\ref{div}) and (\ref{FH2}) that
 $$\|\partial_\mathbf{n} R\cdot \mathbf{n}\|_{L^{\infty}}\lesssim \|ZR\|_{L^{\infty}}+\t^{-\gamma},$$
 so that
  \begin{equation*}
\sqrt{\e}\|\partial_\mathbf{n} R\|_{L^{\infty}}\lesssim \|\tilde{\eta}\|_{L^{\infty}}+\|R\|_{1,\infty}+\sqrt{\e}\t^{-\gamma}.
\end{equation*}
This together with \eqref{defeta} shows that the other side of the inequality \eqref{teta1} holds.  This
 concludes the proof of Lemma \ref{equi-eta}.
\end{proof}
Now let us set
  \begin{equation} \label{fraknm}
\frak{N}_m(t) :=\|R(t)\|_m^2+\|\eta(t)\|_{m-1}^2+\e\|\tilde{\eta}(t)\|_{L^{\infty}}^2.
\end{equation}
Note that (\ref{R0}) implies
\begin{equation}
\|R_0\|_{m}\lesssim \e^{-\frac{1}{4}},\|\nabla R_0\|_{m-1}\lesssim \e^{-\frac{3}{4}}, \|\nabla^2R_0\|_{m-2}\lesssim \e^{-\frac{5}{4}}.
\end{equation}
Hence
\begin{equation}\label{td0}
\|\eta_0\|_{m-1}\lesssim\sqrt{\e}\|\nabla R_0\|_{m-1}\lesssim \e^{-\frac{1}{4}} \text{ and }
\|\tilde{\eta}_0\|_{L^{\infty}(\cO)}\lesssim \sqrt{\e}\|\nabla R_0\|_{H^1(\cO)}\lesssim \e^{-\frac{3}{4}}.
\end{equation}
Therefore
\begin{equation}
\frak{N}_m(0)\lesssim\e^{-\frac{1}{2}}.
\end{equation}
\begin{proposition}\label{S5prop12}
{\sl Let $\frak{N}_m(t) $ be determined by \eqref{fraknm}. Then there exist constant $\e_0, C$ so that
 for $\e\leq \e_0,$ $4\leq m\leq p-4$
  \begin{equation}\label{N}
\frak{N}_m(t)+\e\int_0^t\bigl(\|\nabla R\|_m^2+\|\eta\|_{m-1}^2\bigr)\,ds\leq C\e^{-\frac{1}{2}}\quad\mbox{for}\ \ t\leq T^\e.
\end{equation}}
\end{proposition}

\begin{proof}
In view of (\ref{R}), $\tilde{\eta}$ satisfies
  \begin{equation*}
\partial_t \tilde{\eta}-\e\Delta \tilde{\eta}+u^{\e}\cdot\nabla \tilde{\eta}+\sqrt{\e}\nabla u^{\e}\wedge\nabla R+\sqrt{\e} \nabla \wedge(R\cdot\nabla u^{\e}_a)=\sqrt{\e}\nabla \wedge F.
\end{equation*}
Maximum principle for the transport-diffusion equation ensures that
\begin{equation}\label{N1a}
\|\tilde{\eta}(t)\|_{L^{\infty}}\leq \|\tilde{\eta}_0\|_{L^{\infty}}+\sqrt{\e}\int_0^t\bigl(\|\nabla \wedge F\|_{L^{\infty}}+\|\nabla u^{\e}\wedge\nabla R\|_{L^{\infty}}
+\|\nabla \wedge(R\cdot\nabla u^{\e}_a)\|_{L^{\infty}}\bigr)\,ds.
\end{equation}
Applying (\ref{FH2}) gives
 $$\sqrt{\e}\|\nabla \wedge F(s)\|_{L^{\infty}}\lesssim \s^{-\gamma}.$$
While it follows from \eqref{Ra} that
\begin{equation*}
\sqrt{\e} \|\nabla u^{\e}\wedge\nabla R(s)\|_{L^{\infty}}\lesssim \sqrt{\e}\s^{-\gamma}\|\nabla R(s)\|_{L^{\infty}}+\e^{\frac{5}{2}}\|\nabla R(s)\|_{L^{\infty}}^2.
\end{equation*}
Notice that
 $$\sqrt{\e}\nabla\wedge(R\cdot\nabla u^{\e}_a)
=\sqrt{\e} \left(\p_i R\cdot\nabla (u^{\e}_a)^j-\p_jR\cdot\nabla (u^{\e}_a)^i\right)_{3\times 3}+ R\cdot\nabla(\sqrt{\e}\nabla\wedge u^{\e}_a),$$
we infer
  \begin{equation*}
\sqrt{\e}\| \nabla \wedge(R\cdot\nabla u^{\e}_a)(s)\|_{L^{\infty}}\lesssim \s^{-\gamma}
\bigl(\sqrt{\e}\|\na R(s)\|_{L^{\infty}}+\|R(s)\|_{L^{\infty}}).
\end{equation*}
By inserting the above estimates into \eqref{N1a} and then using
 (\ref{eta}), (\ref{etaL}) and \eqref{td0}, we achieve
  \begin{eqnarray*}
\nonumber\|\tilde{\eta}(t)\|_{L^{\infty}}\!\!\!&\lesssim&\!\!\!\e^{-\frac{3}{4}}+\int_0^t
\Bigl(\s^{-\gamma}+\sqrt{\e}\bigl(\s^{-\gamma}\|\nabla R\|_{L^{\infty}}+\e^2\|\nabla R\|_{L^{\infty}}^2\bigr)+\s^{-\gamma}\bigl(\|\tilde{\eta}\|_{L^{\infty}}+\|R\|_{L^{\infty}}\bigr)\Bigr)\,ds\\
&\lesssim&\!\!\!\e^{-\frac{3}{4}}+\int_0^t \s^{-\gamma}\bigl(\|\tilde{\eta}\|_{L^{\infty}}+\|R\|_{1,\infty}+\e^{\frac{3}{2}}(\|\tilde{\eta}\|^2_{L^{\infty}}+\|R\|^2_{1,\infty})\bigr)\,ds,
\end{eqnarray*}
from which, \eqref{p1ad} and \eqref{fraknm}, we deduce
  \begin{equation*}
\e\|\tilde{\eta}(t)\|_{L^{\infty}}^2\lesssim \e^{-\frac{1}{2}}+\int_0^t\s^{-\gamma}\bigl(\frak{N}_m+\e^2\frak{N}_m^2\bigr)\,ds.
\end{equation*}
For any $t\leq T^\e,$ by integrating (\ref{Reta}) over $[0,t]$ and then summing up the resulting inequality
with the above inequality, we obtain for $2\leq m\leq p-4$ that
\beq \label{fraknm1}
\begin{split}
\frak{N}_m(t)+\e\bigl(\|\nabla R\|_{L^2_t(H^m_{\rm co})}^2+\|\na \eta\|_{L^2_t(H_{\rm co}^{m-1})}^2\bigr)
&\leq
C\Bigl(\e\bigl(\|\nabla R\|_{L^2_t(H^{m-1}_{\rm co})}^2+\|\na \eta\|_{L^2_t(H_{\rm co}^{m-2})}^2\bigr)\\
&+\e^{-\frac{1}{2}}+\int_0^t\bigl((\e+\s^{-\gamma})\frak{N}_m+\e^2 \frak{N}_m^2\bigr)\,ds\Bigr).
\end{split}
\eeq
While thanks to Propositions \ref{S5prop3} and \ref{prop-etae}, we get, by a similar derivation of \eqref{fraknm1}, that
\beno
\begin{split}
\frak{N}_1(t)+\e\bigl(\|\nabla R\|_{L^2_t(H^1_{\rm co})}^2+\|\na \eta\|_{L^2_t(L^2)}^2\bigr)
\leq
C\Bigl(&\e^{-\frac{1}{2}}+\e\|\nabla R\|_{L^2_t(L^2)}^2\\
&+\int_0^t\bigl((\e+\s^{-\gamma})\frak{N}_m+\e^2 \frak{N}_m^2\bigr)\,ds\Bigr),
\end{split}
\eeno
which together with Proposition \ref{S5prop1} ensures that
\beq \label{fraknm2}
\begin{split}
\frak{N}_1(t)+\e\bigl(\|\nabla R\|_{L^2_t(H^1_{\rm co})}^2+\|\na \eta\|_{L^2_t(L^2)}^2\bigr)
\leq
C\Bigl(&\e^{-\frac{1}{2}}+\int_0^t\bigl((\e+\s^{-\gamma})\frak{N}_m+\e^2 \frak{N}_m^2\bigr)\,ds\Bigr).
\end{split}
\eeq

By virtute of \eqref{fraknm1} and \eqref{fraknm2}, we get by an inductive argument that
\begin{equation*}
\frak{N}_m(t)+\e\bigl(\|\nabla R\|_{L^2_t(H^1_{\rm co})}^2+\|\na \eta\|_{L^2_t(L^2)}^2\bigr)
\leq C\Bigl(\e^{-\frac{1}{2}}+\int_0^t\bigl((\e+\s^{-\gamma})\frak{N}_m+\e^2 \frak{N}_m^2\bigr)\,ds\Bigr),
\end{equation*}
from which and a comparison argument, we infer
\beq \label{fraknm3}
\begin{split}
\frak{N}_m(t)+&\e\bigl(\|\nabla R\|_{L^2_t(H^1_{\rm co})}^2+\|\na \eta\|_{L^2_t(L^2)}^2\bigr)\\
\leq &C\e^{-\f12}\left(1-C^2\e^{\f32}t\right)^{-1}\exp\Bigl(C\int_0^t(\e+\s^{-\gamma})\,ds\Bigr)\quad\mbox{for}\ \ t\leq T^\e\leq \f{T}\e.
\end{split}
\eeq
In particular, if we take $\e$ to be so small that $\e\leq (2TC^2)^{-\f23},$  we deduce from \eqref{fraknm3} that
\beno
\frak{N}_m(t)+\e\bigl(\|\nabla R\|_{L^2_t(H^1_{\rm co})}^2+\|\na \eta\|_{L^2_t(L^2)}^2\bigr)\leq Ce^{CT}\e^{-\f12},
\eeno
which yields \eqref{N}. This completes the proof of Proposition \ref{S5prop12}.
\end{proof}

\subsection{End of the proof of \eqref{APRIORI-R}}
\label{enddp}
For our purpose, we can take $(\gamma,k,p,s,q)=(2,2,8,4,4)$ in Section \ref{sec-approx} and $m=4$. By an iteration argument,
 we find that $(\gamma_1,k_1,p_1,s_1,q_1)=(107,166,178,252,107)$ and  $u_0$ and $u_*$  belongs to $H^{177}(\cO)$ are sufficient.

Then for any $t\in (0,T^{\e})$, we deduce from \eqref{eta} that
\begin{eqnarray*}
\e^2\|R(t)\|_{H^1(\cO)}\lesssim\e^{\frac{3}{2}}\bigl(\|R(t)\|_1+\|\eta(t)\|+\sqrt{\e}\t^{-\gamma}\bigr),
\end{eqnarray*}
from which, \eqref{fraknm} and \eqref{N}, we infer
\begin{eqnarray*}
\e^2\|R(t)\|_{H^1(\cO)}\lesssim \e^{\frac{3}{2}}\frak{N}_4^{\frac{1}{2}}(t)+\e^2\leq C \e^{\frac{5}{4}}.
\end{eqnarray*}
This concludes the proof of \eqref{APRIORI-R}.

\section{Proof of Theorem \ref{t-lag}}
\label{sec-proof-lagp}

This section is devoted to the proof of Theorem \ref{t-lag}.
The scheme of the proof of Theorem \ref{t-lag} is very similar to that of
Theorem \ref{th} with some simplifications due to the facts that the statement of Theorem \ref{t-lag}  only promises approximate controllability (see \cite[Remark 3]{GH2}), and for one positive time before the imparted time, which can be chosen arbitrarily small  (recall Remark  \ref{assert}).
 Therefore there is no need of the well-prepared dissipation of the boundary layers  as we did in
Section \ref{serre-chevalier} in the course of proving Theorem \ref{th}.
 Again we make use of a rapid and violent control  so that the behavior of the system will follow from the one of  its  inviscid counterpart.
 Let us therefore recall a few ingredients used in \cite{GH2} to tackle the inviscid case. We recall the notation for the flow map already used in the statement of Theorem \ref{t-lag}:  with a vector field $u$ depending on $t$ in $[0,T]$ and on the space variable $x$,
we associate, when it makes sense (below we will only need flow maps in some cases where the classical Cauchy-Lipschitz theorem applies),
 the flow map $\phi^u$ such that
$ \partial_t\phi^u(t,s,x)=u(t,\phi^u(t,s,x)) $ for any $t,s$ in $[0,T]$ and for any $x$ in $\Omega$,
and  $ \phi^u (s,s,x)=x$ for any $s$ in $[0,T]$ and for any $x$ in $\Omega$.
First thanks to  a construction due to  Krygin \cite{krygin}, given $\gamma_0$ and $\gamma_1$ two Jordan surfaces  included in $\Omega$ such that
$\gamma_{0} $ and $ \gamma_{1} $ are isotopic in $ \Omega$ and surrounding the same volume, there exists a volume-preserving diffeotopy $h \in C^\infty([0,1] \times \Omega;\Omega)$ such that $\partial_{t} h$ is compactly supported in $(0,1) \times \Omega$, $h(0,\gamma_{0})=\gamma_0$ and $h(1,\gamma_{0})=\gamma_1$.
Then the smooth vector field $X(t,x):= \partial_{t} h (t,h^{-1}(x))$
is compactly supported in $(0,1) \times \Omega$ and satisfies
for all $t\in [0,1]$,  $\phi^X(t,0,\gamma_0)\subset \Omega$,
$\phi^{X}(1,0,\gamma_{0}) = \gamma_{1}$ and
 $\dive X =0 \ \text{ in } \ (0,1) \times \Omega$.
Then, thanks to \cite[Proposition 2.2]{GH2}, for any $\nu>0$ and $k\in \N,$ there exists $\theta^0 \in C^{\infty}_0((0,1) \times \overline{\Omega};\R)$ such that
\begin{eqnarray}\label{S6eq2}
\begin{cases}
\forall\ t\in [0,1],\ \Delta_x \theta^0 =0 \ \text{ in } \ \Omega, \\
\dfrac{\partial \theta^0 }{\partial n}=0 \ \text{ on } \ [0,1] \times (\partial \Omega\setminus\Sigma), \\
\forall\ t\in [0,1], \ \phi^{\nabla \theta^0 }(t,0,\gamma_0) \subset \Omega, \\
\| \phi^{\nabla\theta^0 }(1,0,\gamma_0)-\gamma_1\|_{C^k(\Bbb{S}^2)} \leq \nu,
\end{cases}
\end{eqnarray}
up to a reparameterization.

With these ingredients of the inviscid case in hands, let us now start the proof of Theorem \ref{t-lag}. It is split into two parts, depending on the regularity of the initial data.

 \begin{proof}[Proof of the first part of Theorem \ref{t-lag}. Case where  $u_0$ is  in $C^{k,\alpha}({\Omega};\R^{3})$]
 \ \par \

 We first consider the case where $u_0$ is in $C^{k,\alpha}({\Omega};\R^{3})$, with $\alpha \in (0,1)$ and $k \in \N \setminus \{0\}$,  and satisfies $\dive u_{0}=0 $ in $ \Omega$ and $u_{0} \cdot n =0 $ on $ \partial \Omega$.
 One also assumes that
$T_0>0$,  $\gamma_0$ and $\gamma_1$ two Jordan surfaces  included in $\Omega$ such that
$\gamma_{0} $ and $ \gamma_{1} $ are isotopic in $ \Omega$ and surrounding the same volume, 
are given.

We first  use the scaling transformation \eqref{S2scaling} to transform
 our original problem \eqref{NS} to \eqref{NSA}.
Then we consider the same expansion as in the proof of Theorem \ref{th}, that is, \eqref{DEF-R}, with $u^\e_a$ being given by
\eqref{previou} and $u^0 := \frak{e}(\nabla \theta^0)$, where  $\theta^0$ is given by \eqref{S6eq2} and
  $\frak{e}$ is a  linear continuous extension operator from $ C^{k,\beta}(\overline{\Om}; \R^3)\to C_0^{k,\beta}(\cO; \R^3).$
 Of course, $u^0$ thus constructed verifies Lemma \ref{lmu0} except \eqref{flush},
which is unnecessary here.

Let us first focus on proving \eqref{eq:approx} for $k=0$, while maintaining the condition \eqref{eq:exact2}.
It follows from \eqref{S5eq86} and \eqref{N} that
\beq \label{S6eq3}
\begin{split}
\e^2\int_0^1\|{R}(t)\|_{L^{\infty}(\cO)}\,dt
 \leq & C\e^{\f32}\int_0^1\bigl( \|\eta(t)\|_{m-1}+\| {R}(t)\|_m+\e \bigr)\,dt\quad\mbox{if}\  m>2\\
 \leq & C\e\bigl(\|{\e}^{\f12}\eta\|_{L^2((0,1); H^{m-1}_{\rm co})}+\e^{\f12}\|R\|_{L^\infty((0,1); H^{m-1}_{\rm co})}+\e\bigr)\\
 \leq &C\sqrt{\e}.
\end{split}
\eeq
We remark that the choice of $1$ is quite arbitrary but the fact that we consider here times of order $O(1)$,
  not of order $O(1/\e)$ as in the proof of  \eqref{DEF-R},  makes the use of the well-prepared dissipation of the boundary layers  unnecessary here.

With thus obtained $u^\e,$ we define $u$ via \eqref{DFD} and we denote by $p$ the corresponding pressure.
 Then  $(u,p)$ is in $L^{\infty}(0,T;C^{k,\alpha}(\Omega;\R^{4}))$ and satisfies (\ref{NS}) on $[0,\e]$.
We
 denote by $\phi^{u}(t,s,x)$ and $\phi^{u^0}(t,s,x)$ the flow maps associated with $u$ and $u^0$ respectively.
Then  in view of \eqref{DFD} and \eqref{DEF-R},  we write
\beno
\begin{split}
\p_t\bigl(\phi^{u}(t,s,x)-\phi^{u^0}(t/\e,s,x)\bigr)
=&\f1\e\bigl(u^\e(t/\e,\phi^{u}(t,s,x))-u^0(t/\e,\phi^{u^0}(t/\e,s,x))\bigr)\\
=&\f1\e\bigl(u^0(t/\e,\phi^{u}(t,s,x))-u^0(t/\e,\phi^{u^0}(t/\e,s,x))\bigr)\\
&+\f1\e\frak{R}^\e(t/\e,\phi^{u}(t,s,x)) \with
\frak{R}^\e:=u^\e_a-u^0 +{\e}^2R,
\end{split}
\eeno
from which, we get, by applying Gronwall's inequality, that
\beno
\label{STT}
\bigl\|\phi^{u}(t,s,\cdot)-\phi^{u^0}(t/\e,s,\cdot)\bigr\|_{L^\infty(\cO)}\leq
{\e}^{-1}\|\frak{R}^\e(t/\e)\|_{L^1((s,t); L^\infty(\cO))}
\exp\Bigl(\f1\e\int_s^t \|\na u^0(t')\|_{L^\infty(\cO)}\,dt'\Bigr).
\eeno
On the other hand, it follows from \eqref{ua0} and \eqref{S6eq3} that
\beno
\begin{split}
&\|u^a_\e-u^0\|_{L^\infty((0,\e)\times\cO)}\leq C\e^{\f12},\quad \f1\e\int_0^\e \|\nabla u^0(t')\|_{L^\infty(\cO)}\,dt'\leq
\|\na u^0\|_{L^\infty((0,1)\times\cO)} ,\\
&\e \|\frak{R}^\e(t/\e)\|_{L^1((0,\e);L^\infty(\cO))}=\e^2 \|\frak{R}^\e\|_{L^1((0,1);L^\infty(\cO))}\leq C\sqrt{\e},
\end{split}
\eeno
so that for any $t,s\in [0,\e],$ there holds
\beq\label{S6eq5}
\bigl\|\phi^{u}(t,s,\cdot)-\phi^{u^0}(t/\e,s,\cdot)\bigr\|_{L^\infty(\cO)}\leq C\sqrt{\e}.
\eeq
Then \eqref{S6eq2} together with \eqref{S6eq5} ensures that
\beq \label{S6eq6}
\begin{split}
\bigl\|\phi^{u}(\e,0,\gamma_0)-\gamma_1\bigr\|_{L^\infty(\Bbb{S}^2)}\leq
&\bigl\|\phi^{u}(\e,0,\cdot)-\phi^{\na\theta^0}(1,0,\cdot)\bigr\|_{L^\infty(\Omega)}\\
&+\bigl\|\phi^{\na\theta^0}(1,0,\gamma_0)-\gamma_1\bigr\|_{L^\infty(\Bbb{S}^2)}\leq C\bigl(\sqrt{\e}+\nu\bigr).
\end{split}
\eeq
This entails  \eqref{eq:exact2} and \eqref{eq:approx} for $k=0$, with the time  $T:= \e \in (0,T_0)$, by  appropriate choices of $\nu$ and $ \e$.
Now to prove \eqref{eq:approx} for $k> 0$ it is sufficient to use the counterpart of \eqref{STT} for higher order derivatives, see for instance \cite[Equation (23)]{koch}.
This estimate  is performed in a compact set $K$ such that an open neighborhood of  $\cup_{t\in [0, \varepsilon]} \, \phi^u(t,0,\gamma_0)$ is contained in $K$ and such that
 $K$ is included in $\Omega$, the existence of such a compact set is  granted by  the condition \eqref{eq:exact2}.
  The higher order estimates of the velocity field on $K$ are deduced, by Sobolev embedding, from the estimate of $\|R(t)\|_m $ in Proposition
\ref{S5prop12}, since on $K$,  $\|R(t)\|_m $  is equivalent to the usual  Sobolev norm  of order $m$, by the very definition of the
 the Sobolev conormal spaces in \eqref{scn}. The details are left to the reader.

 This completes the proof of the first part of
Theorem \ref{t-lag}.

\end{proof}


 \begin{proof}[Proof of the second part of Theorem \ref{t-lag}. Case where  $u_0$ is in $H^1({\Omega};\R^{3})$]
  \ \par \

 Let us now tackle the case where the initial data  $u_0$ is only in $H^1({\Omega};\R^{3})$, with still the compatibility conditions: $\dive u_{0}=0 $ in $ \Omega$ and $u_{0} \cdot n =0 $ on $ \partial \Omega$.
In this case we first use  the regularization result  of Theorem \ref{S}, or more precisely of Theorem \ref{thS} in the Appendix \ref{appa}.
More precisely, for  $\nu>0$, which will be chosen small enough later on,
 we consider  $u$ to be the unique  solution in
  $u\in C([0,\nu];H^1(\Omega))\cap L^2([0,\nu];H^2(\Omega))$
  of \eqref{SAEQ0}  on $[0,\nu]$  with initial data $u_0$.
  In particular,
 for any $s_0\in (2,3),$  we deduce from interpolation inequality and \eqref{SAeq-2} that
  \beno
 \begin{split}
 \bigl\| t^{\f{s_0}2-1}u\bigr\|_{L^2((0,\nu); H^{s_0}(\Om))}\leq  &C\bigl\| t^{\f{1}2}u\bigr\|_{L^2((0,\nu); H^3(\Om))}^{s_0-2} \, \|u\|_{L^2((0,\nu); H^2(\Om))}^{3-s_0}\leq C(\|u_0\|_{H^1}),
 \end{split}
 \eeno
 from which and Sobolev imbedding theorem, we infer that for any $s_0\in (5/2, 3)$,
 \beno
 \begin{split}
 \|\na u\|_{L^1((0,\nu); L^\infty(\Omega))}\leq &C\bigl\| t^{\f{s_0}2-1}u\bigr\|_{L^2((0,\nu); H^{s_0}(\Om))}\bigl\| t^{1-\f{s_0}2}\bigr\|_{L^2(0,\nu)}
 \leq C(\|u_0\|_{H^1})\nu^{\f{3-s_0}2} .
 \end{split}
 \eeno
 Consequently, according to the classical Cauchy-Lipschitz theorem, the vector field  $u$ generates a unique flow map $\phi^u(t,s,x)$ on $[0,\nu].$ Furthermore, for any $t,s\in [0,\nu],$ there holds
 \beq \label{S6eq4}
 \begin{split}
 \bigl\|\phi^u(t,s,x)-x\bigr\|_{L^\infty(\Om)}\leq &\int_{0}^\nu \|u(t,\cdot)\|_{L^\infty(\Omega)}\,dt
 \\  \leq &
 \bigl\|t^{\f{1}2}u\bigr\|_{L^\infty((0,\nu); H^2(\Om))}\int_{0}^\nu t^{-\f12}\,dt\leq  C(\|u_0\|_{H^1})\sqrt{\nu}.
 \end{split}
 \eeq
In particular, this entails that
for any $t\in [0,\nu]$,
$\phi^u(t,0,\gamma_0)\subset \Omega$ and that
the Jordan surface $\gamma_\ast := \phi^u(\nu,0,\gamma_0)$ satisfies
 \beq \label{S6eq444}
 \begin{split}
 \bigl\| \gamma_\ast -  \gamma_0\bigr\|_{L^\infty( \Bbb{S}^2)} \leq  C(\|u_0\|_{H^1})\sqrt{\nu}.
 \end{split}
 \eeq
Moreover it follows from  \eqref{SAeq-2} that $u_\ast  := u( \nu,\cdot)$ belongs to $ H^\infty(\Om)$.
 Thus we can use the first part of Theorem \ref{t-lag}, in particular the estimate \eqref{S6eq6} on the time interval
  $[\nu, \nu+ \varepsilon]$, so that there exists an extension of $u$, which we still denote by $u$, to the time interval   $[\nu, \nu+ \varepsilon]$ such that
  $u$ is in $C([0,\nu+ \varepsilon];H^1(\Omega))$ and in $L^2([0,\nu+ \varepsilon];H^2(\Omega))$  and generates a flow $\phi^{u}$ such that
for any $t\in [\nu,\nu+\e]$,
$\phi^u(t, \e, \gamma_0 )\subset \Omega$, such that
  \beq \label{S6eq666}
\begin{split}
\bigl\|\phi^{u}(\nu+ \varepsilon ,\nu,\gamma_0)- \gamma_1\bigr\|_{L^\infty(\Bbb{S}^2)} \leq C\sqrt{\e}.
\end{split}
\eeq
Furthermore, $\phi^{u}(\nu+ \varepsilon ,\nu,\cdot)$ is Lipschitz. Thus combining these three last properties with
\eqref{S6eq444},
and choosing $\e$ and $\nu$ small enough,
 we arrive at
\begin{equation*}
\begin{split}
\| \phi^u(\nu+ \varepsilon ,0,\gamma_0)-\gamma_1 \|_{ L^\infty(\Bbb{S}^2)}\leq &\| \phi^u(\nu+ \varepsilon ,\nu,\gamma_*)-\phi^u(\nu+ \varepsilon ,\nu,\gamma_0) \|_{ L^\infty(\Bbb{S}^2)}\\
&+\| \phi^u(\nu+ \varepsilon ,\nu,\gamma_0)-\gamma_1 \|_{ L^\infty(\Bbb{S}^2)}\\
 \leq &C(\|u_0\|_{H^1})\bigl(\sqrt{\e}+\sqrt{\nu}\bigr),
 \end{split}
\end{equation*}
while maintaining the condition that for any $t\in [0,\nu+ \varepsilon]$,
$\phi^u(t,0,\gamma_0)\subset \Omega$.

This completes the proof of the second part of
Theorem \ref{t-lag}.
\end{proof}

\appendix

\renewcommand{\theequation}{\thesection.\arabic{equation}}
\setcounter{equation}{0}

\section{On the regularization of the uncontrolled strong solutions to the Navier-Stokes equations with Navier boundary conditions}
\label{appa}

In this appendix we prove the following  regularization result  of the uncontrolled strong solutions to the Navier-Stokes equations with Navier boundary conditions on the whole boundary $\partial \Omega$, that is to the following system:
\beq \label{SAEQ0}
\begin{cases}
\partial_t u+u\cdot \nabla u- \Delta u +\nabla p=0, \\
  \dive  u=0 \quad \mbox{ in }\Omega,\\
u\cdot \mathbf{n}=0 \quad \mbox{ and }   \quad\CN(u)=0\quad\mbox{ on }\partial\Omega, \\
u=u_0\quad\mbox{ at } t=0 .
\end{cases}
\eeq
\begin{theorem}\label{thS}
{\sl
 Let $T>0$,  $p\geq 1$ and $u_0$   in $H^1(\Omega)$, divergence free and tangent to $\partial \Omega$.
 Then there are $T_1$ in $(0,T)$ and
a continuous function $C_{T_1,p}$ with $C_{T_1,p}(0)=0$,
such that the unique strong solution $u\in C([0,T_1];H^1(\Omega))\cap L^2([0,T_1];H^2(\Omega))$ to \eqref{SAEQ0}
 satisfies
\beq \label{SAeq-2}
\sum_{0\leq j \leq \f{p}2}\bigl\|t^{\f{p-1}2}\p_t^ju\bigr\|_{L^\infty_{T_1}(H^{p-2j}(\Om))}
+\sum_{0\leq j \leq \f{p+1}2}\bigl\|t^{\f{p-1}2}\p_t^ju\bigr\|_{L^2_{T_1}(H^{p+1-2j}(\Om))}\leq C_{p,T_1}(\|u_0\|_{H^1(\Om)}).
\eeq}
\end{theorem}

As recalled in Section \ref{serre-chevalier}
The goal of this section is to present the proof of Theorem \ref{S}, namely, the local-in-time existence and uniqueness
 of  strong solutions with $H^1$ initial data  is classical. The interest of Theorem \ref{thS} is to detail the  regularization in time of this strong solution near the time zero.
In particular it implies the  part  of Theorem \ref{S}  regarding the regularization.

\begin{proof}  
We will proceed by induction on $p$. We start with recalling how to prove the case $p=1$, by proving first  a $L^2(\Om)$ energy estimate and then a $H^1(\Om)$ energy estimate.

\smallskip

\noindent $\bullet$ \underline{$L^2(\Om)$ energy estimate}\vspace{0.2cm}

Indeed, we first get, by taking $L^2(\Om)$ inner product of the $u$ equation
in \eqref{SAEQ0} with $u,$ that
\beq \label{SAeq2}
\f12\f{d}{dt}\|u(t)\|_{L^2(\Om)}^2+\left(u\cdot\na u | u\right)_{L^2(\Om)}-\left(\D u | u\right)_{L^2(\Om)}+\left(\na p | u\right)_{L^2(\Om)}=0.
\eeq
Here and in all that follows, we always denote $(f | g)_{L^2(\Om)}:=\int_{\Om} f g\,dx.$

Due to $\dive u=0$ and $u\cdot \bn|_{\p\Om}=0,$ we have
\beno
\left(u\cdot\na u | u\right)_{L^2(\Om)}=0=\left(\na p | u\right)_{L^2(\Om)}.
\eeno
Whereas it follows from Stokes formula that
\beno
\begin{split}
-\left(\D u | u\right)_{L^2(\Om)}
=&\int_{\p\Om}\left[(\na\times u)\times u\right]\cdot \bn\,dS+\int_{\Om}|\na\times u|^2\,dx.
\end{split}
\eeno
By inserting the above equalities into \eqref{SAeq2}, we obtain
\beq \label{SAeq3}
\f12\f{d}{dt}\|u(t)\|_{L^2(\Om)}^2+\|\na\times u\|_{L^2(\Om)}^2=\int_{\p\Om}\left[u\times(\na\times u)\right]\cdot \bn\,dS.
\eeq
 Let us denote by $M_{\rm w}$ the shape operator associated with $\Omega$.
Recall that, since $\Omega$ is smooth, the shape operator $M_{\rm w}$ is smooth and for any $x\in\p\Om,$
it defines a self-adjoint operator with values in the tangent space $T_x$.  Then we have the following result, see \cite{BdV,GK}.
\begin{lemma} \label{SAlem1}
{\sl For any smooth divergence free vector field $u$
satisfying $u\cdot \bn=0$ on $\p\Om,$ we have
\beq \left[D(u)\bn+M_{\rm w}u\right]_{\mbox{tan}}=\f12(\na\times u)\times \bn.
\label{SAeq1}
\eeq}
\end{lemma}
Yet due to $\CN(u)|_{\p\Om}=0,$ we deduce from Lemma \ref{SAlem1} that
\beq \label{SAeq7}
\begin{split}
\left[u \times (\na\times u)\right]\cdot \bn \bigr|_{\p\Om}=& u\cdot \left[(\na\times u)\times \bn\right] \bigr|_{\p\Om}\\
=&2\left[(M_{\rm w}-M)u\right]_{\mbox{tan}}\cdot u\bigr|_{\p\Om}\\
=&2\left[(M_{\rm w}-M)u\right]\cdot u\bigr|_{\p\Om},
\end{split}
\eeq
where we used $u\cdot \bn \bigr|_{\p\Om}=0$ in the last step. Then by applying Stokes
formula and Young's inequality, we find    that for any $\la>0,$ there exists $C_\la$ so that
\beq \label{SAeq19}
\begin{split}
\bigl|\int_{\p\Om}\left[(\na\times u)\times u\right]\cdot \bn\,dS\bigr|=&2\bigl|\int_{\Om}\dive \left[\bigl((M_{\rm w}-M)u\cdot u\bigr)\bn\right]\,dx\bigr|\\
\leq & \la \|\na u\|_{L^2(\Om)}^2+C_\la \|u\|_{L^2(\Om)}^2,
\end{split}
\eeq
On the other hand, due to $\dive u=0$ in $\Om$ and $u\cdot \bn|_{\p\Om}=0,$ we deduce from
Korn's type inequality (see \cite{DL} for instance) that there exists a positive constant $C_\Om$ so that
\beq \label{SAeq4}
\|\na\times u\|_{L^2(\Om)}^2\geq \f1{C_\Om} \|u\|_{H^1(\Om)}^2-\|u\|_{L^2(\Om)}^2.
\eeq
By inserting the estimates, \eqref{SAeq19} and \eqref{SAeq4}, into \eqref{SAeq3} and taking $\la=\f1{2C_\Om}$ in the resulting inequality,
we achieve
\beq \label{SAeq12}
\f{d}{dt}\|u(t)\|_{L^2(\Om)}^2+\f1{C_\Om} \|u\|_{H^1(\Om)}^2\leq C \|u\|_{L^2(\Om)}^2.
\eeq
Applying Gronwall's inequality gives rise to
\beq \label{SAeq5}
\|u\|_{L^\infty_t(L^2(\Om))}^2+\f1{C_\Om} \|u\|_{L^2_t(H^1(\Om))}^2 \leq \|u_0\|_{L^2(\Om)}^2 e^{Ct}.
\eeq

\noindent $\bullet$ \underline{$H^1(\Om)$ energy estimate}\vspace{0.2cm}

By taking $L^2(\Om)$ inner product of the $u$ equation of \eqref{SAEQ0} with $\p_tu,$ we get
\beq \label{SAeq6}
\|\p_tu\|_{L^2(\Om)}^2-\left(\D u | \p_t u\right)_{L^2(\Om)}+\left( \na p | \p_t u\right)_{L^2(\Om)}=-\left( u\cdot\na u | \p_t u\right)_{L^2(\Om)}.
\eeq
Notice that $\p_tu\cdot \bn|_{\p\Om}=0,$ by applying Stokes formula and along the same line to the proof of \eqref{SAeq7}, we obtain
\beno
\begin{split}
-\left(\D u | \p_t u\right)_{L^2(\Om)}=&\int_{\p\Om}[(\na\times u)\times \p_tu]\cdot \bn\,dS+\int_{\Om}(\na\times u) \cdot (\na\times \p_tu)\,dx\\
=&2\int_{\p\Om} \p_tu(M-M_{\rm w})u\,dS+\f12\f{d}{dt}\int_{\Om}|\na\times u|^2\,dx,
\end{split}
\eeno
which together with the facts:  $M$ is a symmetric matrix and $M_{\rm w}$ is a self-adjoint operator on $T_x,$ ensures that
\beno
-\left(\D u | \p_t u\right)_{L^2(\Om)}=\f{d}{dt}\Bigl(\int_{\p\Om} u(M-M_{\rm w})u\,dS+\f12\int_{\Om}|\na\times u|^2\,dx\Bigr).
\eeno
Again due to $\p_tu\cdot \bn|_{\p\Om}=0,$ one has
\beno
\left( \na p | \p_t u\right)_{L^2(\Om)}=0.
\eeno
By inserting the above equalities into \eqref{SAeq6}, we achieve
\beno
\begin{split}
\f{d}{dt}\Bigl(\int_{\p\Om} u(M-M_{\rm w})u\,dS+&\f12\int_{\Om}|\na\times u|^2\,dx\Bigr)+\|\p_tu\|_{L^2(\Om)}^2
=-\left( u\cdot\na u | \p_t u\right)_{L^2(\Om)}\\
\leq & \|u\|_{L^6(\Om)}\|\na u\|_{L^3(\Om)}\|\p_tu\|_{L^2(\Om)}\\
\leq &C\|u\|_{H^1(\Om)}\|\na u\|_{L^2(\Om)}^{\f12}\|\na u\|_{H^1(\Om)}^{\f12}\|\p_tu\|_{L^2(\Om)}.
\end{split}
\eeno
Applying Young's inequality yields
\beq \label{SAeq8}
\begin{split}
\f{d}{dt}\Bigl(\int_{\p\Om} u(M-M_{\rm w})u\,dS+&\f12\int_{\Om}|\na\times u|^2\,dx\Bigr)+\f34\|\p_tu\|_{L^2(\Om)}^2\\
\leq&  C_\la\bigl(1+\|u\|_{H^1(\Om)}^4\bigr)\|\na u\|_{L^2(\Om)}^2+\la \|\na^2 u\|_{L^2(\Om)}^2.
\end{split}
\eeq
Whereas in view of \eqref{SAEQ0}, we write
\beq \label{SAeq9}
\begin{cases}
- \Delta u +\nabla p=-\partial_t u-u\cdot \nabla u \\
  \dive  u=0 \quad \mbox{ in }\Omega,\\
u\cdot \mathbf{n}=0 \quad \mbox{ and }   \quad\CN(u)=0\quad\mbox{ on }\partial\Omega.
\end{cases}
\eeq

The following type of Cattabriga-Solonnikov estimate can be proved along the same line to that
of Theorem 2.2 in \cite{Seregin}.
\begin{lemma} \label{SAlem2}
{\sl Let $k$ be a non-negative integer and $\Om$ be a bounded domain with sufficiently smooth boundary.
Let $f\in H^k(\Om)$ and $ g\in H^{k+1}(\Om)$ with $\int_\Om g\,dx=0.$ Then the non-homogeneous Stokes problem:
\beno
\begin{cases}
- \Delta u +\nabla p=f  \\
  \dive  u=g \quad \mbox{ in }\Omega,\\
u\cdot \mathbf{n}=0 \quad \mbox{ and }   \quad\CN(u)=0\quad\mbox{ on }\partial\Omega,
\end{cases}
\eeno
has a unique solution $(u,p)$ so that
\beq \label{SAeq-1}
\|\na^2u\|_{H^k(\Om)}+\|\na p\|_{H^k(\Om)}\leq C\bigl(\|f\|_{H^k(\Om)}+\|\na g\|_{H^k(\Om)}\bigr).
\eeq}
\end{lemma}
Then it follows from Lemma \ref{SAlem2} and \eqref{SAeq9} that
\beno
\begin{split}
\|\na^2 u\|_{L^2(\Om)}\leq & C\bigl(\|\p_tu\|_{L^2(\Om)}+\|u\cdot\na u\|_{L^2(\Om)}\bigr)\\
\leq &C\bigl(\|\p_tu\|_{L^2(\Om)}+\|u\|_{H^1(\Om)}\|\na u\|_{L^2(\Om)}^{\f12}\|\na u\|_{H^1(\Om)}^{\f12}\bigr),
\end{split}
\eeno
from which, we infer
\beq \label{SAeq10}
\|\na u\|_{H^1(\Om)}\leq C\bigl(\|\p_tu\|_{L^2(\Om)}+(1+\|u\|_{H^1(\Om)}^2)\|\na u\|_{L^2(\Om)}\bigr).
\eeq
By substituting \eqref{SAeq10} into \eqref{SAeq8} and then taking $\la=\f1{4C},$ we achieve
\beq \label{SAeq11}
\begin{split}
\f{d}{dt}\Bigl(\int_{\p\Om} u(M-M_{\rm w})u\,dS+\f12\int_{\Om}|\na\times u|^2\,dx\Bigr)+&\f12\|\p_tu\|_{L^2(\Om)}^2\\
\leq & C\bigl(1+\|u\|_{H^1(\Om)}^4\bigr)\|\na u\|_{L^2(\Om)}^2.
\end{split}
\eeq
While it follows from trace inequality \eqref{trace} that
\beno
\begin{split}
\bigl|\int_{\p\Om} u(M-M_{\rm w})u\,dS\bigr|\leq  C\|u\|_{L^2(\p\Om)}^2\leq & C\bigl(\|u\|_{L^2(\Om)}^2+\|u\|_{L^2(\Om)}\|\na u\|_{L^2(\Om)}\bigr)\\
\leq & \f1{4C_\Om}\|u\|_{H^1(\Om)}^2+C\|u\|_{L^2(\Om)}^{2},
\end{split}
\eeno
so that in view of \eqref{SAeq4}, there exists a large enough constant $K$ which satisfies
\beq \label{SAeq13}
E_1(u):= K\|u\|_{L^2(\Om)}^{2}+\int_{\p\Om} u(M-M_{\rm w})u\,dS+\f12\int_{\Om}|\na\times u|^2\,dx\geq \f1{4C_\Om}\|u\|_{H^1(\Om)}^2.
\eeq
Then we get, by summing up $K\times$\eqref{SAeq12} and \eqref{SAeq11}, that
\beq \label{SAeq14}
\f{d}{dt}E_1(u)+\f12\|\p_tu\|_{L^2(\Om)}^2\leq CE_1(u)\bigl(1+E_1^2(u)\bigr),
\eeq
from which, we deduce by a comparison argument that there exists a positive time $T_1$ and
a continuous function $C_{T_1,p}$ with $C_{T_1,p}(0)=0$,
such that
 \eqref{SAeq-2}  holds true for $p=1.$

\smallskip

\noindent $\bullet$ \underline{Higher energy estimates}\vspace{0.2cm}

Inductively, we assume that \eqref{SAeq-2} holds for $p\leq \ell-1,$ we are going to show that
\eqref{SAeq-2} holds for $p=\ell.$ Without loss of generality, we may assume that $\ell$ is an even integer. The odd integer
case can be proved along the same line.
Indeed we first get, by applying $\p_t^{\ell /2}$ to \eqref{SAEQ0}, that
\beq \label{SAeq17}
\begin{cases}
\partial_t^{1+\f\ell2} u+\p_t^{\f\ell2}(u\cdot \nabla u)- \Delta \p_t^{\f\ell2}u +\nabla \p_t^{\f\ell2}p=0, \\
  \dive  \p_t^{\f\ell2}u=0 \quad \mbox{ in } (0,T_1)\times \Omega,\\
\p_t^{\f\ell2}u\cdot \mathbf{n}=0 \quad \mbox{ and }   \quad\CN(\p_t^{\f\ell2}u)=0\quad\mbox{ on }\ (0,T_1)\times\partial\Omega,
\end{cases}
\eeq
from which, we get, by a similar derivation of \eqref{SAeq3} that
\beq \label{SAeq18}
\begin{split}
\f12\f{d}{dt}\bigl(&t^{\ell-1}\|\p_t^{\f\ell2}u(t)\|_{L^2(\Om)}^2\bigr)+t^{\ell-1}\|\na\times \p_t^{\f\ell2}u\|_{L^2(\Om)}^2
=\f{\ell-1}2t^{\ell-2}\|\p_t^{\f\ell2}u\|_{L^2(\Om)}^2\\
&\qquad+t^{\ell-1}\int_{\p\Om}\bigl[\p_t^{\f\ell2}u\times(\na\times \p_t^{\f\ell2}u)\bigr]\cdot \bn\,dS
-t^{\ell-1}\bigl(\p_t^{\f\ell2}(u\cdot\na u) | \p_t^{\f\ell2}u\bigr)_{L^2(\Om)}.
\end{split}
\eeq
Similar to \eqref{SAeq19}, we have
\beno
t^{\ell-1}\bigl|\int_{\p\Om}\bigl[\p_t^{\f\ell2}u\times(\na\times \p_t^{\f\ell2}u)\bigr]\cdot \bn\,dS\bigr|
\leq  \la \bigl\|t^{\f{\ell-1}2}\na \p^{\f\ell2}_tu\bigr\|_{L^2(\Om)}^2+C_\la \bigl\|t^{\f{\ell-1}2}\p_t^{\f\ell2}u\bigr\|_{L^2(\Om)}^2.
\eeno
While due to $u\cdot \bn|_{\p\Om}=0$ and $\dive u=0,$ we get, by using integration by parts, that
\beno
\begin{split}
\bigl(\p_t^{\f\ell2}(u\cdot\na u) | \p_t^{\f\ell2}u\bigr)_{L^2(\Om)}=&\bigl(\p_t^{\f\ell2}( u\cdot\na u)-u\cdot\na \p_t^{\f\ell2}u | \p^{\f\ell2}_tu \bigr)_{L^2(\Om)}\\
=&-\sum_{\substack{\ell_1+\ell_2=\f{\ell}2\\ \ell_1\geq 1}}C_{\f{\ell}2}^{\ell_1}\bigl(\p_t^{\ell_1}u\otimes \p_t^{\ell_2}u | \na \p^{\f\ell2}_tu \bigr)_{L^2(\Om)},
\end{split}
\eeno
from which, we infer
\beno
\begin{split}
t^{\ell-1}\bigl|\bigl(\p_t^{\f\ell2}(u\cdot\na u) | \p_t^{\f\ell2}u\bigr)_{L^2(\Om)}\bigr|
\lesssim &\sum_{\substack{\ell_1+\ell_2=\f{\ell}2\\ \ell_1\geq 1}}
t^{\ell-1}\|\p_t^{\ell_1}u\|_{L^3(\Om)}\|\p_t^{\ell_2}u\|_{L^6(\Om)}\|\na \p^{\f\ell2}_tu\|_{L^2(\Om)}\\
\lesssim &\sum_{\substack{\ell_1+\ell_2=\f{\ell}2\\ \ell_1\geq 1}}
t^{\ell-1}\|\p_t^{\ell_1}u\|_{L^2(\Om)}^{\f12}\|\p_t^{\ell_1}u\|_{H^1(\Om)}^{\f12}\|\p_t^{\ell_2}u\|_{H^1(\Om)}\|\na \p^{\f\ell2}_tu\|_{L^2(\Om)}\\
\leq &\la  \bigl\|t^{\f{\ell-1}2} \p_t^{\f\ell2}u\bigr\|_{H^1(\Om)}^2+C_\la \|u\|_{H^1(\Om)}^{4}\bigl\|t^{\f{\ell-1}2} \p^{\f\ell2}_tu\bigr\|_{L^2(\Om)}^2\\
&+C_\la\sum_{\substack{\ell_1+\ell_2=\f{\ell}2\\ 1\leq \ell_1\leq \f{\ell}2-1}}
\bigl\|t^{\ell_1-\f12}\p_t^{\ell_1}u\bigr\|_{H^1(\Om)}^2\bigl\||t^{\ell_2}\p_t^{\ell_2}u\bigr\|_{H^1(\Om)}^2.
\end{split}
\eeno
By substituting the above estimates into \eqref{SAeq18} and using Korn's type inequality \eqref{SAeq4}, we find
\beno
\begin{split}
\f12\f{d}{dt}&\bigl\|t^{\f{\ell-1}2}\p_t^{\f\ell2}u(t)\bigr\|_{L^2(\Om)}^2+\f{1}{C_\Om}\bigl\|t^{\f{\ell-1}2} \p_t^{\f\ell2}u\bigr\|_{H^1(\Om)}^2
\\
\leq &\f{\ell-1}2\bigl\|t^{\f\ell2-1}\p_t^{\f\ell2}u\bigr\|_{L^2(\Om)}^2+C_\la\bigl(1+ \|u\|_{H^1(\Om)}^{4}\bigr)\bigl\|t^{\f{\ell-1}2} \p^{\f\ell2}_tu\bigr\|_{L^2(\Om)}^2\\
&+2\la  \bigl\|t^{\f{\ell-1}2} \p_t^{\f\ell2}u\bigr\|_{H^1(\Om)}^2+ C_\la\sum_{\substack{\ell_1+\ell_2=\f{\ell}2\\ 1\leq \ell_1\leq \f{\ell}2-1}}
\bigl\|t^{\ell_1-\f12}\p_t^{\ell_1}u\bigr\|_{H^1(\Om)}^2\bigl\||t^{\ell_2}\p_t^{\ell_2}u\bigr\|_{H^1(\Om)}^2.
\end{split}
\eeno
By taking $\la=\f1{4C_\Om}$ in the above inequality and then
applying Gronwall's inequality to the resulting inequality, we achieve
\beno
\begin{split}
\bigl\|t^{\f{\ell-1}2}&\p_t^{\f\ell2}u\bigr\|_{L^\infty_t(L^2(\Om))}^2+\f{1}{C_\Om}\bigl\|t^{\f{\ell-1}2} \p_t^{\f\ell2}u\bigr\|_{L^2_t(H^1(\Om))}^2\leq C\exp\left(C\bigl(1+t\|u\|_{L^\infty_{t}(H^1(\Om))}^4\bigr)\right)\\
&\qquad\times  \Bigl(\bigl\|t^{\f\ell2-1}\p_t^{\f\ell2}u\bigr\|_{L^2_t(L^2(\Om))}^2+
\sum_{\substack{\ell_1+\ell_2=\f{\ell}2\\ 1\leq \ell_1\leq \f{\ell}2-1}}
\bigl\|t^{\ell_1-\f12}\p_t^{\ell_1}u\bigr\|_{L^2_t(H^1(\Om))}^2\bigl\||t^{\ell_2}\p_t^{\ell_2}u\bigr\|_{L^\infty_t(H^1(\Om))}^2\Bigr),
\end{split}
\eeno
from which and the inductive assumption, we deduce that
\beq \label{SAeq20}
\begin{split}
\bigl\|t^{\f{\ell-1}2}\p_t^{\f\ell2}u\bigr\|_{L^\infty_{T_1}(L^2(\Om))}^2+\f{1}{C_\Om}\bigl\|t^{\f{\ell-1}2} \p_t^{\f\ell2}u\bigr\|_{L^2_{T_1}(H^1(\Om))}^2\leq  C_{\ell, T_1}(\|u_0\|_{H^1(\Om)}).
\end{split}
\eeq

On the other hand, for any non-negative integer $j\leq \f\ell2-1,$ we infer from the inductive assumption that
\beno
\begin{split}
\bigl\|t^{\f{\ell-1}2}\p_t^{j}u\bigr\|_{L^\infty_{T_1}(H^{\ell-2j}(\Om))}
=&\bigl\|t^{\f{\ell-1}2}\na^2\p_t^{j}u\bigr\|_{L^\infty_{T_1}(H^{\ell-2-2j}(\Om))}
+\bigl\|t^{\f{\ell-1}2}\p_t^{j}u\bigr\|_{L^\infty_{T_1}(H^{\ell-1-2j}(\Om))}\\
\leq &\bigl\|t^{\f{\ell-1}2}\na^2\p_t^{j}u\bigr\|_{L^\infty_{T_1}(H^{\ell-2-2j}(\Om))}+ C_{\ell, T_1}(\|u_0\|_{H^1(\Om)}).
\end{split}
\eeno
Whereas in view of \eqref{SAEQ0}, we write
\beno
- \Delta \p_t^{j}u +\nabla \p_t^{j}p=-\partial_t^{j+1} u-\p_t^{j}(u\cdot \nabla u),
\eeno
from which, and Lemma \ref{SAlem2}, we infer
\beno
\begin{split}
\bigl\|t^{\f{\ell-1}2}\na^2\p_t^{j}u\bigr\|_{L^\infty_{T_1}(H^{\ell-2-2j}(\Om))}\lesssim  & \bigl\|t^{\f{\ell-1}2}\p_t^{j+1}u\bigr\|_{L^\infty_{T_1}(H^{\ell-2-2j}(\Om))}\\
&+
\bigl\|t^{\f{\ell-1}2}\p_t^{j}(u\cdot\na u)\bigr\|_{L^\infty_{T_1}(H^{\ell-2-2j}(\Om))}.
\end{split}
\eeno
As a result, it comes out
\beq \label{SAeq26}
\begin{split}
\bigl\|t^{\f{\ell-1}2}\p_t^{j}u\bigr\|_{L^\infty_{T_1}(H^{\ell-2j}(\Om))}
\leq & C_{\ell, T_1}(\|u_0\|_{H^1(\Om)})+\bigl\|t^{\f{\ell-1}2}\p_t^{j+1}u\bigr\|_{L^\infty_{T_1}(H^{\ell-2-2j}(\Om))}\\
&+
\bigl\|t^{\f{\ell-1}2}\p_t^{j}(u\cdot\na u)\bigr\|_{L^\infty_{T_1}(H^{\ell-2-2j}(\Om))}, \quad\forall\ j\leq\f\ell2-1.
\end{split}
\eeq
Yet it follows from Moser type inequality and the inductive assumption that
\beno
\begin{split}
\bigl\|t^{\f{\ell-1}2}\p_t^{j}\na(u\otimes u)\bigr\|_{L^\infty_{T_1}(H^{\ell-2-2j}(\Om))}
\lesssim &\sum_{j_1+j_2=j}\bigl\|t^{j_1+\f{1}2}\p_t^{j_1}u\bigr\|_{L^\infty_{T_1}(H^{2}(\Om))}\\
&\times \bigl\|t^{\f{\ell-2}2-j+j_2}\p_t^{j_2}u\bigr\|_{L^\infty_{T_1}(H^{\ell-2j-1}(\Om))}\leq  C_{\ell, T_1}(\|u_0\|_{H^1(\Om)}).
\end{split}
\eeno
Substituting the above estimates into \eqref{SAeq26} gives rise to
\beno
\begin{split}
\bigl\|t^{\f{\ell-1}2}\p_t^{j}u\bigr\|_{L^\infty_{T_1}(H^{\ell-2j}(\Om))}
\leq & C_{\ell, T_1}(\|u_0\|_{H^1(\Om)})+\bigl\|t^{\f{\ell-1}2}\p_t^{j+1}u\bigr\|_{L^\infty_{T_1}(H^{\ell-2-2j}(\Om))},
\end{split}
\eeno
from which, \eqref{SAeq20}, we deduce by an iterative argument that
\beq \label{SAeq27}
\sum_{0\leq j \leq \f{\ell}2}\bigl\|t^{\f{\ell-1}2}\p_t^ju\bigr\|_{L^\infty_{T_1}(H^{\ell-2j}(\Om))}
\leq C_{\ell,T_1}(\|u_0\|_{H^1(\Om)}).
\eeq

Exactly along the same line to the proof of \eqref{SAeq27},   for any non-negative integer $j\leq \f{\ell}2-1,$ we infer from the inductive assumption that
\beno
\begin{split}
\bigl\|t^{\f{\ell-1}2}\p_t^{j}u\bigr\|_{L^2_{T_1}(H^{\ell+1-2j}(\Om))}
\leq &\bigl\|t^{\f{\ell-1}2}\na^2\p_t^{j}u\bigr\|_{L^2_{T_1}(H^{\ell-1-2j}(\Om))}+ C_{\ell, T_1}(\|u_0\|_{H^1(\Om)}).
\end{split}
\eeno
While it follows from Lemma \ref{SAlem2} that
\beno
\begin{split}
\bigl\|t^{\f{\ell-1}2}\na^2\p_t^{j}u\bigr\|_{L^2_{T_1}(H^{\ell-1-2j}(\Om))}\lesssim  & \bigl\|t^{\f{\ell-1}2}\p_t^{j+1}u\bigr\|_{L^2_{T_1}(H^{\ell-1-2j}(\Om))}\\
&+
\bigl\|t^{\f{\ell-1}2}\p_t^{j}(u\otimes u)\bigr\|_{L^\infty_{T_1}(H^{\ell-2j}(\Om))}.
\end{split}
\eeno
Yet for any $j\leq \f\ell2-1,$ it follows from Moser type inequality and the inductive assumption that
\beno
\begin{split}
\bigl\|t^{\f{\ell-1}2}\p_t^{j}(u\otimes u)\bigr\|_{L^2_{T_1}(H^{\ell-2j}(\Om))}
\lesssim &\sum_{j_1+j_2=j}\bigl\|t^{j_1+\f{1}2}\p_t^{j_1}u\bigr\|_{L^\infty_{T_1}(H^{2}(\Om))}\\
&\times \bigl\|t^{\f{\ell-2}2-j+j_2}\p_t^{j_2}u\bigr\|_{L^2_{T_1}(H^{\ell-2j}(\Om))}\leq  C_{\ell, T_1}(\|u_0\|_{H^1(\Om)}).
\end{split}
\eeno
As a result, for any $j\leq \f{\ell-1}2,$ we arrive at
\beno
\begin{split}
\bigl\|t^{\f{\ell-1}2}\p_t^{j}u\bigr\|_{L^2_{T_1}(H^{\ell+1-2j}(\Om))}
\leq & C_{\ell, T_1}(\|u_0\|_{H^1(\Om)})+\bigl\|t^{\f{\ell-1}2}\p_t^{j+1}u\bigr\|_{L^2_{T_1}(H^{\ell-1-2j}(\Om))},
\end{split}
\eeno
from which, \eqref{SAeq20}, we deduce by an iterative argument that
\beq \label{SAeq28}
\sum_{0\leq j \leq \f{\ell}2}\bigl\|t^{\f{\ell-1}2}\p_t^ju\bigr\|_{L^2_{T_1}(H^{\ell+1-2j}(\Om))}
\leq C_{\ell,T_1}(\|u_0\|_{H^1(\Om)}).
\eeq

\eqref{SAeq27} along with \eqref{SAeq28} shows that \eqref{SAeq-2} holds for $p=\ell.$ This finishes
the proof of \eqref{SAeq-2} and therefore the proof of Theorem \ref{S}.
\end{proof}

\bigskip
\ \par
\noindent
{\bf Acknowledgements.}
F. Sueur is partially supported  by the Agence Nationale de la Recherche, Project IFSMACS, grant ANR-15-CE40-0010, Project SINGFLOWS, grant ANR-18-CE40-0027-01, and
Project BORDS, grant ANR-16-CE40-0027-01; and by the H2020-MSCA-ITN-2017 program, Project ConFlex, Grant ETN-765579.
P. Zhang is partially supported by NSF of China under Grants 11688101 and 11371347, Morningside Center of Mathematics of The Chinese Academy of Sciences and innovation grant from National Center for Mathematics and Interdisciplinary Sciences.
F. Sueur warmly thanks  Morningside center  of Mathematics, CAS, for its kind hospitality during his stays in May 2018 and October 2019.

\end{document}